\documentclass[11pt]{amsart}

\usepackage[mathscr]{euscript}
\usepackage{amssymb,bbm}
\usepackage{graphicx}
\usepackage{pb-diagram,pb-xy} 
\usepackage[all,cmtip,arrow]{xy} 
\usepackage[margin=1in]{geometry}  
\usepackage{hyperref} 

\numberwithin{equation}{section}

\theoremstyle{plain}
\newtheorem{theorem}[equation]{Theorem}

\newtheorem{lemma}[equation]{Lemma}

\newtheorem{corollary}[equation]{Corollary}
\newtheorem{proposition}[equation]{Proposition}

\newtheorem{prop}[equation]{Proposition}

\newtheorem*{claim*}{Claim}

\theoremstyle{definition}
\newtheorem{definition}[equation]{Definition}

\newtheorem{notation}[equation]{Notation}

\newtheorem{remark}[equation]{Remark}

\newtheorem{example}[equation]{Example}
\newtheorem{examples}[equation]{Examples}

\newcommand{\isom}{\cong}                       
\newcommand{\homeq}{\simeq}                     
\newcommand{\smsh}{\wedge}                      
\newcommand{\binBox}{\mathbin{\Box}}
\newcommand{\homsmsh}{\mathbin{\tilde{\smsh}}}

\newcommand{\Smsh}{\bigwedge}                   
\newcommand{\Wdge}{\bigvee}                     

\newcommand{\N}{\mathsf{N}}                     
\newcommand{\M}{\mathsf{M}}
\newcommand{\A}{\mathsf{A}}
\newcommand{\B}{\mathsf{B}}
\newcommand{\one}{\mathsf{1}}
\newcommand{\F}{\mathbf{F}}
\renewcommand{\P}{\mathsf{P}}

\renewcommand{\S}{\mathsf{S}}
\newcommand{\E}{\mathsf{E}}
\newcommand{\K}{\mathsf{K}}
\renewcommand{\d}{\mathsf{d}}
\newcommand{\Com}{\mathsf{Com}}

\newcommand{\cat}[1]{\mathscr{#1}}              

\newcommand{\C}{\mathbf{C}}
\newcommand{\U}{\mathbf{U}}
\newcommand{\R}{\mathbf{R}}
\renewcommand{\L}{\mathbf{L}}
\renewcommand{\c}{\mathbf{c}}

\newcommand{\colim}{\operatorname{colim}}
\newcommand{\hocolim}{\operatorname{hocolim}}
\newcommand{\holim}{\operatorname{holim}}
\newcommand{\Map}{\operatorname{Map} }
\newcommand{\Nat}{\operatorname{Nat} }

\newcommand{\creff}{\operatorname{cr} }
\newcommand{\Hom}{\operatorname{Hom} }

\newcommand{\Tot}{\operatorname{Tot} }

\newcommand{\Mod}{\operatorname{Mod} }
\newcommand{\Comod}{\operatorname{Comod} }
\newcommand{\Coalg}{\operatorname{Coalg} }

\DeclareMathOperator*{\hofib}{hofib}

\DeclareMathOperator*{\thofib}{thofib}
\DeclareMathOperator*{\thocofib}{thocofib}

\newcommand{\un}[1]{\underline{#1}}

\newcommand{\spaces}{\mathscr{T}op}
\newcommand{\based}{\spaces_*}
\newcommand{\spectra}{{\mathscr{S}p}}              
\newcommand{\finspec}{\mathscr{S}p^{\mathsf{f}}}              
\newcommand{\finbased}{{\mathscr{T}op^{\mathsf{f}}_*}}
\newcommand{\symseq}{[\Sigma,\mathscr{S}p]}

\newcommand{\weq}{\; \tilde{\longrightarrow} \;}      
\newcommand{\lweq}{\; \tilde{\longleftarrow} \;}      
\newcommand{\epi}{\twoheadrightarrow}           
\newcommand{\into}{\hookrightarrow}

\newcommand{\der}{\partial}                     

\newcommand{\ord}[1]{#1\textsuperscript{th}}

\begin{document}

\title{Cross-effects and the classification of Taylor towers}
\author{Greg Arone and Michael Ching}

\begin{abstract}
  Let $F$ be a homotopy functor with values in the category of spectra. We show that partially stabilized cross-effects of $F$ have an action of a certain operad. For functors from based spaces to spectra, it is the Koszul dual of the little discs operad. For functors from spectra to spectra it is a desuspension of the commutative operad. It follows that the Goodwillie derivatives of $F$ are a right module over a certain `pro-operad'. For functors from spaces to spectra, the pro-operad is a resolution of the topological Lie operad. For functors from spectra to spectra, it is a resolution of the trivial operad. We show that the Taylor tower of the functor $F$ can be reconstructed from this structure on the derivatives.
\end{abstract}

\maketitle

Let $\cat{C}$ and $\cat{D}$ each be either the category of based topological spaces, or the category of spectra. Let $F: \cat{C} \to \cat{D}$ be a homotopy functor. Goodwillie's homotopy calculus provides a systematic way to decompose $F$ into homogeneous pieces which are classified by certain spectra with $\Sigma_n$ actions, denoted $\der_nF$ and, by analogy with ordinary calculus, called the \emph{derivatives} or \emph{Taylor coefficients} of $F$.

A key problem in the homotopy calculus is to describe all the relevant structure on the symmetric sequence $\der_*F$, and to reconstruct the original functor $F$ (or at least its Taylor tower) from this structure. In~\cite{arone/ching:2014} we gave a general description of this structure. In this paper we give an alternative description of the structure on $\der_*F$ in cases when $F$ takes values in the category of spectra. We interpret this structure as that of a `module over a pro-operad'. The structure arises from actions of certain operads on the partially stablized cross-effects of $F$. Thus we connect the module structure on $\der_*F$ with Goodwillie's presentation of $\der_*F$ as stablized cross-effects of $F$.

Let us review briefly our previous results on the subject. Let $I_{\based}$ be the identity functor on the category of based spaces. It was shown by the second author in~\cite{ching:2005} that $\der_*I_{\based}$ has an operad structure. In fact, $\der_*I_{\based}$ is a topological realization of the classical Lie operad, and we refer to $\der_*I_{\based}$ informally as the topological Lie operad. In~\cite{arone/ching:2011} we showed that for a functor $F: \finbased \to \spectra$, the symmetric sequence $\der_*F$ is a right module over $\der_*I_{\based}$. We also remarked in~\cite{arone/ching:2011} that the module structure was not sufficient for recovering the Taylor tower of $F$, and therefore there had to be additional structure, that still had to be described.
In \cite{arone/ching:2014} we gave a description of this extra information. Specifically, we constructed a comonad $\C$, defined on the category of right $\der_*I_{\based}$-modules, which acts on the derivatives of a homotopy functor $F$. We showed that the Taylor tower of $F$ can be recovered from this action.

The methods of \cite{arone/ching:2014} were quite general (applying to functors between any combinations of the categories of based spaces and spectra). The comonad $\C$ has  the form $\der_*\Phi$, where $\der_*$ is Goodwillie differentiation, and $\Phi$ is a functor right adjoint to $\der_*$. In this paper we provide a more concrete description of this comonad, for functors that take values in $\spectra$.

To understand our results, recall that the Goodwillie derivatives of $F$ can be recovered from the cross-effects of $F$ via `multilinearization'. Explicitly, for a functor $F: \finbased \to \spectra$ we have
\[ \der_nF \homeq \hocolim_{L \to \infty} \Sigma^{-nL} \creff_nF(S^L,\dots,S^L) \]
where $\creff_nF$ denotes the \ord{$n$} cross-effect of $F$ and $S^L$ is the topological $L$-sphere. We show in Proposition~\ref{prop:dF} that, for fixed $L$, the symmetric sequence of partially-linearized cross-effects
\[ \d^L[F] := \Sigma^{-*L} \creff_*F(S^L,\dots,S^L) \]
naturally forms a right module over an operad $\K\E_L$ that is given by the Koszul dual of the stable little $L$-discs operad $\E_L$. (Strictly speaking, this is true when $F$ is polynomial, or when $F$ is analytic and $L$ is sufficiently large.)

Our construction of the right modules $\d^L[F]$ is quite complicated: see the sequence of steps listed just before Definition~\ref{def:SigmainftyX-Com} for a summary. A key role here is played by a version of Koszul duality that relates comodules over the operad $\E_L$ to modules over the Koszul dual operad $\K\E_L$. This is described further in Propositions~\ref{prop:N-indec} and \ref{prop:indec-triv}.

The Koszul dual operads $\K\E_L$ form an inverse sequence
\begin{equation} \label{eq:intro-KE} \dots \to \K\E_{L+1} \to \K\E_L \to \dots \to \K\E_2 \to \K\E_1 \to \K\E_0. \end{equation}
We refer to such an inverse sequence as a \emph{pro-operad}. We denote the above pro-operad by $\K\E_\bullet$. Note that the homotopy inverse limit of the sequence $\K\E_\bullet$ is $\K\E_\infty$, which is equivalent to the topological Lie operad $\der_*I_{\based}$. Therefore we refer to the sequence $\K\E_\bullet$ as the \emph{Lie pro-operad}.

For polynomial functors $F$, we construct a model for the maps
\[ \d^L[F] \to \d^{L+1}[F] \]
between the partially linearized cross-effects that respects the operad maps
\[ \K\E_{L+1} \to \K\E_L. \]
It follows that the symmetric sequence of derivatives $\der_*F \homeq \d[F] := \hocolim_L \d^L[F]$ inherits a limiting structure. We refer to this structure as a \emph{right module over the Lie pro-operad}. This structure includes, but contains strictly more information than, a right module over the topological Lie operad by itself.

More precisely, we define modules over a pro-operad as coalgebras over a certain comonad, constructed as follows. First note that the category of right modules over any operad $\P$ can be identified with the category of coalgebras over a certain comonad $\C_{\P}$ on the category of symmetric sequences. The sequence (\ref{eq:intro-KE}) of operads induces a directed sequence of comonads
\[ \C_{\K\E_0} \to \C_{\K\E_1} \to \C_{\K\E_2} \to \dots. \]
If we write
\[ \C_{\K\E_\bullet} := \hocolim_L \C_{\K\E_\bullet} \]
then $\C_{\K\E_\bullet}$ has a canonical comonad structure, and we define a module over the pro-operad $\K\E_\bullet$ to be a coalgebra over $\C_{\K\E_\bullet}$.

\begin{remark}
Our notion of a right module over a pro-operad is modeled on that of a discrete set with a continuous action of a profinite group. In particular, a module over a pro-operad is (for us, in this paper) not an object in the pro-category of symmetric sequences. Rather, it often can be presented as a filtered colimit of symmetric sequences. This is analogous to the fact that a discrete set $X$ with a continuous action of a profinite group $G$ is the union (i.e., colimit) of sets with actions of finite quotients of $G$.
\end{remark}

Ultimately, we show (Proposition~\ref{prop:dF}) that $\d[F]$ has a natural $\K\E_\bullet$-module structure for all pointed simplicial functors $F$ from $\based$ to $\spectra$. Our Theorem~\ref{thm:topsp} then says that $\C_{\K\E_\bullet}$ is equivalent to the abstract comonad $\C$ considered in \cite{arone/ching:2014}. It follows that the Taylor tower of a functor $F$ can be recovered from the $\K\E_\bullet$-module structure on $\d[F]$. We show in Theorem~\ref{thm:equiv-topsp-poly} that there is an equivalence between the homotopy categories of polynomial functors and bounded $\K\E_\bullet$-modules. We also show (Theorem~\ref{thm:equiv-topsp-ana}) that there is an equivalence between the homotopy categories of analytic functors (up to Taylor tower equivalence) and $\K\E_\bullet$-modules satisfying suitable connectivity conditions.

There are other approaches to the classification of polynomial functors from based spaces to spectra, most notably by Dwyer and Rezk (unpublished) in terms of functors on the category of finite sets and surjections. Polynomial functors from based spaces to spectra are also equivalent to their left Kan extensions from the full subcategory of finite pointed sets. In Theorem~\ref{thm:classification} we describe how our approach is related to these others.

The story for a functor $F: \spectra \to \spectra$ is quite analogous. The derivatives $\der_*F$ are again the homotopy colimit, over $L$, of partially-linearized cross-effects
\[ \d^L[F] := \Sigma^{-*L}\creff_*F(\Sigma^\infty S^L, \dots, \Sigma^\infty S^L) \]
and again these cross-effects, for fixed $L$, form a right module over a certain operad $\S^{-L}$. The operad $\S^{-L}$ can be thought of as the $L$-fold desuspension of the commutative operad, and all of its terms are sphere spectra of varying dimensions. Again there is an inverse sequence
\begin{equation} \label{eq:intro-S} \dots \to \S^{-(L+1)} \to \S^{-L} \to \dots \to \S^{-2} \to \S^{-1} \to \S^0 \homeq \Com. \end{equation}
We refer to this sequence as the \emph{sphere pro-operad}. Note that the homotopy inverse limit of this sequence is equivalent to the trivial operad.

We show that the derivatives $\der_*(F) \homeq \d[F] := \hocolim_L \d^L[F]$ form a coalgebra over the corresponding comonad
\[ \C_{\S_\bullet} := \hocolim_L \C_{\S^{-L}}. \]
We show in Theorem~\ref{thm:spsp} that the comonad $\C_{\S_\bullet}$ is equivalent to that constructed in \cite{arone/ching:2014}. The Taylor tower of $F: \spectra \to \spectra$ can then be recovered from the $\C_{\S_\bullet}$-coalgebra $\d[F]$ and there is an equivalence between the homotopy categories of polynomial functors and bounded coalgebras, as well as between analytic functors and suitably connected coalgebras.

There is some connection and overlap between our results on functors from $\spectra$ to $\spectra$ and work of Randy McCarthy~\cite{mccarthy:????} which involves a right $\Com$-module structure on the symmetric sequence $\creff_*F(\Sigma^\infty S^L, \ldots, \Sigma^\infty S^L)$. The $\S^{-L}$-module structure that we use is equivalent to a desuspension of McCarthy's structure. Ultimately, McCarthy classifies $n$-excisive endofunctors of $\spectra$ in terms of structure on the spectrum $F(\bigvee_n S^0)$. Namely, he shows that if $F$ is $n$-excisive, then the spectrum $F(\bigvee_n S^0)$ is a module over a certain ring spectrum, and that the functor $F$ can be recovered from this module structure.

Since $F(\bigvee_n S^0)$ is equivalent to a wedge sum of copies of cross-effects $\creff_1 F, \ldots, \creff_n F$ evaluated at $S^0$, it seems likely that McCarthy's result can be rephrased in terms of structure on the sequence of cross-effects of $F$. In this form, it would be analogous to the result of Dwyer and Rezk that classifies polynomial functors from $\based$ to $\spectra$ in terms of structure on the cross-effects.

As far as we know, until now all work on classifying polynomial functors approached the problem via describing the structure on the sequence of cross-effects of a functor. Our approach focuses instead on the derivatives of a functor. For functors from $\based$ to $\spectra$ we have a good understanding of the relationship between the structure on the cross-effects and the structure on derivatives: it is given by a form of Koszul duality between comodules over the commutative operad, and divided power modules over the Lie operad. It would be interesting to find a similar connection in the $\spectra$ to $\spectra$ case and for other classes of functors.

\subsection*{Some open problems, possible directions for future research}
\subsubsection*{Do something similar for space-valued functors.} Consider, for example, functors $F\colon \based\to \based$. By results of~\cite{arone/ching:2011}, $\der_*F$ is a bimodule over $\der_*I_{\based}$, and we seek to describe the additional structure on $\der_*F$. One answer is given in~\cite{arone/ching:2014}, which says that $\der_*I_{\based}$ is a coalgebra over a certain comonad in the category of $\der_*I_{\based}$-bimodules. But it seems desirable to have a more concrete description of the additional structure, perhaps in the form of a compatibility condition between the left and right module structures on $\der_*F$. At the moment we are unable to provide such a description. We also do not know how to relate the module structure on $\der_*F$ with the view of $\der_*F$ as the stabilized cross-effects of $F$.

\subsubsection*{Give an explicit description of $\K\E_L$, and use it to understand the connection between the two classes of functors.} Let $F\colon\based\to \spectra$ be a functor. Then $F\Sigma^\infty$ is a functor from $\spectra$ to $\spectra$. There is an equivalence of symmetric sequences $\der_*F\simeq \der_*F \Sigma^\infty$. This means that there should be a forgetful functor from modules over the sphere pro-operad to modules over the Lie pro-operad. Presumably, the forgetful functor is induced by maps of operads $\K\E_L\to \S^{-L}$. It does not seem obvious how to construct such a map of operads. This is partly because we do not have a geometric model for $\K\E_L$.

This question ties in with another conjecture about the operad $\K\E_L$. Namely, it is conjectured that there is an equivalence of operads $\K\E_L\simeq \Sigma^{-L} \E_L$. Corresponding results on the level of homology are due to Getzler and Jones \cite{getzler/jones:1994}, and on the chain level are due to Fresse \cite{fresse:2011}. Assuming this conjecture, notice that the obvious map of operads $\E_L\to \Com$ induces a map $\K\E_L \simeq \Sigma^{-L}\E_L\to \S^{-L}$. We speculate that the $\K\E_\bullet$-module structure on $\der_*F\Sigma^\infty$ is equivalent to the pullback along this hypothetical map of operads.

One consequence of this would be a new proof of the well-known fact that the Taylor tower of the functor $\Sigma^\infty\Omega^\infty\Sigma^\infty$ splits as a product of its layers. Indeed, we show in Example~\ref{ex:SO} that the action of the pro-sphere operad on $\der_*\Sigma^\infty\Omega^\infty$ is pulled back from an action by the single operad $\Com = \S^0$. By the hypothetical picture above, it would follow that the $\K\E_\bullet$-module structure on $\der_*\Sigma^\infty\Omega^\infty\Sigma^\infty$ is pulled back from a $\K\E_0$-module structure. But $\K\E_0$ is the trivial operad, so it would follow that the Taylor tower of $\Sigma^\infty\Omega^\infty\Sigma^\infty$ splits.

\subsubsection*{Describe the chain rule for functors from $\spectra$ to $\spectra$ on the level of modules over the sphere pro-operad.} Let $F, G$ be functors from spectra to spectra. By the results of~\cite{ching:2010}, there is an equivalence of symmetric sequences, where the right hand side is the composition product of $\der_*F$ and $\der_*G$
$$\der_*(FG)\simeq \der_*F \circ \der_*G.$$
By our present results, $\der_*F$, $\der_*G$ and $\der_*(FG)$ are modules over the sphere pro-operad. This suggests that there should be a composition product in the category of such modules, refining the composition product of symmetric sequences. Moreover, we speculate that the product can be made associative (or at least $A_\infty$-associative), and that the chain rule can be made associative as well.

\subsection*{Outline}
In section~\ref{sec:background} we review basic facts about Taylor towers and derivatives, and also about symmetric sequences, operads and modules. We also review our construction of the homotopy category of coalgebras over a comonad. The construction that we use here is a slight variation of the one in~\cite{arone/ching:2014}. In section~\ref{sec:pro-operads} we introduce pro-operads, and we define the category of modules over a pro-operad as the category of coalgebras of an associated comonad. In the long section~\ref{sec:topsp} we analyze functors from $\based$ to $\spectra$. In sections~\ref{sec:spsp} and~\ref{sec:bar-map} we do the same from functors from $\spectra$ to $\spectra$.

\subsection*{Acknowledgements}
Conversations with Bill Dwyer contributed considerably to the Koszul duality between $\E_L$-comodules and $\K\E_L$-modules that we use in this paper. The first author was supported by National Science Foundation grant DMS-0968221 and the second author by National Science Foundation grant DMS-1144149.

\section{Background}\label{sec:background}

\subsection{Taylor towers and derivatives}

\begin{definition}[Categories] \label{def:categories}
Let $\based$ be the category of based topological spaces and let $\spectra$ be the category of $S$-modules of EKMM \cite{elmendorf/kriz/mandell/may:1997}. We refer to the objects of $\spectra$ as \emph{spectra} to avoid confusion with the other uses of the term `module' in this paper. Each of the categories $\based$ and $\spectra$ is enriched in simplicial sets and we write $\Hom_{\mathscr{C}}(X,Y)$ for the simplicial set of maps from $X$ to $Y$ in the category $\mathscr{C}$. The category $\spectra$ is closed symmetric monoidal and we write $\Map(X,Y)$ for the internal mapping spectrum in $\spectra$.

The suspension spectrum and zeroth space functors give us the standard adjunction
\[ \Sigma^\infty: \based \rightleftarrows \spectra : \Omega^\infty. \]

Let $\finbased$ denote the full subcategory of $\based$ consisting of finite based cell complexes, and let $\finspec$ denote the full subcategory of $\spectra$ consisting of finite cell spectra (with respect to the usual generating cofibrations for the stable model structure on $\spectra$: see \cite[VII]{elmendorf/kriz/mandell/may:1997}).
\end{definition}

\begin{definition}[Functors] \label{def:functors}
Let $\mathscr{C}$ be either $\based$ or $\spectra$, and consider a functor $F: \mathscr{C} \to \spectra$. We say that $F$ is a \emph{homotopy functor} if it preserves weak equivalences, is \emph{simplicial} if it induces maps $\Hom_{\mathscr{C}}(X,Y) \to \Hom_{\spectra}(FX,FY)$ of simplicial sets, is \emph{finitary} if it preserves filtered homotopy colimits, and is \emph{pointed} if $F(*) \isom *$ (where $*$ denotes either a one-point space, in $\based$, or a trivial spectrum, in $\spectra$).
\end{definition}

\begin{definition}[Categories of functors] \label{def:categories-functors}
For $\mathscr{C}$ equal to either $\based$ or $\spectra$, let $[\mathscr{C}^{\mathsf{f}},\spectra]$ denote the category whose objects are the pointed simplicial functors $F: \mathscr{C}^{\mathsf{f}} \to \spectra$, and whose morphisms are the corresponding simplicially enriched natural transformations. Since $\mathscr{C}^{\mathsf{f}}$ is skeletally small, $[\mathscr{C}^{\mathsf{f}},\spectra]$ is a locally small category. Also note that any simplicial functor preserves simplicial homotopy equivalences. Since every object of $\mathscr{C}^{\mathsf{f}}$ is both cofibrant and fibrant (in the standard model structure on either $\based$ or $\spectra$), it follows that any object in $[\mathscr{C}^{\mathsf{f}},\spectra]$ also preserves weak equivalences.

We can extend a pointed simplicial functor $F: \mathscr{C}^{\mathsf{f}} \to \spectra$ to all of $\mathscr{C}$ by (enriched) homotopy left Kan extension along the inclusion $\mathscr{C}^{\mathsf{f}} \to \mathscr{C}$. The result of this construction is a reduced finitary homotopy functor. Moreover, any reduced finitary homotopy functor $F:\mathscr{C} \to \spectra$ arises, up to natural equivalence, in this way. We therefore view $[\mathscr{C}^{\mathsf{f}},\spectra]$ as a model for the collection of reduced finitary homotopy functors $\mathscr{C} \to \spectra$.

For a functor $F \in [\mathscr{C}^{\mathsf{f}},\spectra]$ we say that $F$ is \emph{$n$-excisive} if it takes a strongly-cocartesian $(n+1)$-cube in $\mathscr{C}^{\mathsf{f}}$ to a cartesian cube in $\spectra$. We say $F$ is \emph{polynomial} if it is $n$-excisive for some $n$.

The category $[\mathscr{C}^{\mathsf{f}},\spectra]$ has a projective model structure in which weak equivalences and fibrations are detected objectwise. We denote the associated homotopy category as $[\mathscr{C}^{\mathsf{f}},\spectra]^{\mathsf{h}}$. This homotopy category has subcategories given by restricting to functors satisfying various conditions: for example, we have the homotopy categories $[\mathscr{C}^{\mathsf{f}},\spectra]^{\mathsf{h}}_{n\mathsf{-exc}}$ (of $n$-excisive functors) and $[\mathscr{C}^{\mathsf{f}},\spectra]^{\mathsf{h}}_{\mathsf{poly}}$ (of all polynomial functors).
\end{definition}

\begin{definition}[Taylor tower and derivatives] \label{def:taylor}
For a pointed simplicial functor $F: \mathscr{C}^{\mathsf{f}} \to \spectra$ there is a \emph{Taylor tower} of pointed simplicial functors $P_nF: \mathscr{C}^{\mathsf{f}} \to \spectra$ of the form
\[ F \to \dots \to P_nF \to P_{n-1}F \to \dots \to P_2F \to P_1F \to P_0F = * \]
where $P_nF$ is $n$-excisive in the sense of \cite{goodwillie:1991}. The \emph{layers} of the Taylor tower are the functors $D_nF : \mathscr{C}^{\mathsf{f}} \to \spectra$ given by the objectwise homotopy fibres
\[ D_nF := \hofib(P_nF \to P_{n-1}F). \]
Goodwillie showed that for each $n \geq 1$ there is a spectrum $\der_nF$ with $\Sigma_n$-action such that
\[ D_nF(X) \homeq (\der_nF \smsh X^{\smsh n})_{h\Sigma_n} \]
for $X \in \mathscr{C}^{\mathsf{f}}$. We refer to $\der_nF$ as the \emph{\ord{$n$} derivative} of $F$.
\end{definition}

\begin{definition}[Cross-effects and co-cross-effects] \label{def:creff}
The \ord{$n$} derivative of a functor can be calculated using cross-effects. The \emph{\ord{$n$} cross-effect} of $F: \mathscr{C} \to \spectra$ is the functor of $n$ variables $\creff_nF : \mathscr{C}^{\times n} \to \spectra$ given by
\[ \creff_nF(X_1,\dots,X_n) := \thofib_{S \subset \un{n}} \left\{ F \left( \Wdge_{i \notin S} X_i \right) \right\}. \]
This is the iterated (or total) homotopy fibre of an $n$-cube formed by applying $F$ to the wedge sums of subsets of $X_1,\dots,X_n$, where the morphisms in the cube are given by the relevant collapse maps $X_i \to *$.

For spectrum-valued functors, the cross-effect can also be calculated by taking a total homotopy cofibre. The \emph{\ord{$n$} co-cross-effect} of $F: \mathscr{C}^{\mathsf{f}} \to \spectra$ is the functor $\creff^n F : \mathscr{C}^{\times n} \to \spectra$ given by
\[ \creff^n F(X_1,\dots,X_n) := \thocofib_{S \subset \un{n}} \left\{ F \left( \Wdge_{i \in S} X_i \right) \right\}. \]
The morphisms in this $n$-cube are given by the inclusion maps $* \to X_i$. For any $F: \mathscr{C} \to \spectra$, there is a natural equivalence
\[ \creff_nF(X_1,\dots,X_n) \homeq \creff^nF(X_1,\dots,X_n). \]
In particular, taking cross-effects commutes with both homotopy limits and homotopy colimits.
\end{definition}

One of the Goodwillie's main results is that the derivatives of a functor can be recovered by multilinearizing the cross-effects. Specifically, Theorem~6.1 of \cite{goodwillie:2003} implies that for a functor $F: \mathscr{C} \to \spectra$, we have a natural equivalence
\[ \der_nF \homeq P_{(1,\dots,1)}(\creff_nF)(S^0,\dots,S^0) \]
where the right-hand side is the multilinearization of the \ord{$n$} cross-effect of $F$. The multilinearization $P_{(1,\dots,1)}(\creff_nF)$ is the homotopy colimit of maps
\[ \creff_nF \to T_{(1,\dots,1)}(\creff_nF) \to T^2_{(1,\dots,1)}(\creff_nF) \to T^3_{(1,\dots,1)} \to \dots \]
where, for a functor $G: \mathscr{C}^n \to \spectra$ that is reduced in each variable, there are natural equivalences
\[ T_{(1,\dots,1)}G(X_1,\dots,X_n) \homeq \Sigma^{-n}G(\Sigma X_1,\dots,\Sigma X_n). \]
Thus we can express Goodwillie's result as an equivalence
\begin{equation} \label{eq:creff} \der_nF \homeq \hocolim_L \Sigma^{-nL} \creff_nF(S^L,\dots,S^L). \end{equation}
The maps in this homotopy colimit take the form
\begin{equation} \label{eq:creff-map} \Sigma^{-nL} \creff_nF(S^L,\dots,S^L) \to \Sigma^{-n(L+1)} \creff_nF(S^{L+1},\dots,S^{L+1}). \end{equation}
In building our models for the derivatives of $F$, we need models for these maps.

\begin{definition} \label{def:tensoring-map}
When the functor $H: \mathscr{C} \to \spectra$ is pointed and simplicial, it determines natural maps
\[ t_{X}: X \smsh H(Y) \to H(X \smsh Y) \]
for a based simplicial set $X$ and $Y \in \mathscr{C}$. We refer to this as the \emph{tensoring map} for $H$.
\end{definition}

\begin{lemma} \label{lem:creff-maps}
Let $G: \mathscr{C}^n \to \spectra$ be a model for the \ord{$n$} cross-effect of a functor $F: \mathscr{C} \to \spectra$ that is pointed simplicial in each variable. Then the map (\ref{eq:creff-map}) is, up to equivalence, given by the tensoring map
\[ \dgTEXTARROWLENGTH=4em \Sigma^{-nL}G(S^L,\dots,S^L) \homeq \Sigma^{-n(L+1)}\Sigma^n G(S^L,\dots,S^L) \arrow{e,t}{t_{S^1} \circ \dots \circ t_{S^1}} \Sigma^{-n(L+1)} G(S^{L+1},\dots,S^{L+1}). \]
\end{lemma}
\begin{proof}
It is sufficient to do the case $\mathscr{C} = \based$ since the result for $G : \spectra^n \to \spectra$ follows from that for $G(\Sigma^\infty -,\dots,\Sigma^\infty -) : \based^{n} \to \spectra$. We illustrate with the case $n = 1$. In this case, we have the following diagram. For $X \in \based$, with $CX$ denoting the (reduced) cone on $X$:
\[ \begin{diagram}
  \node[2]{G(X)} \arrow{sw} \arrow{se} \\
  \node{G(\Omega\Sigma X)} \arrow[2]{e,t}{\sim} \arrow{s} \node[2]{G(\holim(CX \to \Sigma X \leftarrow CX))} \arrow{s}  \\
  \node{\Omega G(\Sigma X)} \arrow[2]{e,t}{\sim} \node[2]{\holim(G(CX) \to G(\Sigma X) \leftarrow G(CX))}
\end{diagram} \]
The vertical maps come from the fact that $G$ is pointed simplicial, and the bottom square commutes by naturality of that enrichment. The composite of the left-hand two maps is the canonical map $G(X) \to T_1G(X)$, and the composite of the right-hand two is adjoint to the tensoring map $t_{S^1}$ for $G$. It is sufficient then to show that the top triangle commutes up to homotopy. Since $G$ is simplicial, it is therefore sufficient to show that the underlying triangle of spaces
\[ \begin{diagram}
  \node[2]{X} \arrow{sw} \arrow{se} \\
  \node{\Omega\Sigma X} \arrow[2]{e,t}{\sim} \node[2]{\holim(CX \to \Sigma X \leftarrow CX)}
\end{diagram} \]
commutes up to homotopy. A point in this homotopy limit consists of a path in $\Sigma X$ from a point in one cone to a point in the other cone. Writing $\Sigma X = [0,2]/\{0 \sim 2\} \smsh X$ with cones $[0,1] \smsh X$ and $[1,2] \smsh X$ (where $0$ and $2$ are treated as the basepoints), the required homotopy is given by
\[ (t,x) \mapsto (s \mapsto (t+(2-t)s,x)). \]
\end{proof}

\begin{definition}
We refer to the terms $\Sigma^{-nL} \creff_nF(S^L,\dots,S^L)$ in the homotopy colimit (\ref{eq:creff}) as the \emph{partially-stabilized cross-effects} of the functor $F$.
\end{definition}

\subsection{Symmetric sequences, operads, modules and comodules}

\begin{definition}[Symmetric sequences]
We write $\symseq$ for the category of \emph{symmetric sequences} in $\spectra$, that is, the category of functors $\Sigma \to \spectra$ where $\Sigma$ is the category of nonempty finite sets and bijections. Given $\A \in \symseq$ we often write $\A(n) := \A(\un{n})$ where $\un{n} = \{1,\dots,n\}$ and consider $\A$ as the sequence of spectra $\A(n)$, for $n \geq 1$, each with an action of $\Sigma_n$.

We also define a \emph{bisymmetric sequence} to be a functor $\B: \Sigma \times \Sigma \to \spectra$, so that $\B$ consists of spectra $\B(m,n)$ for $m,n \geq 1$ with commuting actions of $\Sigma_m$ and $\Sigma_n$.

The category $\symseq$ has a cofibrantly-generated model structure in which weak equivalences and fibrations are detected termwise. A symmetric sequence is \emph{$\Sigma$-cofibrant} if it is cofibrant in this model structure. The category $\symseq$ is also enriched in $\spectra$: the spectrum of maps between two symmetric sequences $\A$ and $\A'$ is given by
\[ \Map_{\Sigma}(\A,\A') := \prod_{n} \Map(\A(n),\A'(n))^{\Sigma_n}. \]
The model structure respects this enrichment, making $\symseq$ into a $\spectra$-enriched model category.

A symmetric sequence $\A$ is \emph{$n$-truncated} if $\A(m) = *$ for all $m > n$. We say that $\A$ is \emph{bounded} if it is $n$-truncated for some $n$. For any symmetric sequence $\A$ we have its $n$-truncation $\A_{\leq n}$ which is the symmetric sequence given by
\[ \A_{\leq n}(I) := \begin{cases}
  \A(I) & \text{if $|I| \leq n$}; \\
  \; * & \text{otherwise}.
\end{cases} \]
Associated to $\A$ is its \emph{truncation sequence} consisting of maps of symmetric sequences
\[ \A \to \dots \to \A_{\leq n} \to \A_{\leq(n-1)} \to \dots \to \A_{\leq 1}. \]
\end{definition}

\begin{definition}[Operads of spectra]
An \emph{operad of spectra} is a monoid in the category $\symseq$ with respect to the composition product of symmetric sequences (see, for example, \cite[2.11]{ching:2005}). Explicitly, an operad $\P$ consists of a symmetric sequence together with:
\begin{itemize}
  \item a composition map
  \[ \P_\alpha: \P(k) \smsh \Smsh_{j = 1}^{k} \P(n_j) \to \P(n) \]
  for each surjection $\alpha: \un{n} \epi \un{k}$ of nonempty finite sets, where $n_j := |\alpha^{-1}(j)|$; and
  \item a unit map $\eta: S \to \P(1)$;
\end{itemize}
that satisfy associativity, unitivity and equivariance conditions.

An operad of spectra is \emph{reduced} if the unit map $\eta:S \to \P(1)$ is an isomorphism. We consider only reduced operads in this paper.

There is a cofibrantly-generated model structure on the category of reduced operads of spectra, described in \cite[Appendix]{arone/ching:2011}, in which weak equivalences and fibrations are detected on the underlying symmetric sequences, i.e. termwise. An operad is \emph{$\Sigma$-cofibrant}, \emph{$n$-truncated} or \emph{bounded} if its underlying symmetric sequence has the corresponding property. The truncation sequence associated to an operad is a sequence of operads.

Note that by design our operads do not include a term $\P(0)$ so that they describe only non-unital structures.
\end{definition}

\begin{definition}[Modules over operads of spectra]
Given an operad $\P$ of spectra, a \emph{(right) $\P$-module} consists of a symmetric sequence $\M$ together with a right action of $\P$ with respect to the composition product. Explicitly, such an $\M$ has a structure map
\[ \M_\alpha: \M(k) \smsh \Smsh_{j = 1}^{k} \P(n_j) \to \M(n) \]
for each surjection $\alpha: \un{n} \epi \un{k}$, which satisfy appropriate conditions.

Since we only consider right modules over operads in this paper, we refer to these just as $\P$-modules and write $\Mod(\P)$ for the category of $\P$-modules (whose morphisms are maps of the underlying symmetric sequences that commute with the structure maps). The category $\Mod(\P)$ has a cofibrantly-generated stable model structure that is enriched in $\spectra$, in which weak equivalences and fibrations are detected termwise. Note that all homotopy limits and colimits of diagrams of $\P$-modules are also computed termwise. A $\P$-module is \emph{$\Sigma$-cofibrant}, \emph{$n$-truncated} or \emph{bounded} if its underlying symmetric sequence has the corresponding property. The truncation sequence of a $\P$-module is a sequence of $\P$-modules.
\end{definition}

In this paper we consider also what we call a `comodule' over an operad. In order to say what this means, we recall that a module can be interpreted as a spectrally-enriched functor $\un{\P} \to \spectra$ for a particular category $\un{\P}$ associated to the operad $\P$.

\begin{definition} \label{def:prop}
For an operad $\P$ of spectra we define a $\spectra$-enriched category $\un{\P}$ as follows:
\begin{itemize}
  \item the objects of $\un{\P}$ are the nonempty finite sets;
  \item for two nonempty finite sets $I$ and $J$, the morphism spectrum $\un{\P}(I,J)$ is given by
  \[ \un{\P}(I,J) := \Wdge_{I \epi J} \Smsh_{j \in J} P(I_j) \]
  where the wedge product is taken over all surjections $\alpha: I \epi J$ and we write $I_j := \alpha^{-1}(j)$.
  \item the composition and identity maps for the category $\un{\P}$ are determined by the operad multiplication and unit maps, respectively.
\end{itemize}
The category $\un{\P}$ is also known as the PROP associated to the operad $\P$.
\end{definition}

\begin{lemma} \label{lem:module-prop}
Let $\P$ be an operad of spectra. There is an equivalence between the category $\Mod(\P)$ of (right) $\P$-modules and the category of $\spectra$-enriched functors $\un{\P}^{op} \to \spectra$ (with morphisms the $\spectra$-enriched natural transformations).
\end{lemma}
\begin{proof}
A $\spectra$-enriched functor $\M: \un{\P}^{op} \to \spectra$ consists of objects $\M(I)$ for each finite set $I$, and (suitably associative and unital) maps
\[ \un{\P}(I,J) \to \Map(\M(J),\M(I)). \]
Equivalently, for each surjection $I \epi J$, there is a map
\[ \M(J) \smsh \Smsh_{j \in J} \P(I_j) \to \M(I). \]
These maps form precisely the data for a $\P$-module.
\end{proof}

Lemma~\ref{lem:module-prop} inspires the following definition.

\begin{definition} \label{def:comodule}
Let $\P$ be an operad of spectra. A \emph{(right) $\P$-comodule} is a $\spectra$-enriched functor $\un{\P} \to \spectra$. More explicitly, a $\P$-comodule consists of objects $\N(I)$, one for each nonempty finite set $I$, and maps
\[ \N(I) \smsh \Smsh_{j \in J} \P(I_j) \to \N(J), \]
one for each surjection $\alpha: I \epi J$. In particular, a $\P$-comodule $\N$ has an underlying symmetric sequence.

We write $\Comod(\P)$ for the category of $\P$-comodules (with morphisms given by the $\spectra$-enriched natural transformations). The category $\Comod(\P)$ is enriched over $\spectra$ and has a cofibrantly-generated model structure in which weak equivalences and fibrations are detected termwise, and homotopy limits and colimits are computed termwise. We say that a $\P$-comodule is \emph{$\Sigma$-cofibrant}, \emph{$n$-truncated} or \emph{bounded} if its underlying symmetric sequence has the corresponding property.
\end{definition}

\begin{example}
A right $\Com$-comodule can be identified with a functor $\Omega \to \spectra$ where $\Omega$ is the category of nonempty finite sets and surjections. A right $\Com$-module can, similarly, be identified with a functor $\Omega^{op} \to \spectra$.
\end{example}

The dual definitions of module and comodule permit a natural `coend' construction between a module and comodule. We require a homotopy-invariant version of this construction which we now describe.

\begin{definition} \label{def:homsmsh}
Let $\P$ be a reduced operad of spectra, $\N$ a $\P$-comodule and $\M$ a $\P$-module. We define the spectrum $\N \homsmsh_{\P} \M$ to be the realization of the simplicial spectrum given by
\[ [r] \mapsto \Wdge_{n_0, \dots, n_r} \N(n_0) \smsh_{\Sigma_{n_0}} \un{\P}(n_0,n_1) \smsh_{\Sigma_{n_1}} \dots \smsh_{\Sigma_{n_{r-1}}} \un{\P}(n_{r-1},n_r) \smsh_{\Sigma_{n_r}} \M(n_r) \]
where the wedge sum is taken over all sequences of positive integers $n_0,\dots,n_r$ (though only the non-increasing sequences contribute non-trivial terms) and with face maps
\begin{itemize}
  \item $d_0$ given by the comodule structure maps $\N(n_0) \smsh_{\Sigma_{n_0}} \un{\P}(n_0,n_1) \to \N(n_1)$;
  \item $d_i$ for $i = 1,\dots,r-1$ given by the operad composition maps $\un{\P}(n_{i-1},n_i) \smsh_{\Sigma_{n_i}} \un{\P}(n_i,n_{i+1}) \to \un{\P}(n_{i-1},n_{i+1})$;
  \item $d_r$ given by the module structure maps $\un{\P}(n_{r-1},n_r) \smsh_{\Sigma_{n_r}} \M(n_r) \to \M(n_{r-1})$.
\end{itemize}
and degeneracy maps
\begin{itemize}
  \item $s_j$ for $j = 0,\dots,r$ given by the operad unit map $S \isom \un{\P}(n_j,n_j)_{\Sigma_{n_j}}$.
\end{itemize}
\end{definition}

\begin{lemma} \label{lem:homsmsh}
Let $\P$ be a $\Sigma$-cofibrant operad of spectra and $\N$ a $\Sigma$-cofibrant $\P$-comodule. Let $\M \weq \M'$ be a weak equivalence of $\P$-modules. Then the induced map
\[ \N \homsmsh_{\P} \M \to \N \homsmsh_{\P} \M' \]
is a weak equivalence of spectra. Similarly, let $\M$ be a $\Sigma$-cofibrant $\P$-module and $\N \weq \N'$ a weak equivalence of $\P$-comodules. Then the induced map
\[ \N \homsmsh_{\P} \M \to \N' \homsmsh_{\P} \M \]
is a weak equivalence of spectra.
\end{lemma}
\begin{proof}
The simplicial spectra involved here are all proper in the sense of \cite[X.2.2]{elmendorf/kriz/mandell/may:1997}. The conditions imply that the induced maps of simplicial spectra are levelwise weak equivalences. Then \cite[X.2.4]{elmendorf/kriz/mandell/may:1997} implies that the given maps are weak equivalences.
\end{proof}

Lemma~\ref{lem:homsmsh} tells us that, when $\P$ is $\Sigma$-cofibrant, the homotopy coend $\N \homsmsh_{\P} \M$ has the `correct' homotopy type if either $\N$ or $\M$ is $\Sigma$-cofibrant.

\subsection{Coalgebras over comonads and their homotopy theory} \label{sec:comonads}

In this paper we are concerned with comonads on the category $\symseq$ of symmetric sequences of spectra.

Recall that a \emph{comonad} $\C$ on $\symseq$ is an endofunctor $\C: \symseq \to \symseq$ equipped with natural transformations $\nu : \C \to \C\C$ (the \emph{comultiplication}) and $\epsilon: \C \to \mathbf{I}_{\symseq}$ (the \emph{counit}) that make $\C$ into a comonoid with respect to composition of functors. For a comonad $\C: \symseq \to \symseq$, a \emph{$\C$-coalgebra} is a symmetric sequence $\A$ together with a structure map $\theta: \A \to \C\A$ that forms a coaction of the comonoid $\C$. Morphisms of coalgebras are maps in $\symseq$ that commute with the structure maps, and we have a category $\Coalg(\C)$ of coalgebras over $\C$. We say that a $\C$-coalgebra is \emph{$\Sigma$-cofibrant}, \emph{$n$-truncated} or \emph{bounded} if its underlying symmetric sequence has the corresponding property.

Our main examples of comonads arise from operads.

\begin{definition} \label{def:comonad-operad}
For an operad $\P$ of spectra, we define an endofunctor
\[ \C_\P: \symseq \to \symseq \]
by
\[ \C_\P(\A)(k) := \prod_{n} \left[ \prod_{\un{n} \epi \un{k}} \Map(\P(n_1) \smsh \dots \smsh \P(n_k), \A(n)) \right]^{\Sigma_n} \]
where $\Sigma_n$ acts on the set of surjections $\un{n} \epi \un{k}$ by pre-composition, and $\Sigma_k$ acts by post-composition.

The operad structure on $\P$ determines a comonad structure on $\C_\P$ with the operad composition determining the comultiplication, and the operad unit map determining the counit. Note that the functor $\C_\P$ is enriched in $\spectra$.
\end{definition}

\begin{lemma} \label{lem:operad-cofree}
The comonad $\C_\P$ is that associated to an adjunction
\[ \U_\P : \Mod(\P) \rightleftarrows \symseq : \R_\P \]
where $\U_\P$ is the forgetful functor. In particular, for any $\P$-module $\M$, the symmetric sequence $\U_\P(\M)$ has a canonical $\C_\P$-coalgebra structure.
\end{lemma}
\begin{proof}
The right adjoint $\R_\P$ is defined by the same formula as $\C_\P$. The $\P$-module structure on $\R_\P(\A)$ is determined by the operad composition map for $\P$. It is easy to check that $\R_\P$ so defined is right adjoint to $\U_\P$, and then $\C_P  = \U_\P\R_\P$ is the associated comonad. The $\C_\P$-coalgebra structure on $\U_\P(\M)$ is given by the unit map
\[ \U_\P(\M) \to \U_\P\R_\P\U_\P(\M) \]
for the adjunction $(\U_\P,\R_\P)$.
\end{proof}

\begin{definition}
We refer to the $\P$-module $\R_\P(\A)$ as the \emph{cofree $\P$-module} associated to the symmetric sequence $\A$.
\end{definition}

\begin{lemma}
The functor $\M \mapsto \U_\P\M$ determines an equivalence of categories
\[ \Mod(\P) \simeq \Coalg(\C_\P). \]
\end{lemma}
\begin{proof}
Since $\U_\P$ is right adjoint to the free $\P$-module functor $\F_\P$, the composite $\U_\P \R_\P$ is right adjoint to the functor monad $\U_\P \F_\P$ whose algebras are the $\P$-modules. The structure maps for a right $\P$-module correspond precisely under this adjunction to the structure maps for a $\C_\P$-coalgebra structure on the same symmetric sequence.
\end{proof}

\begin{definition} \label{def:comonad-operad-map}
Suppose that $\phi: \P \to \P'$ is a morphism of operads of spectra. There is an induced map of comonads
\[ \C_\phi: \C_{\P'} \to \C_\P. \]
If $\M$ is a $\C_{\P'}$-coalgebra with structure map $\theta$ then the composite
\[ \dgTEXTARROWLENGTH=2em \M \arrow{e,t}{\theta} \C_{\P'}\M \arrow{e,t}{\C_\phi} \C_\P\M \]
gives $\M$ the structure of a $\C_\P$-module. This corresponds to pulling back the $\P'$-module structure on $\M$ along the map $\phi$. This construction makes $\P \mapsto \C_\P$ into a contravariant functor from the category of operads in $\spectra$ to the category of comonads on $\symseq$.
\end{definition}

\begin{lemma}
Let $\P$ be a cofibrant operad (in the projective model structure on the category of reduced operads of spectra). Then the comonad $\C_\P$ preserves all weak equivalences.
\end{lemma}
\begin{proof}
We can rewrite $\C_\P(A)(k)$ as
\[ \prod_{n} \left[ \prod_{n = n_1+\dots+n_k} \Map(\P(n_1) \smsh \dots \smsh \P(n_k), \A(n))^{\Sigma_{n_1} \times \dots \times \Sigma_{n_r}} \right] \]
Since $\P$ is cofibrant, each $\P(n_i)$ is a cofibrant $\Sigma_{n_i}$-spectrum, and hence $\P(n_1) \smsh \dots \smsh \P(n_k)$ is a cofibrant $\Sigma_{n_1} \times \dots \times \Sigma_{n_k}$-spectrum. It follows that $\C_\P$ preserves weak equivalences (since all objects in $\symseq$ are fibrant).
\end{proof}

We now turn to the homotopy theory of coalgebras over comonads. This topic was studied in detail in \cite[Section 1]{arone/ching:2014} and we first recall some definitions and results from there, though with one change. In this paper we work with comonads on the category of symmetric sequences that respect the spectral enrichment of that category. In this case we are able to define mapping spectra for coalgebras as well as mapping spaces.

\begin{definition} \label{def:derived-map-coalgebras}
Let $\C: \symseq \to \symseq$ be a $\spectra$-enriched comonad on the category of symmetric sequences. We define the \emph{derived mapping spectrum} for two $\C$-coalgebras $\A,\A'$ to be the spectrum
\[ \widetilde{\Map}_{\C}(\A,\A') := \Tot \left[ \Map_{\Sigma}(\A,\A') \rightrightarrows \Map_{\Sigma}(\A,\C(\A')) \Rrightarrow \dots \right]. \]
This is the (fat) totalization of a simplicial spectrum constructing using the comonad structure on $\C$, the coalgebra structures on $\A$ and $\A'$, and the spectral enrichment of $\C$.

As in \cite{arone/ching:2014} (though with mapping spectra instead of mapping spaces) the derived mapping spectra $\widetilde{\Map}_{\C}(\A,\A')$ determine an $A_\infty$-$\spectra$-enriched category with composition maps parameterized by a certain $A_\infty$-operad $\mathscr{A}$. We obtain the \emph{homotopy category of $\C$-coalgebras}, which we denote $\Coalg(\C)^{\mathsf{h}}$, with objects the $\Sigma$-cofibrant $\C$-coalgebras, and morphism sets
\[ [\A,\A']_{\C} := \pi_0 \; \widetilde{\Map}_{\C}(\A,\A'). \]
A morphism $f: \A \to \A'$ in this homotopy category is determined by a \emph{derived morphism of $\C$-coalgebras} which consists of compatible maps of symmetric sequences
\[ f_k: \Delta^k_+ \smsh \A \to \C^k(\A'). \]
By \cite[1.16]{arone/ching:2014}, such a morphism induces an isomorphism in the homotopy category if and only if the map $f_0 : \A \to \A'$ is an equivalence of symmetric sequences.
\end{definition}

In our classifications of analytic functors, we require a slight modification of the homotopy category constructed in Definition~\ref{def:derived-map-coalgebras}.

\begin{definition} \label{def:truncated-map-coalgebras}
Let $\C: \symseq \to \symseq$ be a $\spectra$-enriched comonad on the category of symmetric sequences. We say that $\C$ \emph{preserves truncations} if whenever $\A$ is $n$-truncated, then $\C(\A)$ is also $n$-truncated. In this case, for every $\C$-coalgebra $\A$ and each $n$, there is a unique $\C$-coalgebra structure on the truncations $\A_{\leq n}$ such that the truncation sequence for $\A$ consists of (strict) maps of $\C$-coalgebras.

We define the \emph{pro-truncated mapping spectrum} for two $\C$-coalgebras $\A,\A'$ to be the spectrum
\[ \widetilde{\Map^{\mathsf{t}}}_{\C}(\A,\A') := \holim_n \widetilde{\Map}_{\C}(\A,\A'_{\leq n}) \]
The terms on the right-hand side are the derived mapping spectra of Definition~\ref{def:derived-map-coalgebras} and the maps in the homotopy limit are induced by the truncation maps $\A'_{\leq n} \to \A'_{\leq (n-1)}$ which are (strict) morphisms of $\C$-coalgebras.
\end{definition}

\begin{proposition} \label{prop:pro-truncated-coalgebras}
Suppose that $\C: \symseq \to \symseq$ is a $\spectra$-enriched comonad that preserves truncations and weak equivalences between $\Sigma$-cofibrant objects. Then the pro-truncated mapping spectra of Definition~\ref{def:truncated-map-coalgebras} are the mapping spectra in an $A_\infty$-$\spectra$-enriched category whose objects are the $\Sigma$-cofibrant $\C$-coalgebras.
\end{proposition}
\begin{proof}
The construction is very similar to that of \cite[1.14]{arone/ching:2014}. Since $\C$ preserves truncations, we have isomorphisms
\[ \Map_{\Sigma}(\A,\C^k(\A'_{\leq n})) \isom \Map_{\Sigma}(\A_{\leq n}, \C^k(\A'_{\leq n})). \]
We therefore have, in the notation of \cite{arone/ching:2014}, maps
\[ \Map_{\Sigma}(\A,\C^\bullet(\A'_{\leq n})) \binBox \Map_{\Sigma}(\A',\C^\bullet(\A''_{\leq n})) \to \Map_{\Sigma}(\A,\C^\bullet(\A''_{\leq n})) \]
which, on taking totalizations, give the required composition maps for an $A_\infty$-category.
\end{proof}

\begin{definition} \label{def:truncated-category-coalgebras}
With $\C$ as above, the \emph{pro-truncated homotopy category of $\C$-coalgebras}, which we denote as $\Coalg^{\mathsf{t}}(\C)^{\mathsf{h}}$, has objects the $\Sigma$-cofibrant $\C$-coalgebras and morphism sets
\[ [\A,\A']^{\mathsf{t}}_{\C} := \pi_0 \; \widetilde{\Map^{\mathsf{t}}}_{\C}(\A,\A'). \]
A morphism $f:\A \to \A'$ in this homotopy category is determined by what we call a \emph{pro-truncated derived morphism of $\C$-coalgebras}. This consists of maps of symmetric sequences
\[ f_{k,n} : \Delta^k_+ \smsh \A \to \C^k(\A'_{\leq n}) \]
that are compatible both with the truncation sequence for $\A'$, and with various cosimplicial structure maps.

In particular, the maps $f_{0,n}$ together make up a morphism of symmetric sequences
\[ f_0: \A \to \A' \]
and the maps $f_{1,n}$ ensure that the large rectangle in the following diagram commutes up to homotopy.
\[ \begin{diagram}
  \node{\A} \arrow{e,t}{f_0} \arrow{s} \node{\A'} \arrow{s,..} \arrow{e,t}{\sim} \node{\holim_n \A'_{\leq n}} \arrow{s} \\
  \node{\C(\A)} \arrow{e,t}{\C(f_0)} \node{\C(\A')} \arrow{e} \node{\holim_n \C(\A'_{\leq n})}
\end{diagram} \]
though this does not in general imply that the left-hand square commutes up to homotopy.
\end{definition}

\begin{remark}
The pro-truncated homotopy category is in general coarser than the homotopy category of $\C$-coalgebras defined in \cite{arone/ching:2014}. The subcategories of bounded coalgebras are equivalent in both homotopy categories, but, when $\C$ does not commute with homotopy limits, the pro-truncated category can have more equivalences between unbounded coalgebras.
\end{remark}

\begin{proposition} \label{prop:pro-truncated-equivalence}
A pro-truncated derived morphism of $\Sigma$-cofibrant $\C$-coalgebras $f: \A \to \A'$, in the sense of Definition~\ref{def:truncated-category-coalgebras} induces an isomorphism in the pro-truncated homotopy category if and only if the underlying map
\[ f_0 : \A \to \A' \]
is a weak equivalence of symmetric sequences.
\end{proposition}
\begin{proof}
The argument is almost identical to that of \cite[1.16]{arone/ching:2014} using the fact that a map $f_0: A \to A'$ of $\Sigma$-cofibrant symmetric sequences is a weak equivalence if and only if, for every bounded symmetric sequence $X$, the induced map
\[ f_0^*: \Map_{\Sigma}(A',X) \to \Map_{\Sigma}(A,X) \]
is a weak equivalence.
\end{proof}

\begin{lemma}
Let $\P$ be a $\Sigma$-cofibrant operad of spectra. Then the comonad $\C_\P$ preserves truncations and the following categories are equivalent:
\begin{enumerate}
  \item the homotopy category associated to the projective model structure on $\Mod(\P)$;
  \item the homotopy category of $\C_\P$-coalgebras of Definition~\ref{def:derived-map-coalgebras};
  \item the pro-truncated homotopy category of $\C_\P$-coalgebras of Definition~\ref{def:truncated-category-coalgebras}.
\end{enumerate}
\end{lemma}
\begin{proof}
Let $\M,\M'$ be $\Sigma$-cofibrant $\P$-modules (i.e. $\C_\P$-coalgebras). Then the homotopy category associated to the projective model structure on $\Mod(\P)$ is determined by the mapping spectra
\[ \Map_{P}(\M,\M') \homeq \Tot\Map_{\Sigma}((\U_\P\F_P)^\bullet(\M),\M') \]
where $\F_P: \symseq \to \Mod(\P)$ is the free $\P$-module functor. Applying the adjunctions $(\U_\P,\R_\P)$ and $(\F_\P,\U_\P)$, this is equivalent to the derived mapping spectra of \cite[1.10]{arone/ching:2014}:
\[ \widetilde{\Map}_{\C_{\P}}(\M,\M') = \Tot\Map_{\Sigma}(\M,(\U_\P\R_\P)^\bullet\M'). \]
Finally, notice that $\C_\P$ commutes with homotopy limits, from which it follows that
\[ \widetilde{\Map}_{\C_{\P}}(\M,\M') \homeq \widetilde{\Map^{\mathsf{t}}}_{\C_{\P}}(\M,\M'). \]
The homotopy categories associated to these three mapping spectrum constructions are therefore equivalent.
\end{proof}

\section{Pro-operads and their modules} \label{sec:pro-operads}

We now consider inverse sequences of operads and sequences of modules over them. In this section we develop some theory for these objects and see how they are related to coalgebras over associated comonads.

\begin{definition}[Modules over pro-operads] \label{def:module-pro}
A \emph{pro-operad} is a cofiltered diagram in the category of operads of spectra. The only pro-operads that we consider in this paper are straightforward inverse sequences of operads so we focus on these. Let $\P_\bullet$ denote a sequence of operads of the form
\[ \dots \to \P_2 \to \P_1 \to \P_0. \]

A \emph{module over the pro-operad $\P_\bullet$} is a direct sequence $\M^\bullet$ of symmetric sequences
\[ \M^0 \to \M^1 \to \M^2 \to \dots \]
together with a $\P_L$-module structure on $\M^L$ for each $L$, such that the map
\[ \M^L \to \M^{L+1} \]
is a morphism of $\P_{L+1}$-modules. Here, $\M^L$ inherits a $\P_{L+1}$-module structure via the operad morphism $\P_{L+1} \to \P_L$.

A \emph{morphism} of $\P_\bullet$-modules $f: \M_1^\bullet \to \M_2^\bullet$ consists of maps $f^L: \M_1^L \to \M_2^L$ for each $L$ such that each diagram
\[ \begin{diagram}
  \node{\M_1^L} \arrow{e} \arrow{s} \node{\M_1^{L+1}} \arrow{s} \\
  \node{\M_2^L} \arrow{e} \node{\M_2^{L+1}}
\end{diagram} \]
commutes. There is a category $\Mod(\P_\bullet)$ whose objects are the $\P_\bullet$-modules, with morphisms as defined above. By recent work of Greenlees and Shipley~\cite{greenlees/shipley:2013}, the category $\Mod(\P_\bullet)$ can be given a \emph{strict projective model structure} arising from the projective model structures on the categories $\Mod(\P_L)$, in which a morphism $f$ is a weak equivalence (or a fibration) if and only if each map $f^L$ is a weak equivalence (or, respectively, a fibration) of $\P_L$-modules.
\end{definition}

\begin{remark} \label{rem:module-pro-operad}
More generally, given a pro-operad $\P_\bullet$ indexed by an arbitrary (small) cofiltered category $\mathbb{I}$, we can define a $\P_\bullet$-module to be a diagram of symmetric sequences $\M^\bullet$ indexed by $\mathbb{I}^{op}$ such that each $\M^i$ is a $\P_i$-module, and, for each morphism $i \to i'$ in $\mathbb{I}$, the map $\M^{i'} \to \M^i$ is a morphism of $\P_i$-modules.
\end{remark}

We now consider the structure that is inherited by the homotopy colimit of a module over a pro-operad. The main result we need for this is that the homotopy colimit of a diagram of comonads (calculated objectwise) inherits a canonical comonad structure.

\begin{definition}[Colimits of comonads] \label{def:colim-comonad}
Let $\C_\bullet$ be a diagram in the category of comonads on $\symseq$ indexed by a small category $\mathbb{I}$. We define the colimit $\colim_i \C_i$ objectwise, that is
\[ (\colim_i \C_i)(\A) := \colim_i (\C_i(\A)). \]
Then $\colim_i \C_i$ has a canonical comonad structure with comultiplication given by the composite
\[ \colim_i \C_i(\A) \to \colim_i \C_i(\C_i(\A)) \to \colim_i \C_i(\colim_{i'} \C_{i'}(\A)) \]
in which the first map is built from the comultiplication maps for the comonads $\C_i$ and the second is induced by the natural maps $\C_i \to \colim_{i'} \C_{i'}$; and counit given by
\[ \colim_i \C_i(\A) \to \A \]
built from the counit maps $\C_i(\A) \to \A$ for the individual comonads $\C_i$.
\end{definition}

\begin{definition}[Tensoring of comonads] \label{def:tensor-comonad}
Let $\C$ be a simplicially-enriched comonad on the category $\symseq$ and let $X$ be a simplicial set. We define the objectwise tensoring of $\C$ by $X$ to be the functor $X \otimes \C$ given by
\[ (X \otimes \C)(\A) := X_+ \smsh \C(\A). \]
Then $X \otimes \C$ inherits a canonical comonad structure with comultiplication given by the composite
\[ X_+ \smsh \C(\A) \to X_+ \smsh X_+ \smsh \C(\C(\A)) \to X_+ \smsh \C(X_+ \smsh \C(\A)) \]
in which the first map is built from the diagonal on $X$ and the comultiplication for $\C$, and the second is given by the simplicial enrichment of $\C$, and counit map given by
\[ X_+ \smsh \C(\A) \to \C(\A) \to \A \]
where the first map is induced by the collapse map $X \to *$ and the second is the counit for $\C$.
\end{definition}

\begin{definition}[Homotopy colimits of comonads] \label{def:hocolim-comonad}
Let $\C_\bullet$ be a diagram of simplicially-enriched comonads on $\symseq$ indexed by a small category $\mathbb{I}$. We define the homotopy colimit $\hocolim_i \C_i$ objectwise by the coend
\[ (\hocolim_i \C_i)(\A) := \hocolim_i (\C_i(\A)) = \Delta^r \otimes_{\un{r} \in \mathbf{\Delta}} \left[ \Wdge_{i_0 \to \dots \to i_r} \C_{i_r}(\A) \right] \]
where we use the standard model for the homotopy colimit of a diagram of spectra as the geometric realization of a simplicial object: $\mathbf{\Delta}$ is the simplicial indexing category and this is a coend over $\mathbf{\Delta}$. Notice that the homotopy colimit is defined by a combination of tensoring with the simplicial sets $\Delta^r$ and taking colimits. Using the constructions of Definitions~\ref{def:colim-comonad} and \ref{def:tensor-comonad} we thus obtain a canonical comonad structure on the functor $\hocolim_i \C_i$.
\end{definition}

\begin{definition} \label{def:inverse-operad-comonad}
Let $\P_\bullet$ be an inverse sequence of operads in $\spectra$ as before. By the construction in Definition~\ref{def:comonad-operad-map} we obtain a corresponding sequence of comonads
\[ \C_{\P_0} \to \C_{\P_1} \to \C_{\P_2} \to \dots \]
We define a new comonad
\[ \C_{\P_\bullet} : \symseq \to \symseq \]
by
\[ \C_{\P_\bullet}(\A) := \hocolim_L \C_{\P_L}(\A) \]
with comonad structure given by the construction of Definition~\ref{def:hocolim-comonad}. Note that the functor $\C_{\P_\bullet}$ is enriched in $\spectra$ since each individual $\C_{\P_L}$ is.
\end{definition}

We now see that when $\M^\bullet$ is a module over the pro-operad $\P_\bullet$, the homotopy colimit of the sequence $\M^\bullet$ becomes a coalgebra over the comonad $\C_{\P_\bullet}$ of Definition~\ref{def:inverse-operad-comonad}. The coalgebra structure map can be built from corresponding maps for tensors and strict colimits, just as is the comonad structure on $\C_{\P_\bullet}$.

\begin{definition} [Colimits of coalgebras] \label{def:colim-coalgebra}
Let $\C_\bullet$ be a diagram of comonads on $\symseq$ indexed by a small category $\mathbb{I}$, and let $\M^\bullet$ be a diagram of symmetric sequences indexed by $\mathbb{I}^{op}$ such that, for each morphism $i \to i'$ in $\mathbb{I}$, the map
\[ \M^{i'} \to \M^i \]
is a morphism of $\C_{i'}$-coalgebras (where the coalgebra structure on $\M^i$ is given by the composite $\M^i \to \C_i(\M^i) \to \C_{i'}(\M^i)$). So $\M^\bullet$ is a `module' over the diagram $\C_\bullet$ in the sense of Remark~\ref{rem:module-pro-operad}.

Then we define a $(\colim_i \C_i)$-coalgebra structure on the symmetric sequence $\colim_i \M^i$ with structure map
\[ \colim_i \M^i \to \colim_i \C_i(\M^i) \to \colim_i \C_i(\colim_{i'} \M^{i'}). \]
\end{definition}

\begin{definition}[Tensoring of coalgebras] \label{def:tensor-coalgebra}
Let $\C$ be a simplicially-enriched comonad on $\symseq$, $\M$ a $\C$-coalgebra, and $X$ a simplicial set. Then we make $X \otimes \M := X_+ \smsh \M$ into a $(X \otimes \C)$-coalgebra via the structure map
\[ X_+ \smsh \M \to X_+ \smsh X_+ \smsh \C(\M) \to X_+ \smsh \C(X_+ \smsh \M). \]
\end{definition}

\begin{definition}[Homotopy colimits of coalgebras] \label{def:hocolim-coalgebra}
Let $\C_\bullet$ and $\M^\bullet$ be as in Definition~\ref{def:colim-coalgebra}. We then get a $(\hocolim_i \C_i)$-coalgebra structure on the homotopy colimit
\[ \hocolim_i \M^i = \Delta^r_+ \smsh_{\un{r} \in \mathbf{\Delta}} \left[ \Wdge_{i_0 \to \dots \to i_r} \M^{i_r} \right] \]
by combining the constructions of Definitions~\ref{def:colim-coalgebra} and \ref{def:tensor-coalgebra}, in the manner of Definition~\ref{def:hocolim-comonad}.
\end{definition}

\begin{definition} \label{def:hocolim-module-coalgebra}
Let $\P_\bullet$ be an inverse sequence of operads and $\M^\bullet$ a $\P_\bullet$-module in the sense of Definition~\ref{def:module-pro}. Let $\M$ be the homotopy colimit symmetric sequence
\[ \M := \hocolim_L \M^L. \]
By the construction of Definition~\ref{def:hocolim-coalgebra}, $\M$ inherits the structure of a $\C_{\P_\bullet}$-coalgebra.
\end{definition}

\begin{remark} \label{rem:holim-operad}
If $\M^\bullet$ is a $\P_\bullet$-module, then each term $\M^L$ is a module over the operad $\lim \P_\bullet$ given by the inverse limit (in the category of operads) of the sequence $\P_\bullet$. It follows that the homotopy colimit $\M$ also has a canonical structure of a module over $\lim \P_\bullet$. In general, the structure of a $\C_{\P_\bullet}$-coalgebra includes and extends that of a $\lim \P_\bullet$ module. Another way to see this is via the existence of a canonical map of comonads
\[ \C_{\P_\bullet} \to \C_{\lim \P_\bullet} \]
which is typically not an equivalence.
\end{remark}

There is one technical construction that we need for modules over inverse sequences of operads. That is, we need to be able to construct a replacement for a given module in which all terms are $\Sigma$-cofibrant. We show how to do this now.

\begin{definition} \label{def:cofcoalg}
Let $\P_\bullet$ be an inverse sequence of reduced operads in $\spectra$ such that each $\P_L$ is cofibrant in the projective model structure on reduced operads of spectra. Let
\[ \un{\M}^\bullet: \quad \un{\M}^0 \to \un{\M}^1 \to \un{\M}^2 \to \dots \]
be a $\P_\bullet$-module. We recursively construct a commutative diagram of the form
\begin{equation} \label{eq:eqM} \begin{diagram}
  \node{\M^0} \arrow{s,l}{\sim} \arrow{e} \node{\M^1} \arrow{s,l}{\sim} \arrow{e} \node{\M^2} \arrow{s,l}{\sim} \arrow{e} \node{\dots} \\
  \node{\un{\M}^0} \arrow{e} \node{\un{\M}^1} \arrow{e} \node{\un{\M}^2} \arrow{e} \node{\dots}
\end{diagram} \end{equation}
in which $\M^L \weq \un{\M}^L$ is a weak equivalence of $\P_L$-modules, and each $\M^L$ is $\Sigma$-cofibrant.

First, let $\M^0 \weq \un{\M}^0$ be a cofibrant replacement for $\un{\M}^0$ as a $\P_0$-module. Note that $\M^0$ is then $\Sigma$-cofibrant because a cofibrant module over the cofibrant operad $\P_0$ is $\Sigma$-cofibrant by \cite[2.3.14]{arone/ching:2011}. Now, suppose that we have built the diagram as far as the weak equivalence $\M^L \weq \un{\M}^L$. We can factor the composite
\[ \M^L \weq \un{\M}^L \arrow{e} \un{\M}^{L+1} \]
in the category of $\P_{L+1}$-modules as a cofibration followed by a trivial fibration which we write as
\[ \M^L \; \arrow{e,V} \; \M^{L+1} \; \arrow{e,t,A}{\sim} \; \un{\M}^{L+1}. \]
Then $\M^L \; \arrow{e,V} \; \M^{L+1}$ is a cofibration of $\P_{L+1}$-modules with $\Sigma$-cofibrant domain, hence is a $\Sigma$-cofibration by \cite[A.0.11]{arone/ching:2011}. It follows that $\M^{L+1}$ is $\Sigma$-cofibrant as required. Recursively this defines the diagram (\ref{eq:eqM}).

We refer to the $\P_\bullet$-module $\M^\bullet$ as a \emph{$\Sigma$-cofibrant replacement} for $\un{\M}^\bullet$. Notice that because each $\M^L$ is $\Sigma$-cofibrant, it follows that the homotopy colimit
\[ \M := \hocolim_L \M^L \]
is a $\Sigma$-cofibrant $\C_{\P_\bullet}$-coalgebra. In particular $\M$ determines an object in the homotopy category of $\C_{\P_\bullet}$-coalgebras described in Definition~\ref{def:derived-map-coalgebras}, as well as in the pro-truncated homotopy category of Definition~\ref{def:truncated-category-coalgebras}.
\end{definition}

\section{Functors from based spaces to spectra} \label{sec:topsp}

We now take up the study of homotopy functors from based spaces to spectra. Recall that we are writing $\finbased$ for the full subcategory of $\based$ consisting of the finite cell complexes, and that $[\finbased,\spectra]$ is the category of pointed simplicially-enriched functors $F: \finbased \to \spectra$.

In \cite{arone/ching:2014} we proved that the Taylor tower (expanded at the one-point space $*$) of a functor $F \in [\finbased,\spectra]$ is determined by the action of a certain comonad $\C$ on the symmetric sequence $\der_*F$ formed by the derivatives of $F$ (at $*$). The comonad $\C = \der_*\Phi$ arises from an adjunction
\begin{equation} \label{eq:adj} \der_* : [\finbased,\spectra] \rightleftarrows \symseq : \Phi \end{equation}
in which the left adjoint $\der_*$ is a model for the functor taking a (cofibrant) $F$ to its symmetric sequence of derivatives. The main result of this section is a separate description of $\C$ as the comonad associated to an inverse sequence of operads via the constructions of the previous section. We start by describing that sequence.

\subsection{Koszul duals of the stable little disc operads} \label{sec:little-discs}

The \emph{little disc operads} of Boardman and Vogt \cite{boardman/vogt:1973} were introduced to classify iterated loop spaces, but have since been seen to arise in a variety of other contexts in algebraic topology: for example, Francis \cite{francis:2012} shows that algebras over the little $n$-discs operad classify certain types of homology theory on $n$-dimensional manifolds; the Deligne Conjecture (proved by McClure-Smith \cite{mcclure/smith:2002} and others) shows that the Hochschild complex on an associative algebra forms an algebra over the chain model for the little $2$-discs operad.

In this paper we show that the sequence of operads of spectra given by the suspension spectra of the topological little disc operads captures the information needed to recover the Taylor tower of a functor $F: \finbased \to \spectra$ from its derivatives. More precisely, the Koszul duals of this direct sequence form an inverse sequence $\K\E_\bullet$ whose associated comonad $\C_{\K\E_\bullet}$ coacts on the derivatives of any such $F$.

Throughout this paper we use the Fulton-MacPherson models for the little disc operads, as described by Getzler and Jones \cite{getzler/jones:1994}. The only place where the precise model matters is in our proof of Lemma~\ref{lem:unstable-bar-equiv}.

\begin{definition} \label{def:EL-operad}
For a fixed non-negative integer $L$, we write $\mathbb{E}_L$ for the operad of unbased spaces in which
\[ \mathbb{E}_L(r) \]
is given by the Fulton-MacPherson compactified configuration space of $r$ points in $\mathbb{R}^L$. Recall that a point in $\mathbb{E}_L(r)$ includes an $r$-tuple $y = [y_1,\dots,y_r]$ of points in $\mathbb{R}^L$ defined up to translation and positive scalar multiplication. When two or more of the points $y_i$ are equal, the point $y$ also includes information about the relative directions and distances between these `equal' points. More details on the definition of $\mathbb{E}_L$ can be found in \cite{sinha:2004}.

There is a sequence of operads of the form
\[ \mathbb{E}_0 \to \mathbb{E}_1 \to \dots \to \mathbb{E}_L \to \mathbb{E}_{L+1} \to \dots \]
where the map
\[ \mathbb{E}_L(r) \to \mathbb{E}_{L+1}(r) \]
extends points in $\mathbb{R}^L$ to $\mathbb{R}^{L+1}$ via the standard inclusion. Note that $\mathbb{E}_0$ is the trivial operad of unbased spaces which we also denote by $\mathbbm{1}$.

We write $\mathbbm{Com}$ for the commutative operad of unbased spaces with $\mathbbm{Com}(r) = *$ for all $r \geq 1$. Since $\mathbbm{Com}$ is terminal among operads of unbased spaces, there are operad maps $\mathbb{E}_L \to \mathbbm{Com}$ that are compatible with the maps in the above sequence.
\end{definition}

\begin{definition} \label{def:stable-EL-operad}
The \emph{stable little $L$-discs operad} is the reduced operad $\E_L$ of spectra given by
\[ \E_L(n) := \Sigma^\infty \mathbb{E}_L(n)_+ \]
with operad structure maps induced by those of $\mathbb{E}_L$. There is corresponding sequence of operads of spectra of the form
\begin{equation} \label{eq:EL-seq} \E_0 \to \E_1 \to \dots \to \E_L \to \E_{L+1} \to \dots \end{equation}
Note that $\E_0$ is the trivial operad of spectra which we also denote by $\one$.

Let $\Com$ be the commutative operad of spectra, given by
\[ \Com(n) := \Sigma^\infty \mathbbm{Com}(n)_+ \isom S, \text{ the sphere spectrum}. \]
There are then operad maps $\E_L \to \Com$ that are compatible with the maps in the sequence (\ref{eq:EL-seq}). The induced map
\[ \E_\infty := \colim \E_L \to \Com \]
is a weak equivalence of operads.
\end{definition}

\begin{definition}[Koszul dual]
Let $\P$ be a reduced operad of spectra. Then the \emph{Koszul dual} of $\P$ is a cofibrant replacement (in the projective model category of reduced operads of spectra) for the Spanier-Whitehead dual of the reduced bar construction on $\P$. We denote such a replacement by $\K\P$. The construction $\K$ is a contravariant functor from the category of reduced operads of spectra to itself, and for operads formed from finite spectra, there is an equivalence of operads $\K\K\P \homeq \P$ described in \cite{ching:2012}.
\end{definition}

\begin{definition} \label{def:KE}
Applying the Koszul duality functor to the sequence (\ref{eq:EL-seq}), we obtain an inverse sequence of operads of spectra
\begin{equation} \label{eq:KEL-seq} \K\E_\bullet: \quad \dots \to \K\E_2 \to \K\E_1 \to \K\E_0. \end{equation}
Associated to the pro-operad $\K\E_\bullet$ we have a comonad $\C_{\K\E_\bullet}$ given as in Definition~\ref{def:inverse-operad-comonad}.

Explicitly, for a symmetric sequence $\A$ and positive integer $k$, we have an equivalence
\begin{equation} \label{eq:KE} \C_{\K\E_\bullet}(\A)(k) \homeq \hocolim_{L} \prod_{n \geq k} \left[ \prod_{\un{n} \epi \un{k}} \Map(\K\E_L(n_1) \smsh \dots \smsh \K\E_L(n_k), \A(n)) \right]^{h\Sigma_n}. \end{equation}
\end{definition}

\begin{remark}
It is conjectured that the operad $\K\E_L$ is equivalent to an $L$-fold desuspension of $\E_L$ itself. Corresponding results on the level of homology are due to Getzler and Jones \cite{getzler/jones:1994}, and on the chain level are due to Fresse \cite{fresse:2011}.
\end{remark}

\begin{remark}
Recall from Remark~\ref{rem:holim-operad} that a $\C_{\K\E_\bullet}$-coalgebra forms, in particular, a module over the operad $\lim \K\E_L$ which is equivalent to $\K\Com$, and hence to $\der_*I_{\based}$, the operad formed by the derivatives of the identity on the category of based spaces, described in \cite{ching:2005}. We therefore have a forgetful functor
\[ \Coalg(\C_{\K\E_\bullet}) \to \Mod(\der_*I_{\based}) \]
associated to the comonad map $\C_{\K\E_\bullet} \to \C_{\K\Com}$.
\end{remark}

The main goal of the rest of this section is to prove that the comonad $\C_{\K\E_\bullet}$ acts on (a suitable choice of model of) the derivatives of a functor $F: \based \to \spectra$, that $\C_{\K\E_\bullet}$ is equivalent to the comonad $\C$ constructed from the adjunction (\ref{eq:adj}), and hence that the Taylor tower of $F$ can be recovered from the derivatives of $F$ together with their structure as a $\C_{\K\E_\bullet}$-coalgebra.

We should remark here that the construction of models for the derivatives of a functor $F: \finbased \to \spectra$ on which $\C_{\K\E_\bullet}$ acts is rather involved (see steps (1)-(5) described below), as is the proof of Theorem~\ref{thm:topsp} that $\C_{\K\E_\bullet}$ is equivalent to $\C$. The reader who wishes to skip some of the technical constructions is advised to turn Theorem~\ref{thm:equiv-topsp-poly} and following, where the main results of this section are laid out.

\subsection{Models for the derivatives of functors from based spaces to spectra}

Let $F: \finbased \to \spectra$ be a pointed simplicial functor. In this section we construct a module $\d^\bullet[F]$ over the sequence of operads $\K\E_\bullet$, such that the homotopy colimit $\d[F] = \hocolim \d^L[F]$ is equivalent to the symmetric sequence $\der_*F$ of Goodwillie derivatives of $F$. It then follows that these derivatives form a $\C_{\K\E_\bullet}$-coalgebra by the construction of Definition~\ref{def:hocolim-module-coalgebra}.

When $F$ is a polynomial functor, the term $\d^L[F]$ in the module $\d^\bullet[F]$ (which is a $\K\E_L$-module) is equivalent to the symmetric sequence of partially-stabilized cross-effects of $F$, that is,
\[ \d^L[F](n) \homeq \Sigma^{-nL}\creff_nF(S^L,\dots,S^L). \]
More generally, when $F$ is $\rho$-analytic in the sense of \cite{goodwillie:1991}, this equivalence holds for $L \geq \rho$.

Here is an outline of the construction of $\d[F]$ for a given $F: \finbased \to \spectra$:
\begin{enumerate}
  \item construct a $\Com$-comodule $\N[F]$ (and hence, by pulling back, a $\E_L$-comodule for each $L$) that models the cross-effects of $F$ evaluated at $S^0$, and show that a polynomial functor $F$ can be recovered from $\N[F]$ (Proposition~\ref{prop:F-eq});
  \item show that the `derived indecomposables' of $\N[F]$ as a $\E_L$-comodule (which can be described as a homotopy coend of the form $\N[F] \homsmsh_{\E_L} \un{\one}$) recover the partially-stabilized cross-effects of a polynomial functor $F$ (Proposition~\ref{prop:topsp-creff});
  \item construct a model for the derived indecomposables of an $\E_L$-comodule that has a $\K\E_L$-module structure (Proposition~\ref{prop:N-indec}) and thus obtain the required construction for polynomial functors (Proposition~\ref{prop:dF-polynomial});
  \item generalize the construction to analytic functors $F$ using the Taylor tower (Lemma~\ref{lem:dF-analytic});
  \item produce the necessary models for any functor by left Kan extension from representables (Proposition~\ref{prop:dF}).
\end{enumerate}

We start then with (1).

\begin{definition} \label{def:SigmainftyX-Com}
For each positive integer $n$, the construction $Y \mapsto \Sigma^\infty Y^{\smsh n}$ defines a pointed simplicial functor $\finbased \to \spectra$. Let $\widetilde{\Sigma^\infty Y^{\smsh n}}$ denote a (functorial) cofibrant replacement of this in the category $[\finbased,\spectra]$. A surjection $\un{n} \epi \un{k}$ determines a natural transformation
\[ \Sigma^\infty Y^{\smsh k} \to \Sigma^\infty Y^{\smsh n} \]
by way of the diagonal on a based space $Y$, and hence a natural map
\[ \widetilde{\Sigma^\infty Y^{\smsh k}} \to \widetilde{\Sigma^\infty Y^{\smsh n}}. \]
For each $Y$, these maps make $\widetilde{\Sigma^\infty Y^{\smsh *}}$ into a $\Com$-module, naturally in $Y \in \finbased$.
\end{definition}

\begin{definition} \label{def:F-comodule}
For $F \in [\finbased,\spectra]$ and $n \in \mathbb{N}$ we define
\[ \N[F](n) := \Nat_{Y \in \finbased}(\widetilde{\Sigma^\infty Y^{\smsh n}},FY). \]
This is the spectrum of natural transformations between two functors in $[\finbased,\spectra]$, that is, the enrichment of $[\finbased,\spectra]$ over $\spectra$. The $\Com$-module structure maps on $\widetilde{\Sigma^\infty Y^{\smsh *}}$ make the symmetric sequence $\N[F]$ into a $\Com$-comodule (naturally in $F$).
\end{definition}

\begin{lemma} \label{lem:creff}
For any $F \in [\finbased,\spectra]$, there is a natural equivalence
\[ \N[F](n) \homeq \creff_nF(S^0,\dots,S^0) \]
where the right-hand side is the \ord{n} cross-effect of $F$, evaluated at the zero-sphere in each variable.
\end{lemma}
\begin{proof}
We can write $\Sigma^\infty Y^{\smsh n}$ as the total homotopy cofibre of the $n$-cube of spectra given by
\[ J \mapsto \Sigma^\infty \Hom_{\finbased}(J_+,Y) \]
for subsets $J \subseteq \un{n}$, and whose edges are determined by the maps $J_+ \to I_+$ for $I \subseteq J$ that collapse $J-I$ to the basepoint. Applying $\Nat(-,F)$ to this total homotopy cofibre $n$-cube, and using the Yoneda Lemma we get the total homotopy fibre of the $n$-cube that defines the given cross-effect.
\end{proof}

\begin{corollary} \label{cor:N-hocolim}
The construction $F \mapsto \N[F]$ preserves homotopy colimits (defined objectwise on both sides).
\end{corollary}
\begin{proof}
The cross-effects in Lemma~\ref{lem:creff} are equivalent to the corresponding co-cross-effects. It is easy to see that taking co-cross-effects preserves homotopy colimits.
\end{proof}

The next result is unpublished work of Bill Dwyer and Charles Rezk. It says that the value $F(X)$ of a polynomial functor $F$ at a based space $X$ can be recovered from the $\Com$-comodule $\N(F)$ by way of a homotopy coend with the $\Com$-module $\Sigma^\infty X^{\smsh *}$.

\begin{prop} \label{prop:F-eq}
Let $F: \finbased \to \spectra$ be a polynomial functor and let $\tilde{\N}[F]$ denote a $\Sigma$-cofibrant replacement for the $\Com$-comodule $\N[F]$. Then the canonical evaluation map
\[ \epsilon: \tilde{\N}[F] \homsmsh_{\Com} \widetilde{\Sigma^\infty X^{\smsh *}} \longrightarrow F(X) \]
is a weak equivalence for each $X \in \finbased$.
\end{prop}
\begin{proof}
Since both sides preserve objectwise (co)fibration sequences in $F$, we can use the Taylor tower to reduce to the case that $F$ is $n$-homogeneous, that is, of the form
\[ F(X) = (E \smsh X^{\smsh n})_{h\Sigma_n} \]
for some spectrum $E$ with $\Sigma_n$-action. Both sides also commute with the homotopy orbit construction, the left-hand side by Corollary~\ref{cor:N-hocolim}, and with smashing with a fixed spectrum so we can reduce to the case that
\[ F(X) = \Sigma^\infty X^{\smsh n}. \]
For this case, we first claim that there is an equivalence of $\Com$-comodules
\begin{equation} \label{eq:creff-X^*} \phi: \un{\Com}(n,*) \weq \Nat_{Y \in \finbased}(\widetilde{\Sigma^\infty Y^{\smsh *}}, \Sigma^\infty Y^{\smsh n}). \end{equation}
The map $\phi$ is constructed as follows. Notice that $\un{\Com}(n,k) = \Wdge_{\un{n} \epi \un{k}} S$. Then for each surjection $\alpha: \un{n} \epi \un{k}$, we have a corresponding natural transformation
\[ \dgTEXTARROWLENGTH=3em \widetilde{\Sigma^\infty Y^{\smsh k}} \arrow{e,t}{\sim} \Sigma^\infty Y^{\smsh k} \arrow{e,t}{\Delta_\alpha} \Sigma^\infty Y^{\smsh n} \]
where $\Delta_{\alpha}$ is the diagonal map $X^{\smsh k} \to X^{\smsh n}$ associated to $\alpha$. It is clear from the definition that $\phi$ is a morphism of $\Com$-comodules, where $\un{\Com}(n,*)$ has the free comodule structure built from the composition maps in the category $\un{\Com}$.

To see that $\phi$ is an equivalence, we consider the following commutative diagram
\[ \begin{diagram}
  \node{\Wdge_{\un{n} \epi \un{k}} S} \arrow{s,l}{\sim} \arrow{e,t}{\phi}
    \node{\Nat_{Y \in \finbased}(\widetilde{\Sigma^\infty Y^{\smsh k}}, \Sigma^\infty Y^{\smsh n})} \arrow{s,r}{\sim} \\
  \node{\thofib_{J \subseteq \un{k}} \Sigma^\infty J^n_+} \arrow{e,t}{\sim}
    \node{\thofib_{J \subseteq \un{k}} \Nat_{Y \in \finbased}(\Sigma^\infty \Hom_{\based}(J_+,Y), \Sigma^\infty Y^{\smsh n})}
\end{diagram} \]
The right-hand vertical map is the equivalence described in the proof of Lemma~\ref{lem:creff}. The left-hand vertical map is given as follows: a surjection $\un{n} \epi \un{k}$ determines an element of $\un{k}^n_+$ that maps to the basepoint in $J^n_+$ for any proper subset $J \subset \un{k}$. There is therefore an induced map into the total homotopy fibre. It is easy to check that this map is an equivalence. The bottom horizontal map is induced by a levelwise equivalence of $k$-cubes by the enriched Yoneda Lemma, so is an equivalence. It follows that $\phi$ is an equivalence.

Now consider the following diagram
\begin{equation} \label{eq:eval-eq} \begin{diagram}
  \node{\un{\Com}(n,*) \homsmsh_{\Com} \widetilde{\Sigma^\infty X^{\smsh *}}} \arrow{s,lr}{\phi}{\sim} \arrow{e}
    \node{\widetilde{\Sigma^\infty X^{\smsh n}}} \arrow{s,r}{\sim} \\
  \node{\Nat_{Y \in \finbased}(\widetilde{\Sigma^\infty Y^{\smsh *}}, \Sigma^\infty Y^{\smsh n}) \homsmsh_{\Com} \widetilde{\Sigma^\infty X^{\smsh *}}} \arrow{e,t}{\epsilon}
    \node{\Sigma^\infty X^{\smsh n}}
\end{diagram} \end{equation}
where the bottom map is the evaluation map that we want to show is an equivalence, the left vertical map is induced by $\phi$ and so is an equivalence, the right-hand map is the fixed equivalence used in the construction of $\phi$, and the top horizontal map is induced by the augmentation of the simplicial object underlying $\un{\Com}(n,*) \homsmsh_{\Com} \widetilde{\Sigma^\infty X^{\smsh *}}$, that is built from the composition maps
\[ \Wdge_{n_0,\dots,n_r} \un{\Com}(n,n_0) \smsh_{\Sigma_{n_0}} \un{\Com}(n_0,n_1) \smsh_{\Sigma_{n_1}} \dots \smsh_{\Sigma_{n_{r-1}}} \un{\Com}(n_{r-1},n_r) \smsh_{\Sigma_{n_r}} \widetilde{\Sigma^\infty X^{\smsh n_r}} \to \widetilde{\Sigma^\infty X^{\smsh n}}. \]
This simplicial object has extra degeneracies and so the augmentation map is an equivalence. It can be checked that the diagram (\ref{eq:eval-eq}) is commutative and so it follows that the evaluation map is an equivalence as required.
\end{proof}

\begin{remark}
The equivalence of Proposition~\ref{prop:F-eq} is part of a classification of reduced polynomial functors $\finbased \to \spectra$ in terms of $\Com$-comodules (also from unpublished work of Dwyer and Rezk) to which we return in Theorem~\ref{thm:classification}.
\end{remark}

We now turn to part (2) of our construction of $\d[F]$. We start by taking derivatives of each side of the equivalence in Proposition~\ref{prop:F-eq} to get models for the derivatives of a polynomial functor $F$. To state this result we use the following definition.

\begin{definition} \label{def:one}
Let $\un{\one}$ be the bisymmetric sequence given by
\[ \un{\one}(I,J) := \Wdge_{I \isom J} S \]
where $S$ is the sphere spectrum and the wedge sum is taken over the set of bijections from $I$ to $J$. Thus $\un{\one}(I,J)$ is trivial if $|I| \neq |J|$. Notice also that this notation is consistent with that of Definition~\ref{def:prop} applied to the trivial operad $\one$.
\end{definition}

\begin{corollary} \label{cor:dF-eq}
Let $F: \finbased \to \spectra$ be a polynomial functor. Then there are natural $\Sigma_n$-equivariant equivalences
\[ \N[F] \homsmsh_{\Com} \un{\one}(*,n) \weq \der_n(F) \]
where, for each $n$, the symmetric sequence $\un{\one}(*,n)$ is given a trivial $\Com$-module structure.
\end{corollary}
\begin{proof}
This follows from Proposition~\ref{prop:F-eq} by applying $\der_*$ to each side, using the fact that taking derivatives commutes with arbitrary homotopy colimits (for $\spectra$-valued functors) and that there are equivalences of $\Com$-modules:
\[ \der_n(\Sigma^\infty X^{\smsh *}) \homeq \un{\one}(*,n). \]
\end{proof}

\begin{notation}
Let $\P$ be an operad, $\N$ a $\P$-comodule, and $\un{\B}$ a bisymmetric sequence that has a $\P$-module structure on its first variable. We then write $\N \homsmsh_{\P} \un{\B}$ for the symmetric sequence given by
\[ (\N \homsmsh_{\P} \un{\B})(n) := \N(*) \homsmsh_{\P} \un{\B}(*,n). \]
\end{notation}

\begin{remark}
For a $\P$-comodule $\N$, the symmetric sequence
\[ \N \homsmsh_{\P} \un{\one} \]
is a model for the `derived indecomposables' of $\N$, that is, for the left derived functor of the adjoint to the trivial comodule structure functor $\symseq \to \Comod(\Com)$. We can rephrase Corollary~\ref{cor:dF-eq} as saying that, for a polynomial functor $F: \finbased \to \spectra$, the derived indecomposables of the $\Com$-comodule $\N[F]$ recover the derivatives of $F$.
\end{remark}

Now recall from Definition~\ref{def:EL-operad} that we have a sequence of operads
\[ \E_0 \to \E_1 \to \E_2 \to \dots \]
and an equivalence of operads $\hocolim_L \E_L \weq \Com$. It can be shown using Definition~\ref{def:homsmsh} (and the fact that homotopy colimits commute with each other) that there is a corresponding equivalence
\begin{equation} \label{eq:hocolimEL} \hocolim_L \N[F] \homsmsh_{\E_L} \un{\one} \weq \N[F] \homsmsh_{\Com} \un{\one} \homeq \der_*F \end{equation}
where $\N[F]$ is given an $\E_L$-module structure by pulling back the $\Com$-module structure along the operad map $\E_L \to \Com$. This expresses the derivatives of $F$ as a homotopy colimit of the derived indecomposables of $\N[F]$ as an $\E_L$-comodule, as $L \to \infty$.

We now prove that these derived indecomposables recover the partially-stabilized cross-effects of $F$, thus completing part (2) of our construction.

\begin{proposition} \label{prop:topsp-creff}
Let $F$ be a polynomial functor. Then there are equivalences of symmetric sequences
\[ \N[F] \homsmsh_{\E_L} \un{\one} \homeq \Sigma^{-*L}\creff_*F(S^L,\dots,S^L) \]
that are natural in $F$ and make the following diagrams commute in the homotopy category of symmetric sequences:
\begin{equation} \label{eq:topsp-creff} \begin{diagram}
  \node{\N[F] \homsmsh_{\E_L} \un{\one}} \arrow{s} \arrow{e,t}{\sim} \node{\Sigma^{-*L}\creff_*F(S^L,\dots,S^L)} \arrow{s} \\
  \node{\N[F] \homsmsh_{\E_{L+1}} \un{\one}} \arrow{e,t}{\sim} \node{\Sigma^{-*(L+1)} \creff_*F(S^{L+1},\dots,S^{L+1})}
\end{diagram} \end{equation}
Here the left-hand vertical map is induced by the operad map $\E_L \to \E_{L+1}$, and the right-hand map is that of (\ref{eq:creff-map}).
\end{proposition}
\begin{proof}
Our strategy is to use Proposition~\ref{prop:F-eq} to reduce to the case that $F$ is one of the functors $\Sigma^\infty X^{\smsh n}$. Proving the result in those cases seems to require some significant geometric input about the relationship of the operad $\E_L$ to the sphere $S^L$. For us, this input is contained in the construction of the equivalence in Lemma~\ref{lem:unstable-bar-equiv} below.

By Proposition~\ref{prop:F-eq}, we can write
\[ F(X) \homeq \N[F] \homsmsh_{\Com} \widetilde{\Sigma^\infty X^{\smsh *}}. \]
Since taking cross-effects commutes with homotopy colimits, this gives us
\begin{equation} \label{eq:creff1} \Sigma^{-kL} \creff_kF(S^L,\dots,S^L) \homeq \N[F] \homsmsh_{\Com} \Sigma^{-kL}\creff_k(\Sigma^\infty X^{\smsh *})(S^L,\dots,S^L). \end{equation}
The cross-effects of the homogeneous functors $\Sigma^\infty X^{\smsh *}$ are easily calculated: we have
\[ \creff_k(\Sigma^\infty X^{\smsh n})(X_1, \dots, X_k) \homeq \Wdge_{\alpha: \un{n} \epi \un{k}} \Sigma^\infty X_1^{n_1} \smsh \dots \smsh \Sigma^\infty X_k^{n_k}. \]
where, as usual, we write $n_i := |\alpha^{-1}(i)|$ for a surjection $\alpha: \un{n} \epi \un{k}$.

This model for $\creff_k(\Sigma^\infty X^{\smsh n})$ is pointed simplicial in each variable with tensoring maps
\[ S^k \smsh \Wdge_{\un{n} \epi \un{k}} \Sigma^\infty X_1^{n_1} \smsh \dots \smsh \Sigma^\infty X_k^{n_k} \to \Wdge_{\un{n} \epi \un{k}} \Sigma^\infty (\Sigma X_1)^{\smsh n_1} \smsh \dots \smsh \Sigma^\infty (\Sigma X_k)^{\smsh n_k} \]
induced by the diagonal maps $S^k = (S^1)^{\smsh k} \arrow{e,t}{\Delta_\alpha} (S^1)^{\smsh n}$ associated to the surjections $\un{n} \epi \un{k}$. By Lemma~\ref{lem:creff-maps}, the map (\ref{eq:creff-map}) for the functor $\Sigma^\infty X^{\smsh n}$ is modeled by the map
\[ \Wdge_{\un{n} \epi \un{k}} \Sigma^{-kL} \Sigma^\infty (S^L)^{\smsh n} \to \Wdge_{\un{n} \epi \un{k}} \Sigma^{-k(L+1)} \Sigma^\infty (S^{L+1})^{\smsh n} \]
induced also by the diagonal maps $\Delta_{\alpha}$. For each fixed $k$, we therefore have a commutative diagram in the homotopy category of $\Com$-modules (with variable $n$) of the form
\begin{equation} \label{eq:creff2}
\begin{diagram}
  \node{\Wdge_{\un{n} \epi \un{k}} \Sigma^{-kL} (\Sigma^\infty S^L)^{\smsh n}} \arrow{e,t}{\sim} \arrow{s}
     \node{\Sigma^{-kL} \creff_k(\Sigma^\infty X^{\smsh n})(S^L,\dots,S^L)} \arrow{s} \\
  \node{\Wdge_{\un{n} \epi \un{k}} \Sigma^{-k(L+1)} (\Sigma^\infty S^{L+1})^{\smsh n}} \arrow{e,t}{\sim}
    \node{\Sigma^{-k(L+1)} \creff_k(\Sigma^\infty X^{\smsh n})(S^{L+1},\dots,S^{L+1}).}
\end{diagram}
\end{equation}
The $\Com$-module structure on the left-hand side terms in \ref{eq:creff2} is given by composition of surjections combined with the diagonal maps for $S^L$ and $S^{L+1}$.

The next step is to construct a diagram (in the homotopy category of $\Com$-modules) of the form
\begin{equation} \label{eq:creff3}
\begin{diagram}
    \node{\un{\Com}(n,*) \homsmsh_{\E_L} \un{\one}(*,k)} \arrow{e,t}{\sim} \arrow{s}
        \node{\Wdge_{\un{n} \epi \un{k}} \Sigma^{-kL} (\Sigma^\infty S^L)^{\smsh n}} \arrow{s} \\
    \node{\un{\Com}(n,*) \homsmsh_{\E_{L+1}} \un{\one}(*,k)} \arrow{e,t}{\sim}
        \node{\Wdge_{\un{n} \epi \un{k}} \Sigma^{-k(L+1)} (\Sigma^\infty S^{L+1})^{\smsh n}}
\end{diagram}
\end{equation}
Combining (\ref{eq:creff3}) with (\ref{eq:creff2}) and applying $\N[F] \homsmsh_{\Com} -$ to the resulting diagram, we get, by (\ref{eq:creff1}), the required diagram (\ref{eq:topsp-creff}). Note that there are equivalences $\N[\Sigma^\infty X^{\smsh n}] \homeq \un{\Com}(n,*)$ (which preserve both the $\Com$-comodule structure and the $\Com$-module structure on the variable $n$) so the construction of diagram (\ref{eq:creff3}) amounts to proving the Proposition in the case $F = \Sigma^\infty X^{\smsh n}$.

To get (\ref{eq:creff3}) we first note that there are natural isomorphisms of $\Com$-modules
\begin{equation} \label{eq:creff4} \un{\Com}(n,*) \homsmsh_{\E_L} \un{\one}(*,k) \isom \Wdge_{\un{n} \epi \un{k}} \Smsh_{i = 1}^{k} \B(\one,\E_L,\Com)(n_i). \end{equation}
where the right-hand side is the two-sided operadic bar construction, as described in \cite{ching:2005}. The above isomorphism can be identified by comparing the structure of the simplicial spectra underlying the homotopy coend and the bar construction.

The most significant part of the proof now concerns the construction of equivalences of $\Com$-modules
\begin{equation} \label{eq:EL-eq} \B(\one,\E_L,\Com) \arrow{e,t}{\sim} \Sigma^{-L} (\Sigma^\infty S^L)^{\smsh *}. \end{equation}
We make this construction at the unstable level. For an operad $\mathbb{P}$ of unbased spaces, let us write $\mathbb{P}_+$ for the corresponding operad of based spaces given by adding a disjoint basepoint to each term. The following lemma is the key calculation underlying the proof of Proposition~\ref{prop:topsp-creff}.

\begin{lemma} \label{lem:unstable-bar-equiv}
There is a weak equivalences of $\mathbbm{Com}_+$-modules (in the category of based spaces)
\[ f_L: \B(\mathbbm{1}_+,{\mathbb{E}_L}_+,\mathbbm{Com}_+) \arrow{e,t}{\sim} S^{L(*-1)} \]
where we identify $S^{L(n-1)}$ with the one-point compactification of $(\mathbb{R}^L)^n/\mathbb{R}^L$ (the space of $n$-tuples $[y_1,\dots,y_n]$ of vectors in $\mathbb{R}^L$ defined up to translation) and give the collection $S^{L(*-1)}$ a $\mathbbm{Com}_+$-module structure via the diagonal map $\mathbb{R}^L \to \mathbb{R}^L \times \mathbb{R}^L$. These maps make the following diagrams commute:
\[ \begin{diagram}
  \node{\B(\mathbbm{1}_+,{\mathbb{E}_L}_+,\mathbbm{Com}_+)} \arrow{e,t}{f_L} \arrow{s}
    \node{S^{L(*-1)}} \arrow{s} \\
  \node{\B(\mathbbm{1}_+,{\mathbb{E}_{L+1}}_+,\mathbbm{Com}_+)} \arrow{e,t}{f_{L+1}}
    \node{S^{(L+1)(*-1)}}
\end{diagram} \]
where the right-hand vertical map is induced by the fixed inclusion $\mathbb{R}^L \to \mathbb{R}^{L+1}$ which induces the operad map $\mathbb{E}_L \to \mathbb{E}_{L+1}$.
\end{lemma}
\begin{proof}
Write $\breve{\mathbb{E}}_L$ for the symmetric sequence (of unbased spaces) given by
\[ \breve{\mathbb{E}}_L(r) := \begin{cases}
  \mathbb{E}_L(r) & \text{if $r \geq 2$}; \\
  \; \; \emptyset & \text{if $r = 1$}.
\end{cases} \]
Note that $\breve{\mathbb{E}}_L$ has a right $\mathbb{E}_L$-module structure coming from the operad structure on $\mathbb{E}_L$ and that there is a homotopy cofibre sequence
\[ \breve{\mathbb{E}}_L{}_+ \to {\mathbb{E}_L}_+ \to \mathbbm{1}_+ \]
of right ${\mathbb{E}_L}_+$-modules. Applying the bar construction, we obtain a homotopy cofibre sequence
\begin{equation} \label{eq:barseq1} \B(\breve{\mathbb{E}}_L{}_+, {\mathbb{E}_L}_+, \mathbbm{Com}_+) \to \B({\mathbb{E}_L}_+, {\mathbb{E}_L}_+, \mathbbm{Com}_+) \to \B(\mathbbm{1}_+, {\mathbb{E}_L}_+, \mathbbm{Com}_+) \end{equation}
of right $\mathbbm{Com}_+$-modules.

Now let $S^{L(n-1)-1}$ denote the sphere in the vector space $(\mathbb{R}^L)^n/\mathbb{R}^L$, or equivalently, the space of $n$-tuples $[y_1,\dots,y_n]$ of vectors in $\mathbb{R}^L$, not all equal, defined up to translation and positive scalar multiplication. Notice that $S^{L(n-1)}$ is the unreduced suspension of $S^{L(n-1)-1}$ so that we have a homotopy cofibre sequence of right $\mathbbm{Com}_+$-modules
\begin{equation} \label{eq:barseq2} S^{L(*-1)-1}_+ \to \mathbbm{Com}_+ \to S^{L(*-1)}. \end{equation}
The second map here sends the non-basepoint of $\mathbbm{Com}_+(n)$ to the point in $S^{L(n-1)}$ that represents the $n$-tuple $[0,\dots,0]$.

Our goal now is to construct an equivalence between the sequences (\ref{eq:barseq1}) and (\ref{eq:barseq2}) of which the required equivalence $f_L$ will be part.

Notice first that
\[ \B(\breve{\mathbb{E}}_L{}_+, {\mathbb{E}_L}_+, \mathbbm{Com}_+) = \B(\breve{\mathbb{E}}_L, {\mathbb{E}_L}, \mathbbm{Com})_+. \]
We now claim that $\breve{\mathbb{E}}_L$ is cofibrant as an $\mathbb{E}_L$-module (in the projective model structure on right modules over an operad in the category of compactly generated spaces). To see this, suppose given a diagram of $\mathbb{E}_L$-modules
\begin{equation} \label{eq:lift1} \begin{diagram}
  \node[2]{\mathbb{M}} \arrow{s,r,A}{\sim} \\
  \node{\breve{\mathbb{E}}_L} \arrow{e} \node{\mathbb{M}'}
\end{diagram} \end{equation}
where the right-hand vertical map is a trivial fibration in the projective model structure on modules over the operad $\mathbb{E}_L$ (i.e. each map $\mathbb{M}(n) \to \mathbb{M}'(n)$ is a Serre fibration and weak homotopy equivalence). We recursively construct a lifting $l: \breve{\mathbb{E}}_L \to \mathbb{M}$ as follows. Suppose we have constructed maps
\[ l_r: \breve{\mathbb{E}}_L(r) \to \mathbb{M}(r), \]
for each $r < k$, that commute with the relevant $\mathbb{E}_L$-module structure maps. Together the $l_r$ determine the top horizontal map in a commutative diagram
\begin{equation} \label{eq:lift} \begin{diagram}
  \node{\partial \breve{\mathbb{E}}_L(k)} \arrow{s,V} \arrow{e} \node{\mathbb{M}(k)} \arrow{s,r,A}{\sim} \\
  \node{\breve{\mathbb{E}}_L(k)} \arrow{e} \node{\mathbb{M}'(k)}
\end{diagram} \end{equation}
where $\partial \breve{\mathbb{E}}_L(k)$ denotes the boundary of the manifold with corners $\breve{\mathbb{E}}_L(k)$. (Recall that this boundary is identical to the union of the images of the non-trivial module structure maps
\[ \breve{\mathbb{E}}_L(r) \times \mathbb{E}_L(k_1) \times \dots \times \mathbb{E}_L(k_r) \to \breve{\mathbb{E}}_L(k).) \]
Since the inclusion $\partial \breve{\mathbb{E}}_L(k) \to \breve{\mathbb{E}}_L(k)$ is a relative cell complex, we can choose a lift
\[ l_k: \breve{\mathbb{E}}_L(k) \to \mathbb{M}(k) \]
for the diagram (\ref{eq:lift}) which continues the recursion. Together the maps $l_k$ determine the required lift for the diagram (\ref{eq:lift1}). It follows that $\breve{\mathbb{E}}_L$ is a cofibrant $\mathbb{E}_L$-module, as claimed.

It then follows that there is an equivalence of $\mathbbm{Com}$-modules
\[ \gamma_L: \B(\breve{\mathbb{E}}_L, {\mathbb{E}_L}, \mathbbm{Com}) \weq \breve{\mathbb{E}}_L \circ_{\mathbb{E}_L} \mathbbm{Com}. \]
We claim that there is then an isomorphism of $\mathbbm{Com}$-modules
\begin{equation} \label{eq:bar-eq1} \breve{\mathbb{E}}_L \circ_{\mathbb{E}_L} \mathbbm{Com} \isom S^{L(*-1)-1}. \end{equation}
To see this we first define a map of $\mathbbm{Com}$-modules
\[ g: \breve{\mathbb{E}}_L \circ \mathbbm{Com} \to S^{L(*-1)-1}. \]
The \ord{$n$} term on the left-hand side can be written as
\[ \coprod_{k \geq 2} \left[ \coprod_{\un{n} \epi \un{k}} \mathbb{E}_L(k) \right]_{\Sigma_k}. \]
Given an integer $k$ and surjection $\alpha: \un{n} \epi \un{k}$, we can define a map
\[ g_\alpha: \breve{\mathbb{E}}_L(k) \to S^{L(n-1)-1} \]
as follows. A point $x$ in $\breve{\mathbb{E}}_L(k)$ consists of a $k$-tuple $[x_1,\dots,x_k]$ of vectors in $\mathbb{R}^L$ defined up to translation and scalar multiplication, together with extra information if any of the $x_i$ are equal. We set
\[ g_\alpha(x) := [x_{\alpha(1)}, \dots, x_{\alpha(n)}] \in S^{L(n-1)-1}. \]
Together the maps $g_\alpha$ determine the required map $g$. We now show that $g$ induces an isomorphism of the form (\ref{eq:bar-eq1}).

First note that the two composites in the following diagram are equal:
\[ \breve{\mathbb{E}}_L \circ \mathbb{E}_L \circ \mathbbm{Com} \rightrightarrows \breve{\mathbb{E}}_L \circ \mathbbm{Com} \to S^{L(*-1)-1}, \]
so that $g$ determines a map
\[ \bar{g}: \breve{\mathbb{E}}_L \circ_{\mathbb{E}_L} \mathbbm{Com} \to S^{L(*-1)-1}. \]
To see that each $\bar{g}_n$ is a bijection, we construct its inverse which we denote $\bar{h}_n$.

A point $y \in S^{L(n-1)-1}$ is represented by an $n$-tuple $[y_1,\dots,y_n]$ of vectors in $\mathbb{R}^L$, not all equal, defined up to translation and positive scalar multiplication. We can write $y_i = x_{\alpha(i)}$ for a $k$-tuple $x = [x_1,\dots,x_k]$ of \emph{distinct} vectors, also defined up to translation and positive scalar multiplication, where $k \geq 2$ is a positive integer and $\alpha: \un{n} \epi \un{k}$ is a surjection. Note that $k$ and $\alpha$ are uniquely determined, up to action of $\Sigma_k$, by the condition that all $x_i$ are distinct. This data determines a point in $\breve{\mathbb{E}}_L \circ \mathbbm{Com}$ (where, in particular, the $k$-tuple $x$ represents a point in the \emph{interior} of the space $\breve{\mathbbm{E}}_L(k)$), and hence a point in $(\breve{\mathbb{E}}_L \circ_{\mathbb{E}_L} \mathbbm{Com})(n)$ which we take to be $\bar{h}_n(y)$.

It is now relatively easy to see that $\bar{h}_n$ is inverse to $\bar{g}_n$ and so that $\bar{g}_n$ is a continuous bijection. The space $(\breve{\mathbb{E}}_L \circ_{\mathbb{E}_L} \mathbbm{Com})(n)$ is compact (since each $\mathbb{E}_L(k)$ is for $2 \leq k \leq n$) and $S^{L(n-1)-1}$ is Hausdorff. Therefore $\bar{g}_n$ is a homeomorphism, and we obtain the required isomorphism (\ref{eq:bar-eq1}). We then have a commutative diagram of $\mathbbm{Com}_+$-modules
\[ \begin{diagram}
  \node{\B(\breve{\mathbb{E}}_L{}_+, {\mathbb{E}_L}_+, \mathbbm{Com}_+)} \arrow{s,lr}{\gamma_L}{\sim} \arrow{e}
    \node{\B({\mathbb{E}_L}_+, {\mathbb{E}_L}_+, \mathbbm{Com}_+)} \arrow{s,r}{\sim} \\
  \node{S^{L(*-1)-1}_+} \arrow{e}
    \node{\mathbbm{Com}_+}
\end{diagram} \]
where the right-hand map is the standard bar resolution map. We therefore get an induced equivalence between the homotopy cofibres which provides the required equivalence
\[ f_L: \B(\mathbbm{1}_+,{\mathbb{E}_L}_+,\mathbbm{Com}_+) \weq S^{L(*-1)}. \]
It is easy to see that the maps $g_L$ commute with the maps induced by the inclusion $\mathbb{R}^L \to \mathbb{R}^{L+1}$, and hence the maps $f_L$ do too. This completes the proof of Lemma~\ref{lem:unstable-bar-equiv}.
\end{proof}

\begin{remark}
Low-dimensional calculations suggest that the based space $\B(\mathbbm{1}_+,\mathbb{E}_L{}_+,\mathbbm{Com}_+)(n)$ is actually homeomorphic to $S^{L(n-1)}$ though we do not need that claim here.
\end{remark}

{\it Proof of \ref{prop:topsp-creff} (continued)}.
We now get the required stable equivalences (\ref{eq:EL-eq}) from the maps $f_L$ of Lemma~\ref{lem:unstable-bar-equiv} by the composite
\[ \B(\one,\E_L,\Com) \isom \Sigma^\infty \B(\mathbbm{1}_+,\mathbb{E}_L{}_+,\mathbbm{Com}_+) \weq \Sigma^\infty S^{L(*-1)} \weq \Sigma^{-L} (\Sigma^\infty S^L)^{\smsh *}. \]
The final equivalence is adjoint to the map
\[ S^{L(k-1)} \smsh S^L \to (S^L)^{\smsh k} \]
that is the one-point compactification of the map of $\mathbbm{Com}$-modules (in the $k$ variables)
\[ (\mathbb{R}^L)^k/\mathbb{R}^L \times \mathbb{R}^L \to (\mathbb{R}^L)^k \]
given by
\[ ([y_1,\dots,y_k], u) \mapsto (y_1-m+u, \dots, y_k-m+u) \]
where $m$ is the coordinate-wise `minimum' of the vectors $y_1,\dots,y_k$. (That is, $m_i := \min_{1 \leq j \leq k} (y_j)_i$.)

Using (\ref{eq:creff4}) we then get the required diagram (\ref{eq:creff3}). This completes the proof of Proposition~\ref{prop:topsp-creff}.
\end{proof}

We now turn to part (3) of the construction of $\d[F]$: showing that the derived indecomposables $\N[F] \homsmsh_{\E_L} \un{\one}$ can be given the structure of a $\K\E_L$-module. Behind this construction is a form of derived Koszul duality for comodules and modules over an operad $\P$ and its Koszul dual $\K\P$. For any operad $\P$ of spectra and $\P$-comodule $\N$, we construct a model for the symmetric sequence
\[ \N \homsmsh_{\P} \un{\one} \]
that is a module over the Koszul dual operad $\K\P$. At the heart of our construction is a particular bisymmetric sequence $\un{\B}_\P$, equivalent to $\un{\one}$, that has a $\P$-module structure in one variable, and a $\K\P$-module structure in the other. We define this now.

\begin{definition} \label{def:B}
For an operad $\P$ of spectra, let $\un{\B}_\P$ be the bisymmetric sequence given on finite nonempty sets $I,J$ by
\[ \un{\B}_\P(I,J) := \prod_{I \epi J} \Smsh_{j \in J} \B(\one,\P,\P)(I_j) \]
where the product is indexed by the set of surjections from $I$ to $J$. Recall from \cite{ching:2005} that the symmetric sequence $\B(\one,\P,\P)$ has both a right $\P$-module structure and a left $\B\P$-comodule structure (where $\B\P = \B(\one,\P,\one)$ is the cooperad given by the reduced bar construction on $\P$). We use these structures to produce the required operad actions on $\un{\B}_\P$.

We make $\un{\B}_\P$ into a right $\P$-module in its first variable as follows. For each surjection $\alpha: I \epi I'$ we have to describe a map
\[ r_\alpha: \un{\B}_\P(I',J) \smsh \Smsh_{i' \in I'} \P(I_{i'}) \to \un{\B}_\P(I,J). \]
The target here is a product indexed by surjections $\beta: I \epi J$ and we define each component individually. There are two cases. If $\beta$ does not factor as $\gamma \alpha$ for some surjection $\gamma: I' \epi J$ then we take the relevant component of $r_\alpha$ to be the trivial map. If $\beta$ does factor in this way then $\gamma$ is uniquely determined. We then build the relevant component of $r_\alpha$ from the projection map
\[ p_{\gamma}: \un{\B}_\P(I',J) \to \Smsh_{j \in J} \B(\one,\P,\P)(I'_j) \]
and by taking the smash product, over $j \in J$, of the maps
\[ \B(\one,\P,\P)(I'_j) \smsh \Smsh_{i' \in I'_j} \P(I_{i'}) \to \B(\one,\P,\P)(I_j) \]
associated to the right $\P$-action on $\B(\one,\P,\P)$ and the surjections $I_j \epi I'_j$ given by restricting $\alpha$.

We also make $\un{\B}_\P$ into a right $\K\P$-module in its second variable as follows. For each surjection $\alpha: J \epi J'$ we have to describe a map
\begin{equation} \label{eq:s-alpha} s_\alpha: \un{\B}_\P(I,J') \smsh \Smsh_{j' \in J'} \K\P(J_{j'}) \to \un{\B}_\P(I,J). \end{equation}
The target is still indexed by surjections $\beta: I \epi J$ and again we define each component individually. Let $\gamma = \alpha \beta: I \epi J'$. Then we build the relevant component of $s_\alpha$ from the projection
\[ p_{\gamma}: \un{\B}_\P(I,J') \to \Smsh_{j' \in J'} \B(\one,\P,\P)(I_{j'}) \]
and by taking the smash product, over $j' \in J'$, of maps
\[ \B(\one,\P,\P)(I_{j'}) \smsh \K\P(J_{j'}) \to \Smsh_{j \in J_{j'}} \B(\one,\P,\P)(I_j) \]
associated to the \emph{left} coaction by the cooperad $\B\P$ on $\B(\one,\P,\P)$. Recall that the Koszul dual $\K\P$ is the Spanier-Whitehead dual of $\B\P$ and the above map is adjoint to the coaction map associated to the surjection $I_{j'} \epi J_{j'}$ given by restricting $\beta$.
\end{definition}

\begin{lemma} \label{lem:B}
The constructions of Definition~\ref{def:B} make $\un{\B}_\P$ into a spectrally-enriched functor $\un{\P}^{op} \times \un{\K\P}^{op} \to \spectra$, that is, a bisymmetric sequence with commuting right module structures for the operads $\P$ and $\K\P$ on the first and second variables respectively.
\end{lemma}
\begin{proof}
We have to check associativity and unit conditions for each of the actions, and then a commutativity condition between them. For the $\P$-action on the first variable, the unit condition is that the map $r_\alpha$ associated to the identity $\alpha: I \epi I$ be the identity on $\un{\B}_\P(I,J)$, which it is.

The associativity condition concerns the maps $r_\alpha$, $r_{\alpha'}$ and $r_{\alpha' \alpha}$ for surjections $\alpha: I \epi I'$ and $\alpha': I' \epi I''$. The key point is then the following: a surjection $\beta: I \epi J$ factors via $\alpha' \alpha$ if and only if $\beta$ factors as $\gamma \alpha$, \emph{and} then $\gamma$ factors via $\alpha'$. If these factorizations do not exist, then the components of $r_{\alpha'\alpha}$ and of $r_{\alpha'}r_\alpha$ corresponding to $\beta$ are both trivial. If the factorizations do exist, then those components are equal by the associativity of the right $\P$-action on $\B(\one,\P,\P)$.

For the $\K\P$-action on the second variable, the corresponding checks are similar, but they are easier since there are no separate cases in the construction of the maps $s_\alpha$. The associativity condition follows from the coassociativity of the coaction of $\B\P$ on $\B(\one,\P,\P)$.

To see that the two actions commute, suppose we have surjections $\alpha: I \epi I'$ and $\delta: J \epi J'$. We have to show that the following square commutes:
\begin{equation} \label{eq:commute} \begin{diagram}
  \node{\Smsh_{J' \in J'} \K\P(J_{j'}) \smsh \un{\B}_\P(I',J') \smsh \Smsh_{i' \in I'}^{\phantom{I'}} \P(I_{i'})} \arrow{e,t}{r_\alpha} \arrow{s,l}{s_\delta}
    \node{\Smsh_{j' \in J'}^{\phantom{I'}} \K\P(J_{j'}) \smsh \un{\B}_\P(I,J')} \arrow{s,r}{s_{\delta}} \\
  \node{\un{\B}_\P(I',J) \smsh \Smsh_{i' \in I'}^{\phantom{I'}} \P(I_{i'})} \arrow{e,t}{r_\alpha}
    \node{\un{\B}_\P(I,J)^{\phantom{I'}}}
\end{diagram} \end{equation}
Since the target is a product, we can check commutativity by considering the components of the composite maps corresponding to each surjection $\beta: I \epi J$ in turn. There are three cases.

If $\beta = \gamma \alpha$ for some $\gamma: I' \epi J$ then also $\delta \beta = (\delta \gamma) \alpha$. In this case, the relevant components of the horizontal maps in (\ref{eq:commute}) are given by the $\P$-action on $\B(\one,\P,\P)$. The required commutativity then follow from the commutativity of the corresponding diagram
\[ \begin{diagram}
  \node{\B(\one,\P,\P)(I'_{j'}) \smsh \Smsh_{i' \in I'_{j'}}^{\phantom{I'_{j'}}} \P(I_{i'})} \arrow{e} \arrow{s}
    \node{\B(\one,\P,\P)(I_{j'})}_{\phantom{I'}}^{\phantom{I'}} \arrow{s} \\
  \node{\B\P(J_{j'}) \smsh \Smsh_{j \in J_{j'}}^{\phantom{I'_{j'}}} \B(\one,\P,\P)(I'_j) \smsh \Smsh_{i' \in I'_{j'}}^{\phantom{I'_{j'}}} \P(I_{i'})} \arrow{e}
    \node{\B\P(J_{j'}) \smsh \Smsh_{j \in J_{j'}}^{\phantom{I'_{j'}}} \B(\one,\P,\P)(I_j)}
\end{diagram} \]
where the vertical maps are given by the $\B\P$-coaction maps on $\B(\one,\P,\P)$ associated to the restrictions of $\beta$ to $I_{j'}$ for $j' \in J'$.

Secondly, suppose that $\beta$ does not factor via $\alpha$, but that $\delta \beta = \gamma' \alpha$ for some $\gamma': I' \epi J'$. Then the relevant component of the bottom horizontal map in (\ref{eq:commute}) is trivial. We need to show that the other composite in (\ref{eq:commute}) is trivial. This composite is given by the smash product, over $j' \in J'$, of the following part of the previous diagram:
\[ \begin{diagram}
  \node{\B(\one,\P,\P)(I'_{j'}) \smsh \Smsh_{i' \in I'_{j'}}^{\phantom{I'_{j'}}} \P(I_{i'})} \arrow{e}
    \node{\B(\one,\P,\P)(I_{j'})} \arrow{s} \\
  \node[2]{\B\P(J_{j'}) \smsh \Smsh_{j \in J_{j'}} \B(\one,\P,\P)(I_j)}
\end{diagram} \]
The only obstruction to factoring $\beta$ via $\alpha$ is that there must be some two elements $i_1,i_2 \in I$ with $\alpha(i_1) = \alpha(i_2)$ but $\beta(i_1) \neq \beta(i_2)$. In this case, let $j' = \gamma'\alpha(i_1) = \gamma'\alpha(i_2)$ and consider the above diagram for this particular $j'$. To see that this composite is trivial, we recall the structure of the bar construction $\B(\one,\P,\P)$. In particular, $\B(\one,\P,\P)(I'_{j'})$ is stratified by trees with leaves labelled surjectively by the set $I'_{j'}$. Since $\alpha(i_1) = \alpha(i_2)$, these two elements must label the same leaf of such a tree. It follows that the image of the top horizontal map is contained in the strata corresponding to trees (with leaves labelled by $I_{j'}$) where $i_1$ and $i_2$ label the same leaf. But then according to the definition of the vertical map, since $\beta(i_1) \neq \beta(i_2)$, each such stratum is mapped by to the basepoint in the bottom-right corner. It follows that the above composite is trivial as
required.

Finally, suppose that neither $\beta$ not $\delta \beta$ factors via $\alpha$. In this case, the relevant components of both horizontal maps in (\ref{eq:commute}) are trivial and the square commutes.
\end{proof}

\begin{lemma} \label{lem:B-1}
There is an equivalence of bisymmetric sequences
\[ \un{\B}_\P \homeq \un{\one} \]
that preserves the $\P$-module structure on the first variable.
\end{lemma}
\begin{proof}
For finite sets $I,J$, the equivalence of right $\P$-modules $\B(\one,\P,\P) \weq \one$ determines a natural equivalence
\[ \prod_{I \epi J} \Smsh_{j \in J} \B(\one,\P,\P)(I_j) \to \prod_{I \epi J} \Smsh_{j \in J} \one(I_j) \]
and this preserves the right $\P$-module structures in the $I$ variable (with that on the right-hand side the trivial structure). The right-hand side is the same as $\un{\one}(I,J)$ but with $\Wdge_{\Sigma_n} S$ replaced by the equivalent $\prod_{\Sigma_n} S$ (when $n = |I| = |J|$).
\end{proof}

\begin{remark} \label{rem:B}
Lemma~\ref{lem:B-1} tells us that the $\P$-module structure on each symmetric sequence $\un{\B}_\P(-,J)$ is equivalent to the trivial structure. The same is true about the $\K\P$-module structure on each $\un{\B}_\P(I,-)$. This is proved via equivalences
\[ \un{\one} \homeq \un{\B}_\P \]
that preserve the $\K\P$-module structure on the second variables. The essential point, however, is that these two structures are \emph{not} equivalent when taken together. That is, the bisymmetric sequence $\un{\B}_\P$ is \emph{not} equivalent to $\un{\one}$ in the category of functors $\un{\P}^{op} \times \un{\K\P}^{op} \to \spectra$. There is no zigzag of equivalences between $\un{\B}_\P$ and $\un{\one}$ that preserves both structures simultaneously.
\end{remark}

\begin{prop} \label{prop:N-indec}
Let $\P$ be an operad of spectra and $\tilde{\N}$ a $\Sigma$-cofibrant $\P$-comodule. Then the symmetric sequence
\[ \tilde{\N} \homsmsh_{\P} \un{\B}_\P \]
is equivalent to $\tilde{\N} \homsmsh_{\P} \un{\one}$, and has a canonical $\K\P$-module structure.
\end{prop}
\begin{proof}
The only part to prove is the equivalence to $\tilde{\N} \homsmsh_{\P} \un{\one}$ which follows from Lemma~\ref{lem:B-1} and Lemma~\ref{lem:homsmsh}.
\end{proof}

We now consider the naturality of the construction of $\un{\B}_\P$ in the operad $\P$.

\begin{lemma} \label{lem:Bphi}
Let $\phi: \P \to \P'$ be a morphism of operads in $\spectra$. The induced natural transformation of bisymmetric sequences
\[ \un{\B}_\phi: \un{\B}_\P \to \un{\B}_{\P'} \]
is a morphism of $\P$-modules (in the first variables, with the $\P$-action on $\un{\B}_{\P'}(-,J)$ given by pulling back along $\phi$), and of $\K\P'$-modules (in the second variables, with the $\K\P'$-action on $\un{\B}_\P(I,-)$ given by pulling back the $\K\P$-action along the operad map $\K\phi: \K\P' \to \K\P$).
\end{lemma}
\begin{proof}
This is an easy but lengthy diagram-chase.
\end{proof}

We are now in position to describe models for the partially-stabilized cross-effects of a polynomial functor that have the desired $\K\E_L$-module structure.

\begin{definition} \label{def:dF-polynomial}
For a polynomial functor $F \in [\finbased,\spectra]$, write $\tilde{\N}[F]$ for a (functorial) $\Sigma$-cofibrant replacement of the $\Com$-comodule $\N[F]$ of Definition~\ref{def:F-comodule}. We then define a $\K\E_L$-module $\un{\d}^L[F]$ by
\[ \un{\d}^L[F] := \tilde{\N}[F] \homsmsh_{\E_L} \un{\B}_{\E_L} \]
with $\K\E_L$-module structure arising as in Proposition~\ref{prop:N-indec}. We also have, for each $L$, a map of symmetric sequences
\[ \un{m}^L: \un{\d}^L[F] \to \un{\d}^{L+1}[F] \]
given by the composite
\[ \tilde{\N}[F] \homsmsh_{\E_L} \un{\B}_{\E_L} \to \tilde{\N}[F] \homsmsh_{\E_L} \un{\B}_{\E_{L+1}} \to \tilde{\N}[F] \homsmsh_{\E_{L+1}} \un{\B}_{\E_{L+1}}. \]
in which the second map is induced by the operad map $\E_L \to \E_{L+1}$ and the first is the corresponding map of $\E_L$-modules $\un{\B}_{\E_L} \to \un{\B}_{\E_{L+1}}$ from Lemma~\ref{lem:Bphi}. The map $\un{m}^L$ is also a morphism of $\K\E_{L+1}$-modules by Lemma~\ref{lem:Bphi}.

The sequence
\[ \un{\d}^\bullet[F] : \un{\d}^0[F] \to \un{\d}^1[F] \to \un{\d}^2[F] \to \dots \]
is therefore a module over the pro-operad $\K\E_\bullet$, in the sense of Definition~\ref{def:module-pro}. It follows that the homotopy colimit
\[ \un{\d}[F] := \hocolim_L \un{\d}^L[F] \]
is a coalgebra over the comonad $\C_{\K\E_\bullet}$.
\end{definition}

\begin{proposition} \label{prop:dF-polynomial}
Let $F \in [\finbased,\spectra]$ be polynomial. Then we have natural equivalences of symmetric sequences
\[ \Sigma^{-*L}\creff_*F(S^L,\dots,S^L) \homeq \un{\d}^L[F] \]
such that the following diagrams commute in the homotopy category of symmetric sequences:
\[ \begin{diagram}
  \node{\Sigma^{-*L} \creff_*F(S^L,\dots,S^L)} \arrow{e,t}{\sim} \arrow{s} \node{\un{\d}^L[F]} \arrow{s,r}{\un{m}^L} \\
  \node{\Sigma^{-*(L+1)} \creff_*F(S^{L+1},\dots,S^{L+1})} \arrow{e,t}{\sim} \node{\un{\d}^{L+1}[F]}
\end{diagram} \]
where the left-hand maps are those of (\ref{eq:creff-map}). Taking homotopy colimits as $L \to \infty$ we get an equivalence
\[ \der_*(F) \homeq \un{\d}[F]. \]
\end{proposition}
\begin{proof}
This follows from Propositions~\ref{prop:N-indec} and \ref{prop:topsp-creff}.
\end{proof}

We have therefore constructed models for the partially-stabilized cross-effects of a polynomial functor as a module over the sequence of operads $\K\E_\bullet$, and for the Goodwillie derivatives as a coalgebra over the associated comonad $\C_{\K\E_\bullet}$. We now extend these constructions to analytic functors by taking the inverse limit over the Taylor tower (part (4) of our overall construction).

Recall that a functor $F: \finbased \to \spectra$ is \emph{$\rho$-analytic} if it satisfies Goodwillie's condition for being stably $n$-excisive in a way that depends in a controlled manner on $\rho$. Specifically, there should exist a number $q$ such that $F$ satisfies condition $E_n(n\rho - q, \rho + 1)$ of \cite[4.1]{goodwillie:1991}. We say that $F$ is \emph{analytic} if it is $\rho$-analytic for some $\rho$.

\begin{definition} \label{def:dF-analytic}
For an analytic functor $F \in [\finbased,\spectra]$, we define
\[ \un{\d}^L[F] := \holim_n \un{\d}^L[P_nF] \]
where the homotopy inverse limit is formed in the category of $\K\E_L$-modules (or, equivalently, in the underlying category of symmetric sequences). The maps in the inverse diagram for the homotopy limit are those induced by the maps in the Taylor tower of $F$. The maps $\un{m}^L$ for each $P_nF$ induce corresponding maps (of $\K\E_{L+1}$-modules):
\[ \un{m}^L: \un{\d}^L[F] \to \un{\d}^{L+1}[F] \]
and hence we obtain a $\K\E_\bullet$-module $\un{\d}^\bullet[F]$ and hence a $\C_{\K\E_\bullet}$-coalgebra
\[ \un{\d}[F] := \hocolim_L \un{\d}^L[F]. \]
(If $F$ is polynomial, this new definition of $\un{\d}^\bullet[F]$ is termwise-equivalent to that of Definition~\ref{def:dF-polynomial}.)
\end{definition}

\begin{lemma} \label{lem:dF-analytic}
For a $\rho$-analytic functor $F: \finbased \to \spectra$ and positive integer $L$, there are natural maps in the homotopy category of symmetric sequences of the form
\[ \Sigma^{-*L}\creff_*F(S^L,\dots,S^L) \to \un{\d}^L[F] \]
that are equivalences when $L \geq \rho$, and such that the following diagrams commute
\[ \begin{diagram}
  \node{\Sigma^{-*L} \creff_*F(S^L,\dots,S^L)} \arrow{e} \arrow{s} \node{\un{\d}^L[F]} \arrow{s,r}{\un{m}^L} \\
  \node{\Sigma^{-*(L+1)} \creff_*F(S^{L+1},\dots,S^{L+1})} \arrow{e} \node{\un{\d}^{L+1}[F]}
\end{diagram} \]
where the left-hand maps are those of (\ref{eq:creff-map}). Taking the homotopy colimit as $L \to \infty$, we obtain an equivalence of symmetric sequences
\[ \der_*(F) \homeq \un{\d}[F]. \]
\end{lemma}
\begin{proof}
By Proposition~\ref{prop:dF-polynomial} we have
\[ \un{\d}[P_nF] \homeq \Sigma^{-*L}\creff_*(P_nF)(S^L,\dots,S^L). \]
Therefore
\[ \un{\d}^L[F] \homeq \holim \Sigma^{-*L}\creff_*(P_nF)(S^L,\dots,S^L) \homeq \Sigma^{-*L}\creff_*(\holim P_nF)(S^L,\dots,S^L) \]
since taking cross-effects commutes with taking homotopy limits. The required map is then induced by the natural transformation
\[ F \to \holim P_nF. \]
Since $F$ is $\rho$-analytic, the natural map $F \to \holim P_nF$ is an equivalence on $\rho$-connected spaces. It follows that we have an equivalence
\[ \Sigma^{-*L}\creff_*F(S^L,\dots,S^L) \weq \Sigma^{-*L}\creff_*(\holim P_nF)(S^L,\dots,S^L) \]
when $L \geq \rho$, which yields the first claim. Taking homotopy colimits over the $L$ variable on each side then gives the second claim.
\end{proof}

The final goal of this section (part (5)) is to extend our constructions to arbitrary (i.e. non-analytic) functors from based spaces to spectra. For this, we recall from \cite[4.3]{arone/ching:2014} that the derivatives of an arbitrary functor $F$ can be obtained by forming the coend between the values $F(X)$ and the derivatives of the representable functors $\Sigma^\infty \Hom_{\based}(X,-)$. Since these representable functors are analytic, we can use the models for their derivatives from Definition~\ref{def:dF-analytic}, and since the comonad $\C_{\K\E_\bullet}$ is enriched in spectra, its coaction on those derivatives extends to the coend.

\begin{definition} \label{def:rep}
For $X \in \finbased$, we write
\[ R_X : \finbased \to \spectra \]
for the representable functor given by
\[ R_X(Y) := \Sigma^\infty \Hom_{\based}(X,Y). \]
\end{definition}

\begin{definition} \label{def:dRX}
For $X \in \finbased$, we now apply the construction of Definition~\ref{def:cofcoalg} to the module $\un{\d}^\bullet[R_X]$ of Definition~\ref{def:dF-analytic} to get a $\K\E_\bullet$-module $\d^\bullet[R_X]$ and a morphism of $\K\E_\bullet$-modules
\[ \d^\bullet[R_X] \to \un{\d}^\bullet[R_X] \]
formed from weak equivalences
\[ \d^L[R_X] \weq \un{\d}^L[R_X] \]
and such that each $\d^L[R_X]$ is $\Sigma$-cofibrant. In particular, the homotopy colimit
\[ \d[R_X] := \hocolim \d^L[R_X] \]
is a $\Sigma$-cofibrant $\C_{\K\E_\bullet}$-coalgebra that is equivalent, as a symmetric sequence, to $\der_*(R_X)$.
\end{definition}

\begin{remark}
It is important that the functor
\[ (\finbased)^{op} \to \symseq; \quad X \mapsto \d[R_X] \]
be simplicially-enriched in order to make Definition~\ref{def:dF} below. The reader can check that we have built $\d[R_X]$ from a succession of simplicially-enriched constructions, including: applying Goodwillie's construction $P_n$; forming the natural transformation objects $\N[F]$; taking cofibrant replacements (which can be done simplicially by work of Rezk-Schwede-Shipley~\cite{rezk/schwede/shipley:2001}; forming homotopy coends over $\E_L$; taking homotopy limits; applying functorial factorizations as in the construction of Definition~\ref{def:cofcoalg}; and taking homotopy colimits.
\end{remark}

Finally, then, we obtain models for the derivatives of any pointed simplicial functor $F: \finbased \to \spectra$ by enriched left Kan extension from our models for the representable functors.

\begin{definition} \label{def:dF}
For any $F \in [\finbased,\spectra]$ and each $L$, define a $\K\E_L$-module $\d^L[F]$ by taking the enriched coend
\[ \d^L[F] := F(X) \smsh_{X \in \finbased} \d^L[R_X] \]
over the simplicial category $\finbased$. We then obtain a $\C_{\K\E_\bullet}$-coalgebra $\d[F]$ given by
\[ \d[F] := \hocolim_L \d^L[F] \isom F(X) \smsh_{X \in \finbased} \d[R_X]. \]
Notice that, by the dual Yoneda Lemma, we have $R_Y(X) \smsh_{X} \d[R_X] \isom \d[R_Y]$ so that this definition of $\d[F]$ is consistent with that already made for the functors $R_X$ themselves.
\end{definition}

Note that in Definition~\ref{def:dF} we use a strict (not homotopy) coend which typically only has the correct homotopy type when $F$ is a cofibrant object in the projective model structure on $[\finbased,\spectra]$.

\begin{proposition} \label{prop:dF}
Let $F: \finbased \to \spectra$ be pointed simplicial and let $\c$ denote an arbitrary cofibrant replacement functor for the projective model structure on $[\finbased,\spectra]$. There are natural maps (in the homotopy category of symmetric sequences)
\[ \eta_L(F): \Sigma^{-*L}\creff_*F(S^L,\dots,S^L) \to \d^L[\c F] \]
such that the following diagrams commute
\[ \begin{diagram}
  \node{\Sigma^{-*L} \creff_*F(S^L,\dots,S^L)} \arrow{e,t}{\eta_L(F)} \arrow{s} \node{\d^L[\c F]} \arrow{s} \\
  \node{\Sigma^{-*(L+1)} \creff_*F(S^{L+1},\dots,S^{L+1})} \arrow{e,t}{\eta_{L+1}(F)} \node{\d^{L+1}[\c F]}
\end{diagram} \]
where the left-hand maps are those of (\ref{eq:creff-map}) and the right-hand maps are those that form the $\K\E_\bullet$-module $\d^\bullet[\c F]$. Taking the homotopy colimit as $L \to \infty$ we therefore get a map
\[ \eta(F): \der_*(F) \weq \d[\c F] \]
and this is an equivalence for all $F$.
\end{proposition}
\begin{proof}
From Lemma~\ref{lem:dF-analytic} we have a zigzag of maps (in which all the backwards maps are equivalences)
\begin{equation} \label{eq:prop:dF1} \Sigma^{-*L}\creff_*(R_X)(S^L,\dots,S^L) \arrow{e,..} \d^L[R_X]. \end{equation}
Since $\c F$ is a cofibrant functor, the coend construction $\c F(X) \smsh_{X} -$ preserves weak equivalences between objectwise-$\Sigma$-cofibrant diagrams $(\finbased)^{op} \to \symseq$. So applying this coend construction to the zigzag of maps (\ref{eq:prop:dF1}), we obtain a corresponding zigzag
\begin{equation} \label{eq:prop:dF2} \c F(X) \smsh_X \Sigma^{-*L} \widetilde{\creff}_*(R_X)(S^L,\dots,S^L) \arrow{e,..} \d^L[F] \end{equation}
where $\widetilde{\creff}_*(R_X)$ denotes a $\Sigma$-cofibrant replacement of the cross-effects of $R_X$. The strict coend on the left-hand side is equivalent to the corresponding homotopy coend since $\c F$ is cofibrant. Therefore, because taking cross-effects commutes with homotopy colimits, the left-hand term in (\ref{eq:prop:dF2}) is equivalent to
\[ \Sigma^{-*L}\creff_*(\c F(X) \smsh_{X} R_X)(S^L,\dots,S^L) \]
which is equivalent to
\[ \Sigma^{-*L}\creff_*(F)(S^L,\dots,S^L) \]
by the dual Yoneda Lemma. It follows that the zigzag (\ref{eq:prop:dF2}) represents the required map $\eta_L(F)$ in the homotopy category of symmetric sequences.

It remains to show that $\eta(F) = \hocolim \eta_L(F)$ is an equivalence. We know from Lemma~\ref{lem:dF-analytic} that the homotopy colimit of the maps (\ref{eq:prop:dF1}), as $L \to \infty$, is an equivalence for each $X$. Since this homotopy colimit commutes with the homotopy coend, we deduce that $\eta(F)$ is an equivalence as required.
\end{proof}

This completes the construction (started in Definition~\ref{def:SigmainftyX-Com}) of a model for the derivatives of a pointed simplicial functor $F: \finbased \to \spectra$ that form a $\C_{\K\E_\bullet}$-coalgebra.

Note that for an analytic functor $F \in [\finbased,\spectra]$ we have now defined two different $\C_{\K\E_\bullet}$-coalgebra models for the derivatives of $F$: the object $\un{\d}[F]$ of Definition~\ref{def:dF-analytic} and the object $\d[F]$ of Definition~\ref{def:dF}. It is important to know that these objects are actually equivalent as $\C_{\K\E_\bullet}$-coalgebras.

There is a small technical obstacle to constructing an equivalence of $\C_{\K\E_\bullet}$-coalgebras between $\d[F]$ and $\un{\d}[F]$, which is that the construction $F \mapsto \un{\d}[F]$ is not enriched in spectra. (In particular, we do not have a natural map (of $\C_{\K\E_\bullet}$-coalgebras) of the form $X \smsh \un{\d}[F] \to \un{\d}[X \smsh F]$ for a spectrum $X$.) The underlying reason for this is that the cofibrant replacement in the category of $\Com$-comodules used to form the object $\tilde{\N}[F]$ is not enriched in spectra. The following lemma shows that there is at least a zigzag of coalgebra maps (which are, in fact, equivalences).

\begin{lemma} \label{lem:d-enriched}
For $\rho$-analytic $F \in [\finbased,\spectra]$, cofibrant $X \in \spectra$, and $L \geq \rho$, there is a zigzag of equivalences of $\K\E_L$-modules
\[ \un{\d}^L[X \smsh F] \homeq X \smsh \un{\d}^L[F], \]
and a zigzag of equivalences of $\C_{\K\E_\bullet}$-coalgebras
\[ \un{\d}[X \smsh F] \homeq X \smsh \un{\d}[F]. \]
The $\C_{\K\E_\bullet}$-coalgebra structure on the right-hand side relies on the fact that the comonad $\C_{\K\E_\bullet}$ itself \emph{is} enriched over spectra.
\end{lemma}
\begin{proof}
Consider first the case that $F$ is polynomial. Then there is a zigzag of equivalences of $\Sigma$-cofibrant $\Com$-comodules
\begin{equation} \label{eq:N-enriched}
    X \smsh \tilde{\N}[F] \lweq \widetilde{[X \smsh \tilde{\N}[F]]} \weq \widetilde{[X \smsh \N[F]]} \weq \tilde{\N}[X \smsh F]
\end{equation}
where the tildes denote a fixed cofibrant replacement functor in the category of $\Com$-comodules. The final equivalence is induced by the natural map of $\Com$-comodules
\[ X \smsh \Nat_{Y \in \finbased}(\widetilde{\Sigma^\infty Y^{\smsh *}},FY) \to \Nat_{Y \in \finbased}(\widetilde{\Sigma^\infty Y^{\smsh *}}, X \smsh FY), \]
and is an equivalence by Lemma~\ref{lem:creff}.

We get from (\ref{eq:N-enriched}) the required chain of equivalences of $\K\E_L$-modules
\[ X \smsh \un{\d}^L[F] := X \smsh \tilde{\N}[F] \homsmsh_{\E_L} \un{\B}_{\E_L} \homeq \tilde{\N}[X \smsh F] \homsmsh_{\E_L} \un{\B}_{\E_L} = \un{\d}^L[X \smsh F]. \]
It is important for this that each term in the chain (\ref{eq:N-enriched}) is $\Sigma$-cofibrant. Note that these equivalences are compatible with the operad maps $\K\E_{L+1} \to \K\E_L$ and thus induce a corresponding chain of equivalences of $\C_{\K\E_\bullet}$-coalgebras
\[ X \smsh \un{\d}[F] \homeq \un{\d}[X \smsh F]. \]

Now suppose that $F$ is $\rho$-analytic with $L \geq \rho$. Then we have a sequence of maps/equivalences of $\K\E_L$-modules of the form
\begin{equation} \label{eq:ana-tensor} \begin{split}
  X \smsh \un{\d}^L[F] = X \smsh \holim_n \un{\d}^L[P_nF] &\to \holim_n X \smsh \un{\d}^L[P_nF] \\
    &\homeq \holim_n \un{\d}^L[X \smsh P_nF] \\
    &\homeq \holim_n \un{\d}^L[P_n(X \smsh F)] = \un{\d}^L[X \smsh F].
\end{split} \end{equation}
By Lemma~\ref{lem:dF-analytic} the composite can be identified (in the homotopy category of symmetric sequences) with the map
\[ X \smsh \Sigma^{-*L}\creff_*F(S^L,\dots,S^L) \to \Sigma^{-*L}\creff_*(X \smsh F)(S^L,\dots,S^L) \]
which is an equivalence since cross-effects are equivalent to co-cross-effects, and hence commute with smashing by $X$. Taking the homotopy colimit of the equivalences (\ref{eq:ana-tensor}) as $L \to \infty$ we get the required equivalence
\[ X \smsh \un{\d}[F] \homeq \un{\d}[X \smsh F]. \]
\end{proof}

\begin{lemma} \label{lem:d-d-eq}
For a cofibrant analytic functor $F \in [\finbased,\spectra]$, there are zigzags of equivalences of $\K\E_L$-modules
\[ \un{\d}^L[F] \homeq \d^L[F] \]
and of equivalences of $\C_{\K\E_\bullet}$-coalgebras
\[ \un{\d}[F] \homeq \d[F]. \]
\end{lemma}
\begin{proof}
In the case $F = R_X$, the required equivalence follows immediately from the construction of $\d^\bullet[R_X]$ from $\un{\d}^\bullet[R_X]$ in Definition~\ref{def:dRX}. For arbitrary $F$, we have
\[ \d[F] = F(X) \smsh_{X} \d[R_X] \homeq F(X) \smsh_{X} \un{\d}[R_X] \homeq \un{\d}[F(X) \smsh_{X} R_X] \isom \un{\d}[F] \]
(and similarly for $\d^L$ instead of $\d$) with the penultimate equivalence arising from Lemma~\ref{lem:d-enriched}.
\end{proof}

\begin{corollary} \label{cor:eta}
If $F$ is $\rho$-analytic and $L \geq \rho$, then the map
\[ \eta_L(F) : \Sigma^{-*L} \creff_*F(S^L,\dots,S^L) \to \d^L[\c F] \]
of Proposition~\ref{prop:dF} is an equivalence.
\end{corollary}
\begin{proof}
The map $\eta_L(F)$ is equivalent to the composite
\[ \Sigma^{-*L}\creff_*(\c F)(S^L,\dots,S^L) \weq \un{\d}^L[\c F] \weq \d^L[\c F] \]
given by the equivalences of Lemmas~\ref{lem:dF-analytic} and \ref{lem:d-d-eq}.
\end{proof}

\subsection{Classification of polynomial functors from based spaces to spectra}

The main goal of this section is to prove Theorem~\ref{thm:equiv-topsp-poly} which provides an equivalence between the homotopy theory of polynomial functors from based spaces to spectra and the homotopy theory of bounded $\C_{\K\E_\bullet}$-coalgebras. Our approach is to make use of the general theory of \cite{arone/ching:2014} in which we have already constructed another comonad $\C$ whose coalgebras classify polynomial functors.

\begin{definition} \label{def:dF-adjunction}
The functor
\[ \d: [\finbased,\spectra] \to \symseq \]
constructed in Definition~\ref{def:dF} has a right adjoint $\Phi$ given by
\[ \Phi(\A)(X) := \Map_{\Sigma}(\d[R_X],\A). \]
Since $\d[R_X]$ is a $\Sigma$-cofibrant symmetric sequence, the right adjoint $\Phi$ preserves fibrations and acyclic fibrations, so that $(\d,\Phi)$ is a Quillen adjunction (with respect to the projective model structures on both sides).
\end{definition}

We can now apply the work of \cite{arone/ching:2014} to the adjunction $(\d,\Phi)$. This gives us a comonad $\C$ on the category of symmetric sequences that acts on $\d[F]$ for any $F \in [\finbased,\spectra]$. According to \cite[3.14]{arone/ching:2014}, the Taylor tower of a functor $F$ can then be recovered from $\d[F]$ by a cobar construction, and the homotopy theory of pointed simplicial polynomial functors $\finbased \to \spectra$ is equivalent to the homotopy theory of bounded $\C$-coalgebras. Here we show that $\C$ is equivalent, as a comonad, to $\C_{\K\E_\bullet}$, and we deduce that the the Taylor tower of $F$ can be recovered from the $\C_{\K\E_\bullet}$-coalgebra structure on $\d[F]$ constructed in Definition~\ref{def:dF}.

\begin{definition} \label{def:theta}
The comonad $\C: \symseq \to \symseq$ can be written as
\[ \C := \d \c \Phi \]
where $\c$ is a cofibrant replacement comonad for the category $[\finbased,\spectra]$. See \cite[Sec. 3]{arone/ching:2014} for more details.\footnote{More precisely, we constructed $\C$ in \cite{arone/ching:2014} using a Quillen equivalence $u: [\finbased,\spectra]_c \rightleftarrows [\finbased,\spectra]: c$ where $[\finbased,\spectra]_c$ is a model category in which every object is cofibrant. The comonad $\C$ is then given by $\der_* u c \Phi$. In this case the composite $uc$ provides a cofibrant replacement comonad for the category $[\finbased,\spectra]$ which we denote by $\c$.}

The functor $\C$ takes values in $\C_{\K\E_\bullet}$-coalgebras (since $\d$ does) and so we get a morphism of comonads
\[ \theta: \C \to \C_{\K\E_\bullet} \]
given by the composite
\[ \C(\A) \to \C_{\K\E_\bullet}(\C(\A)) \arrow{e} \C_{\K\E_\bullet}(\A) \]
where the first map is given by the $\C_{\K\E_\bullet}$-coalgebra structure on $\C(\A)$ and the second by the counit for the comonad $\C$.
\end{definition}

\begin{theorem} \label{thm:topsp}
The comonad map
\[ \theta_{\A}: \C(\A) \to \C_{\K\E_\bullet}(\A) \]
is an equivalence for any symmetric sequence $\A$.
\end{theorem}
\begin{proof}
Our proof is based on the construction, for bounded symmetric sequences $\A$, of natural zigzags of equivalences of $\K\E_L$-modules
\begin{equation} \label{eq:zeta^L} \d^L[\c\Phi\A] \homeq \C_{\K\E_L}(\A) \end{equation}
that commute with the $\K\E_\bullet$-module structure maps
\[ \d^L[\c\Phi\A] \to \d^{L+1}[\c\Phi\A] \]
on the left-hand side, and the comonad maps
\[ \C_{\K\E_L}(\A) \to \C_{\K\E_{L+1}}(\A) \]
on the right-hand side.

We start by noting that there are equivalences of $\Com$-comodules
\[ \begin{split} \N[\Phi\A] &= \Nat_{Y \in \finbased}(\widetilde{\Sigma^\infty Y^{\smsh *}}, \Phi\A(Y)) \\
    &\isom \Map_{\Sigma}(\d[\widetilde{\Sigma^\infty Y^{\smsh *}}],\A) \\
    &\homeq \Map_{\Sigma}(\un{\one}(-,*),\A(-)) \\
    &\homeq \A
\end{split} \]
where $\A$ has the trivial $\Com$-comodule structure. We therefore get, using Lemma~\ref{lem:d-d-eq}, an equivalence of $\K\E_L$-modules
\begin{equation} \label{eq:dL-calc} \d^L[\c\Phi\A] \homeq \tilde{\N}[\Phi\A] \homsmsh_{\E_L} \un{\B}_{\E_L} \homeq \A \homsmsh_{\E_L} \un{\B}_{\E_L} \end{equation}
where $\A$ is given the trivial $\E_L$-comodule structure.

We next claim that, when $\A$ is bounded, $\A \homsmsh_{\E_L} \un{\B}_{\E_L}$ is equivalent to the cofree $\K\E_L$-module generated by $\A$. This is part of a version of Koszul duality between $\E_L$-comodules and $\K\E_L$-modules under which trivial $\E_L$-comodules correspond to cofree $\K\E_L$-modules. We record this claim as a separate proposition.

\begin{prop} \label{prop:indec-triv}
Let $\A$ be a bounded trivial $\Sigma$-cofibrant $\E_L$-comodule. Then there is a zigzag of equivalences of $\K\E_L$-modules
\[ \A \homsmsh_{\E_L} \un{\B}_{\E_L} \homeq \C_{\K\E_L}(\A). \]
In other words, the derived indecomposables of a bounded trivial $\E_L$-comodule form a cofree $\K\E_L$-module.
\end{prop}
\begin{proof}
To define the required zigzag, we introduce a symmetric sequence equivalent to $\A$, but that accepts a natural map from $\A \homsmsh_{\E_L} \un{\B}_{\E_L}$. We define
\[ \hat{\A}(n) := \A(n) \smsh_{\Sigma_n} \prod_{\Sigma_n} S. \]
Note that there is an isomorphism $\A(n) \isom \A(n) \smsh_{\Sigma_n} \Wdge_{\Sigma_n} S$ and that, since $\A$ is $\Sigma$-cofibrant, the map
\[ \A \to \hat{\A}, \]
induced by the map from coproduct to product, is an equivalence.

We can now define a map of symmetric sequences $\A \homsmsh_{\E_L} \un{\B}_{\E_L} \to \hat{\A}$ as follows. For given $n$, we have a map
\begin{equation} \label{eq:ABEL} [\A \homsmsh_{\E_L} \un{\B}_{\E_L}](n) \to [\A_{=n} \homsmsh_{\E_L} \un{\B}_{\E_L}](n) \end{equation}
induced by the projection $\A \to \A_{=n}$ where $\A_{=n}$ is the symmetric sequence with $\A(n)$ in the \ord{$n$} term and trivial otherwise. The right-hand side in \ref{eq:ABEL} is the geometric realization of a simplicial spectrum that is constant with terms
\[ \A(n) \smsh_{\Sigma_n} \un{\B}_{\E_L}(n,n) \isom \hat{\A}(n) \]
since $\un{\B}_{\E_L}(n,n) \isom \prod_{\Sigma_n} S$ by Definition~\ref{def:B}. It follows that we have isomorphisms
\begin{equation} \label{eq:ABEL2} [\A_{=n} \homsmsh_{\E_L} \un{\B}_{\E_L}](n) \isom \hat{\A}(n). \end{equation}
Composing (\ref{eq:ABEL}) and (\ref{eq:ABEL2}) we get the required map of symmetric sequences
\[ \epsilon: \A \homsmsh_{\E_L} \un{\B}_{\E_L} \to \hat{\A}. \]

We can now define the zigzag of maps required by the Proposition. We have maps of $\K\E_L$-modules
\begin{equation} \label{eq:ABEL3} \A \homsmsh_{\E_L} \un{\B}_{\E_L} \to \C_{\K\E_L}(\A \homsmsh_{\E_L} \un{\B}_{\E_L}) \arrow{e,t}{\epsilon} \C_{\K\E_L}(\hat{\A}) \lweq \C_{\K\E_L}(\A) \end{equation}
where the first map is given by the $\K\E_L$-module structure on $\un{\B}_{\E_L}$. It is now sufficient to check that this composite is an equivalence.

We start by reducing to the case that $\A = \un{\one}(n,*)$, i.e. $\A$ consists of only a free $\Sigma_n$-spectrum concentrated in arity $n$. Since $\A$ is a trivial $\E_L$-comodule, we have an equivalence of $\E_L$-comodules
\[ \A \homeq \Wdge_n \A(n) \smsh_{h\Sigma_n} \un{\one}(n,*). \]
To reduce our claim to the case $\A = \un{\one}(n,*)$ it is sufficient to check that the functors $- \homsmsh_{\E_L} \un{\B}_{\E_L}$ and $\C_{\K\E_L}$ commute with: (1) smashing with a cofibrant spectrum; (2) taking homotopy orbits with respect to $\Sigma_n$-action; and (3) finite coproducts. For $-\homsmsh_{\E_L} \un{\B}_{\E_L}$ these are clear and for $\C_{\K\E_L}$ they follow from the next lemma.

\begin{lemma} \label{lem:CKE-eq}
For bounded symmetric sequences $\A$, there is a natural equivalence of spectra
\[ \C_{\K\E_L}(\A)(k) \homeq \Wdge_{m = m_1+\dots+m_k} \left[ \B\E_L(m_1) \smsh \dots \smsh \B\E_L(m_k) \smsh \A(m) \right]_{h\Sigma_{m_1} \times \dots \times \Sigma_{m_k}} \]
where the coproduct is taken over all ordered partitions of a positive integer $m$ as a sum of positive integers $m_1,\dots,m_k$.
\end{lemma}
\begin{proof}
From Definition~\ref{def:comonad-operad} we have
\[ \C_{\K\E_L}(\A)(k) = \prod_{m} \left[ \prod_{\un{m} \epi \un{k}} \Map(\K\E_L(m_1) \smsh \dots \smsh \K\E_L(m_k), \A(m)) \right]^{\Sigma_m}. \]
First note that the terms in the outer product can be written as
\[ \prod_{m = m_1 + \dots + m_k} \left[\Map(\K\E_L(m_1) \smsh \dots \smsh \K\E_L(m_k), \A(m)) \right]^{\Sigma_{m_1} \times \dots \times \Sigma_{m_k}} \]
where the product is taken over all ordered partitions as in the statement of the lemma. Since each $\K\E_L(m_j)$ is cofibrant as a $\Sigma_{m_j}$-spectrum, the strict fixed points here are equivalent to the corresponding homotopy fixed points. Furthermore, we have equivariant equivalences
\begin{equation} \label{eq:dual-KE} \Map(\K\E_L(m_1) \smsh \dots \smsh \K\E_L(m_k), \A(m)) \homeq \B\E_L(m_1) \smsh \dots \smsh \B\E_L(m_k) \smsh \A(m) \end{equation}
and each of the products involved here is finite (since $\A$ is bounded), so altogether we can write
\[ \C_{\K\E_L}(\A)(k) \homeq \Wdge_{m = m_1+\dots+m_k} \left[ \B\E_L(m_1) \smsh \dots \smsh \B\E_L(m_k) \smsh \A(m) \right]^{h\Sigma_{m_1} \times \dots \times \Sigma_{m_k}}. \]
All that remains is to show that the homotopy fixed points here are equivalent to the corresponding homotopy orbits. To see this we note that the $\Sigma_{m}$-spectrum $\B\E_L(m)$ can be built from finitely many $\Sigma_{m}$-free cell spectra (because each space $\mathbb{E}_L(k)$ is a finite free $\Sigma_k$-cell complex). Therefore the norm map from homotopy orbits to homotopy fixed points is an equivalence.
\end{proof}

Continuing with the proof of \ref{prop:indec-triv}, this completes the reduction to the case $\A = \un{\one}(n,*)$ so we now consider this case. Recall that we have to prove that the composite
\begin{equation} \label{eq:KE-1-map} \un{\one}(n,*) \homsmsh_{\E_L} \un{\B}_{\E_L} \to \C_{\K\E_L}(\un{\one}(n,*) \homsmsh_{\E_L} \un{\B}_{\E_L}) \to \C_{\K\E_L}(\hat{\un{\one}}(n,*)) \end{equation}
is an equivalence. For this we really need to get our hands on the definition of $\un{\B}_{\E_L}$ in \ref{def:B}, including the $\E_L$ and $\K\E_L$-module structures described there.

The \ord{$k$} term in this map of symmetric sequences takes the form
\[ \begin{split}  \un{\one}(n,*) \homsmsh_{\E_L} \un{\B}_{\E_L}(*,k)
    &\to \left[ \prod_{\un{n}\epi\un{k}}\Map(\K\E_L(n_1)\smsh\dots \smsh \K\E_L(n_k), \un{\one}(n,*) \homsmsh_{\E_L} \un{\B}_{\E_L}(*,n)) \right]^{\Sigma_n} \\
    &\isom \left[ \prod_{\un{n} \epi \un{k}} \Map(\K\E_L(n_1) \smsh \dots \smsh \K\E_L(n_k), \hat{\un{\one}}(n,n)) \right]^{\Sigma_n}
\end{split} \]
where the first map comes from the right $\K\E_L$-module structure on $\un{\B}_{\E_L}$, as constructed in (\ref{eq:s-alpha}). The relevant structure map is given by smashing together maps
\[ \B(\one,\E_L,\E_L)(n_i) \smsh \K\E_L(n_i) \to \B(\one,\E_L,\E_L)(1) \smsh \dots \smsh \B(\one,\E_L,\E_L)(1) \isom S \]
that are dual to the left $\B\E_L$-comodule structure on $\B(\one,\E_L,\E_L)$ associated to the identity surjection $\un{n}_i \arrow{e,t}{=} \un{n}_i$.

It follows that the right $\K\E_L$-module structure on $\un{\one}(n,*) \homsmsh_{\E_L} \un{\B}_{\E_L}$ can be described, in the adjoint form, by the smash products of the maps
\[ \B(\one,\E_L,\one)(n_i) \to \Map(\K\E_L(n_i),\B(\one,\E_L,\one)(1) \smsh \B(\one,\E_L,\one)(1)) \isom \Map(\K\E_L(n_i),S). \]
These are just evaluation maps for the Spanier-Whitehead duality of the finite spectra $\B\E_L(n_i)$ and hence are equivalences. It follows that (\ref{eq:KE-1-map}) is an equivalence as required. This completes the proof of Proposition~\ref{prop:indec-triv}.
\end{proof}

We return now to the rest of the proof of Theorem~\ref{thm:topsp}. Together with (\ref{eq:dL-calc}), the equivalence of Proposition~\ref{prop:indec-triv} forms the required equivalences (\ref{eq:zeta^L}) of the form
\[ \zeta^L: \d^L[\c\Phi\A] \homeq \C_{\K\E_L}(\A) \]
for bounded symmetric sequences $\A$. Note that taking homotopy colimits over $L$ gives us an equivalence of $\C_{\K\E_\bullet}$-coalgebras $\C(\A) \homeq \C_{\K\E_\bullet}(\A)$.

Suppose now that $\A$ is an arbitrary symmetric sequence. We then have a zigzag of equivalences of $\C_{\K\E_\bullet}$-coalgebras
\[ \zeta: \hocolim_L \holim_n \d^L[\c\Phi\A_{\leq n}] \homeq \hocolim_L \holim_n \C_{\K\E_L}(\A_{\leq n}) \]
where we taken the homotopy limit along the truncation maps $\A_{\leq n} \to \A_{\leq (n-1)}$ and the homotopy colimit as $L \to \infty$. Now the right-hand side above is equivalent to $\C_{\K\E_\bullet}(\A)$ (because the homotopy limit commutes with $\C_{\K\E_L}$ and $\holim_n \A_{\leq n} \homeq \A$). For the left-hand side, we have a diagram (in the homotopy category of symmetric sequences)
\[ \begin{diagram}
  \node{\hocolim_L \Sigma^{-*L} \creff_*(\Phi\A)(S^L,\dots,S^L)} \arrow{s,l}{\sim} \arrow{e,t}{\sim}
    \node{\hocolim_L \holim_n \Sigma^{-*L} \creff_*(\Phi\A_{\leq n})(S^L,\dots,S^L)} \arrow{s,r}{\sim} \\
  \node{\hocolim_L \d^L[\c\Phi\A]} \arrow{e}
    \node{\hocolim_L \holim_n \d^L[\c\Phi\A_{\leq n}]}
\end{diagram} \]
in which the vertical maps are the equivalences of Proposition~\ref{prop:dF}. The top horizontal map is an equivalence because cross-effects commute with homotopy limits, so we deduce that the bottom horizontal map is an equivalence. It therefore follows that $\zeta$ is a zigzag of equivalences of $\C_{\K\E_\bullet}$-coalgebras
\[ \zeta: \C(\A) \homeq \C_{\K\E_\bullet}(\A). \]

Finally, we consider the following diagram (in the homotopy category)
\[ \begin{diagram}
  \node{\C(\A)} \arrow{s,lr}{\zeta}{\sim} \arrow{e}
    \node{\C_{\K\E_\bullet}(\C(\A))} \arrow{s,lr}{\zeta}{\sim} \arrow{e,t}{\epsilon_{\C}}
    \node{\C_{\K\E_\bullet}(\A)} \arrow{s,lr}{\zeta}{\sim} \\
  \node{\C_{\K\E_\bullet}(\A)} \arrow{e}
    \node{\C_{\K\E_\bullet}(\C_{\K\E_\bullet}(\A))} \arrow{e,t}{\epsilon_{\C_{\K\E_\bullet}}}
    \node{\C_{\K\E_\bullet}(\A)}
\end{diagram} \]
where the left-hand horizontal maps are given by the $\C_{\K\E_\bullet}$-coalgebra structure on $\C$ and $\C_{\K\E_\bullet}$ respectively.

The composite of the bottom row is the identity, and the composite of the top row is $\theta$. To show that $\theta$ is an equivalence, it is sufficient then to show that this diagram commutes in the homotopy category. The left-hand square commutes because $\zeta$ is a zigzag of equivalences of $\C_{\K\E_\bullet}$-coalgebras. To see that the right-hand square commutes (in the homotopy category) it is sufficient to show that
\[ \begin{diagram}
  \node{\C(\A)} \arrow[2]{s,l}{\zeta} \arrow{se,t}{\epsilon_{\C}} \\
    \node[2]{\A} \\
  \node{\C_{\K\E_\bullet}(\A)} \arrow{ne,b}{\epsilon_{\C_{\K\E_\bullet}}}
\end{diagram} \]
commutes up to homotopy. This follows by comparing the description of $\epsilon_{\C}$ given in \cite[4.11]{arone/ching:2014} with that of the counit for $\C_{\K\E_\bullet}$ in Definition~\ref{def:comonad-operad}. This completes the proof of Theorem~\ref{thm:topsp}.
\end{proof}

\begin{remark} \label{rem:thm-topsp}
Theorem~\ref{thm:topsp} provides an alternative proof of the result of \cite[6.1]{arone/ching:2014} that describes the comonad $\C$ up to homotopy in terms of `divided power' module structures for the operad $\der_*I$ formed from the derivatives of the identity on based spaces. For a bounded symmetric sequence $\A$ we have, using Lemma~\ref{lem:CKE-eq}:
\[ \begin{split} \C(\A)(k) &\homeq \C_{\K\E_\bullet}(\A)(k) \\
    &\homeq \hocolim_{L} \Wdge_{m = m_1+\dots+m_k} \left[ \B\E_L(m_1) \smsh \dots \smsh \B\E_L(m_k) \smsh \A(m) \right]_{h\Sigma_{m_1} \times \dots \Sigma_{m_k}} \\
    &\homeq  \Wdge_{m = m_1+\dots+m_k} \left[ \hocolim_L (\B\E_L(m_1) \smsh \dots \smsh \B\E_L(m_k)) \smsh \A(m) \right]_{h\Sigma_{m_1} \times \dots \Sigma_{m_k}} \\
    &\homeq  \Wdge_{m = m_1+\dots+m_k} \left[ \B\Com(m_1) \smsh \dots \smsh \B\Com(m_k) \smsh \A(m) \right]_{h\Sigma_{m_1} \times \dots \Sigma_{m_k}} \\
    &\homeq \prod_{m = m_1+\dots+m_k} \left[ \Map(\der_{m_1}I \smsh \dots \smsh \der_{m_k}I, \A(m)) \right]_{h\Sigma_{m_1} \times \dots \Sigma_{m_k}}
\end{split} \]
where the last equivalence is given by the equivalence $\der_*I \homeq \K\Com$ of \cite{ching:2005}.
\end{remark}

We showed in \cite{arone/ching:2014} that the functor $\d$ induces an equivalence between the homotopy theory of polynomial functors in $[\finbased,\spectra]$ and the homotopy theory of bounded $\C$-coalgebras. Using Theorem~\ref{thm:topsp} we can now replace $\C$ with $\C_{\K\E_\bullet}$ in that statement.

\begin{theorem} \label{thm:equiv-topsp-poly}
The functor $\d$ induces an equivalence
\[ [\finbased,\spectra]_{\mathsf{poly}}^{\mathsf{h}} \weq \Coalg_{\mathsf{b}}(\C_{\K\E_\bullet})^{\mathsf{h}} \]
between the homotopy category of polynomial pointed simplicial functors $\finbased \to \spectra$ and the homotopy category of bounded $\K\E_\bullet$-coalgebras of Definition~\ref{def:derived-map-coalgebras}. Moreover, for cofibrant polynomial functors $F,G \in [\finbased,\spectra]$, we have equivalences of mapping spectra
\begin{equation} \label{eq:nat-map} \Nat_{\finbased}(F,G) \homeq \widetilde{\Map}_{\C_{\K\E_\bullet}}(\d[F],\d[G]). \end{equation}
\end{theorem}
\begin{proof}
We prove (\ref{eq:nat-map}) first. This is made more difficult by the fact that the comonad $\C$ is not enriched in $\spectra$ and so we are unable to define mapping spectra for $\C$-coalgebras. Instead we have only mapping \emph{spaces} $\widetilde{\Hom}_{\C}(-,-)$ defined as in \cite[1.10]{arone/ching:2014}. We do then have a sequence
\[ \Omega^\infty \Nat_{\finbased}(F,G) \weq \widetilde{\Hom}_{\C}(\d[F],\d[G]) \weq \widetilde{\Hom}_{\C_{\K\E_\bullet}}(\d[F],\d[G]) = \Omega^\infty \widetilde{\Map}_{\C_{\K\E_\bullet}}(\d[F],\d[G]) \]
where the first equivalence is \cite[Corollary~3.15]{arone/ching:2014} and the second follows from Theorem~\ref{thm:topsp}. This implies that the map
\[ \Nat_{\finbased}(F,G) \to \widetilde{\Map}_{\C_{\K\E_\bullet}}(\d[F],\d[G]) \]
induces an isomorphism on homotopy groups $\pi_k$ for $k \geq 0$. For $k < 0$, we can apply the result just proved to see that
\[ \Sigma^{-k}\Nat_{\finbased}(F,G) \homeq \Nat_{\finbased}(F,\Sigma^{-k}G) \to \widetilde{\Map}_{\C_{\K\E_\bullet}}(\d[F],\d[\Sigma^{-k}G]) \homeq \Sigma^{-k} \widetilde{\Map}_{\C_{\K\E_\bullet}}(\d[F],\d[G]) \]
is an equivalence on $\pi_0$. Hence the map (\ref{eq:nat-map}) induces an isomorphism on $\pi_k$ for all $k$, so is an equivalence of spectra.

It now follows that $\d$ determines a fully faithful embedding of the homotopy theory of polynomial functors in $[\finbased,\spectra]$ into the homotopy theory of bounded $\C_{\K\E_\bullet}$-coalgebras. To complete the proof, we must show that every bounded $\C_{\K\E_\bullet}$-coalgebra $\A$ is in the image of this embedding.

Given a bounded $\C_{\K\E_\bullet}$-coalgebra $\A$, we define a functor $F_{\A} : \finbased \to \spectra$ by
\[ F_{\A}(X) := \widetilde{\Map}_{\C_{\K\E_\bullet}}(\d[R_X],\A) \isom \Tot(\Phi\C_{\K\E_\bullet}^{\bullet}\A). \]
We then claim that $F_{\A}$ is polynomial, and that $\d[\c F_{\A}]$ is equivalent to $\A$ in the homotopy category of $\C_{\K\E_\bullet}$-coalgebras.

If $\A$ is $N$-truncated, then so is $\C_{\K\E_\bullet}^r(\A)$ for any $r$. Therefore, each term in the cosimplicial object $\Phi\C_{\K\E_\bullet}^\bullet\A$ is $N$-excisive, by \cite[3.10]{arone/ching:2014}. It follows that the (fat) totalization $\Tot(\Phi\C_{\K\E_\bullet}^\bullet\A) = F_{\A}$ is also $N$-excisive.

Since $\d$ is a simplicial functor we have natural maps
\[ f_r: \Delta^r_+ \smsh \d[\c\Tot(\Phi\C_{\K\E_\bullet}^{\bullet}A)] \to \d\c\Phi \C_{\K\E_\bullet}^r\A \to \C_{\K\E_\bullet}^r\A \]
where the first map is the projection from the totalization, and the second is the counit associated to the comonad $\C = \d\c\Phi$. These commute with coface maps in the relevant way and determine a derived $\C_{\K\E_\bullet}$-coalgebra map $f: \d[\c F_{\A}] \to \A$ (in the sense of Definition~\ref{def:derived-map-coalgebras}).

By \cite[1.16]{arone/ching:2014}, it is sufficient to show that $f_0$ is an equivalence of symmetric sequences. We can write $f_0$, up to equivalence, as a composite
\[ \d[\c\Tot(\Phi\C_{\K\E_\bullet}^{\bullet}\A)] \to \Tot(\d\c\Phi\C_{\K\E_\bullet}^{\bullet}\A) \weq \Tot(\C_{\K\E_\bullet}^{\bullet+1}\A) \weq \A \]
where the second map is induced by the equivalence of Theorem~\ref{thm:topsp} and the third is the coaugmentation equivalence associated to the extra codegeneracies in the cosimplicial object $\C_{\K\E_\bullet}^{\bullet+1}\A$. It remains then to show that the first map is an equivalence, that is: $\d$ commutes with the totalization of the cosimplicial object $\Phi\C_{\K\E_\bullet}^{\bullet}\A$. The proof of this is virtually identical to that of \cite[3.16]{arone/ching:2014}, using the equivalence $\theta: \d\c\Phi \weq \C_{\K\E_\bullet}$ of Theorem~\ref{thm:topsp}.
\end{proof}

\subsection{Comparison of classifications of polynomial functors}

We take this opportunity to describe how the result of Theorem~\ref{thm:equiv-topsp-poly} is related to other approaches to the classification of polynomial functors from based spaces to spectra. We have already mentioned the work of Dwyer and Rezk in which a polynomial functor $F: \finbased \to \spectra$ corresponds to the bounded $\Com$-comodule $\N[F]$. One half of this equivalence was observed in Proposition~\ref{prop:F-eq} where we showed that $F$ can be recovered from $\N[F]$ via a homotopy coend with $\Sigma^\infty X^{\smsh *}$. In Theorem~\ref{thm:classification} below we show the other half: that any bounded $\Com$-comodule is equivalent to $\N[F]$ for some polynomial functor $F$.

Another classification of polynomial functors from based spaces to spectra follows from the observation that such a functor $F$ is determined, by left Kan extension, by its values on the full subcategory of finite pointed sets. In fact, an $n$-excisive $F$ is determined by its values on sets of cardinality at most $n$ (not including the basepoint). This fact can be related to the classification described in the previous paragraph via a homotopical version of an equivalence of Pirashvili \cite{pirashvili:2000}.

\begin{definition}
Let $\Omega_{\leq n}$ denote the category whose objects are the non-empty finite sets of cardinality at most $n$, and whose morphisms are the surjections. Notice that an $n$-truncated $\Com$-comodule can be identified with a functor $\Omega_{\leq n} \to \spectra$. Also, let $\Gamma_{\leq n}$ denote the category whose objects are the pointed finite sets of cardinality at most $n$ (not including the basepoint), and whose morphisms are functions that preserve the basepoint.
\end{definition}

In \cite{pirashvili:2000} Pirashvili showed, among other things, that there is an equivalence of categories
\[ [\Omega_{\leq n},\mathscr{A}b]_{\mathsf{unptd}} \homeq [\Gamma_{\leq n},\mathscr{A}b] \]
between the category of all functors from $\Omega_{\leq n}$ to the category of abelian groups $\mathscr{A}b$, and the category of \emph{pointed} functors from $\Gamma_{\leq n}$ to $\mathscr{A}b$ (i.e. those that send the one-point set to the zero group). In fact, if one includes the empty set as an object of $\Omega_{\leq n}$, then one can remove the pointedness restriction on the right-hand side.

A homotopical version of this equivalence for \emph{contravariant} functors was constructed by Helmstutler in \cite{helmstutler:2008} but the covariant version does not appear to be in the literature, so we prove it now. This result includes the case $n = \infty$ where there is no restriction, beyond finiteness, on the cardinality of the sets involved.

\begin{theorem} \label{thm:pirashvili}
For each $1 \leq n \leq \infty$, there is a Quillen equivalence of the form
\begin{equation} \label{eq:pirashvili} \L: \Comod_{\leq n}(\Com) = [\Omega_{\leq n},\spectra]_{\mathsf{unptd}} \rightleftarrows [\Gamma_{\leq n},\spectra] : \R \end{equation}
where $[\Omega_{\leq n},\spectra]_{\mathsf{unptd}}$ is the category of all functors $\Omega_{\leq n} \to \spectra$ (the subscript is there to emphasize that these functors need not be pointed), and $[\Gamma_{\leq n},\spectra]$ is the category of \emph{pointed} functors $\Gamma_{\leq n} \to \spectra$ (i.e. those which take the one-element set to the trivial spectrum. Each of these functor categories has the projective model structure in which weak equivalences and fibrations are detected objectwise, and the Quillen functors are given by
\[ \L(\N)(J_+) := \N \smsh_{\Com} \widetilde{\Sigma^\infty (J_+)^{\smsh *}} \]
and
\[ \R(G) := \Map_{J_+ \in \Gamma_{\leq n}}(\widetilde{\Sigma^\infty (J_+)^{\smsh *}}, G(J_+)). \]
Here $\widetilde{\Sigma^\infty (J_+)^{\smsh *}}$ is as in Definition~\ref{def:SigmainftyX-Com}.
\end{theorem}
\begin{proof}
First note that for each $k$, the functor $\Gamma_{\leq n} \to \spectra$ given by
\[ J_+ \mapsto \widetilde{\Sigma^\infty (J_+)^{\smsh k}} \]
is cofibrant in the projective model structure. It follows that $\R$ preserves fibrations and trivial fibrations, so $(\L,\R)$ is a Quillen adjunction.

Next note that for any $J_+ \in \Gamma_{\leq n}$, there is an isomorphism of $\Com$-modules
\[ \Sigma^\infty (J_+)^{\smsh *} \isom \Wdge_{\emptyset \neq K \subseteq J} \un{\Com}(*,K). \]
This can be seen by identifying $(J_+)^{\smsh I}$ with the set of maps $I \to J$, with a disjoint basepoint. Each such map, at least when $I$ is nonempty, factors uniquely as $I \epi K \into J$ for some nonempty subset $K \subseteq J$. For any $\Com$-comodule $\N$, there is then a chain of natural equivalences (of functors $\Gamma_{\leq n} \to \spectra$) of the form
\begin{equation} \label{eq:pirashvili1} \L\N(J_+) \homeq \Wdge_{\emptyset \neq K \subseteq J} \N \homsmsh_{\Com} \un{\Com}(*,K) \homeq \Wdge_{\emptyset \neq K \subseteq J} \N(K). \end{equation}
It also follows from this that a morphism $\N \to \N'$ of $n$-truncated $\Com$-comodules is an equivalence if and only if the induced map
\[ \L\N \to \L\N' \]
is an objectwise equivalence.

Now note that there is also a natural isomorphism (of functors $\Gamma_{\leq n} \to \spectra$) of the form
\[ \Sigma^\infty \Hom_{\Gamma_{\leq n}}(J_+,I_+) \isom \Wdge_{\emptyset \neq K \subseteq J} \Sigma^\infty (I_+)^{\smsh K}. \]
This can be seen by noting that any pointed map $f: J_+ \to I_+$ that is not constant to the basepoint corresponds uniquely to a map $K \to I$ where $K = f^{-1}(I)$ is a non-empty subset of $J$. For any $G: \Gamma_{\leq n} \to \spectra$, there is then a chain of equivalences
\begin{equation} \label{eq:pirashvili2} \Wdge_{\emptyset \neq K \subseteq J} \R G(K) \homeq \Map_{I_+ \in \Gamma_{\leq n}}(\Sigma^\infty \Hom_{\Gamma_{\leq n}}(J_+,I_+),G(I_+)) \isom G(J_+). \end{equation}
Combining this with (\ref{eq:pirashvili1}) we see that the derived counit of the adjunction is an equivalence, that is
\[ \L \R G \weq G. \]
It then follows that for $\N \in \Comod_{\leq n}(\Com)$, we have an equivalence $\L \R \L \N \weq \L \N$ and so, by the triangle identity, the map
\[ \L \N \to \L \R \L \N \]
given by applying $\L$ to the derived unit of the adjunction, is also an equivalence. By our previous calculation this implies that the derived unit itself is an equivalence. Since the derived unit and counit are both equivalences, we have a Quillen equivalence.
\end{proof}

\begin{theorem} \label{thm:classification}
There is a diagram of functors, as follows, which commutes up to natural equivalence and in which each functor determines an equivalence of homotopy categories:
\begin{equation} \label{eq:classification} \begin{diagram} \dgARROWLENGTH=4em
  \node[2]{\Coalg_{\leq n}(\C_{\K\E_\bullet})} \\
  \node{\Comod_{\leq n}(\Com)} \arrow{ne,t}{\hocolim_L - \homsmsh_{\E_L} \un{\B}_{\E_L}} \arrow[2]{e,t}{- \smsh_{\Com} \widetilde{\Sigma^\infty X^{\smsh *}}} \arrow{se,b}{\L}
    \node[2]{[\finbased,\spectra]_{n-\mathsf{exc}}} \arrow{nw,t}{\d} \\
  \node[2]{[\Gamma_{\leq n},\spectra]}  \arrow{ne,b}{\mathsf{LKan}}
\end{diagram} \end{equation}
in which
\begin{itemize}
  \item $\Coalg_{\leq n}(\C_{\K\E_\bullet})$ is the category of $n$-truncated $\C_{\K\E_\bullet}$-coalgebras;
  \item $\Comod_{\leq n}(\Com)$ is the category of $n$-truncated $\Com$-comodules;
  \item $[\finbased,\spectra]_{n-\mathsf{exc}}$ is the category of $n$-excisive pointed simplicial functors $\finbased \to \spectra$.
  \item $[\Gamma_{\leq n},\spectra]$ is as in Theorem~\ref{thm:pirashvili};
\end{itemize}
and
\begin{itemize}
  \item $\L$ is the left adjoint of the Quillen equivalence in Theorem~\ref{thm:pirashvili};
  \item $\mathsf{LKan}$ denotes left Kan extension along the inclusion $\Gamma_{\leq n} \to \finbased$;
  \item $\d$ is as in Definition~\ref{def:dF}.
\end{itemize}
\end{theorem}
\begin{proof}
Theorem~\ref{thm:equiv-topsp-poly} implies that $\d$ determines an equivalence of homotopy theories. The top triangle commutes up to natural equivalence by the construction of $\d$. We turn then to the horizontal functor in diagram (\ref{eq:classification}).

We have already shown in Proposition~\ref{prop:F-eq} that the evaluation map
\[ \Nat_{Y \in \finbased}(\widetilde{\Sigma^\infty Y^{\smsh *}},FY) \homsmsh_{\Com} \widetilde{\Sigma^\infty X^{\smsh *}} \to F(X) \]
is an equivalence when $F$ is polynomial. We now show that the unit map
\[ \N \to \Nat_{Y \in \finbased}(\widetilde{\Sigma^\infty Y^{\smsh *}}, \N \homsmsh_{\Com} \widetilde{\Sigma^\infty Y^{\smsh *}}) \]
is an equivalence when $\N$ is a cofibrant $\Com$-comodule. To see this, note that by Corollary~\ref{cor:N-hocolim}, it it sufficient to show that, for each $r$, the map
\[ \N(r) \to \N \homsmsh_{\Com} \Nat_{Y \in \finbased}(\widetilde{\Sigma^\infty Y^{\smsh r}},\widetilde{\Sigma^\infty Y^{\smsh *}}). \]
induced by the identity on $\widetilde{\Sigma^\infty Y^{\smsh *}}$, is an equivalence. But by (\ref{eq:creff-X^*}) there is an equivalence
\[ \Nat_{Y \in \finbased}(\widetilde{\Sigma^\infty Y^{\smsh r}},\widetilde{\Sigma^\infty Y^{\smsh *}}) \homeq \un{\Com}(*,r) \]
and the map
\[ \N(r) \to \N \homsmsh_{\Com} \un{\Com}(*,r) \]
is an equivalence because the right-hand side is a simplicial object with extra degeneracies. It now follows that the horizontal map in (\ref{eq:classification}) induces an equivalence of homotopy theories. It then also follows that the top-left map induces an equivalence.

By Theorem~\ref{thm:pirashvili}, $\L$ also induces an equivalence, so it only remains to show that the bottom triangle commutes up to natural equivalence. Notice that we have a commutative diagram
\[ \begin{diagram} \dgARROWLENGTH=4em
  \node{\Comod_{\leq n}(\Com)} \arrow[2]{e,t}{- \smsh_{\Com} \widetilde{\Sigma^\infty X^{\smsh *}}} \arrow{se,b}{- \smsh_{\Com} \widetilde{\Sigma^\infty (J_+)^{\smsh *}}}
    \node[2]{[\finbased,\spectra]_{n-\mathsf{exc}}} \arrow{sw,b}{\mathsf{res}} \\
  \node[2]{[\Gamma_{\leq n},\spectra]}
\end{diagram} \]
where $\mathsf{res}$ is restriction to the subcategory $\Gamma_{\leq n} \subseteq \finbased$. Since the other two functors induce equivalences, it follows that $\mathsf{res}$ does too. Since $\mathsf{LKan}$ is left adjoint to $\mathsf{res}$, it induces the inverse equivalence to $\mathsf{res}$ on the homotopy category and hence the bottom triangle in the original diagram commutes up to natural equivalence.
\end{proof}

\subsection{Classification of analytic functors from based spaces to spectra} \label{sec:topsp-ana}

One of the main advantages of using $\C_{\K\E_\bullet}$-coalgebras to classify polynomial functors from based spaces to spectra (over the other categories appearing in Theorem~\ref{thm:classification}) is that this approach generalizes, to some extent, to non-polynomial functors.

Here we consider functors that are `analytic at the one-point space $*$' in the sense described below. This condition is weaker than Goodwillie's notion of analyticity since it only concerns the values of a functor in a `neighbourhood' of a $*$, that is, only on highly connected spaces.

\begin{definition} \label{def:analytic}
We say that a homotopy functor $F: \finbased \to \spectra$ \emph{satisfies condition $E^*_n(c,\kappa)$} if: for any strongly cocartesian $(n+1)$-cube $\mathcal{X} : \mathcal{P}(S) \to \finbased$ such that for any $s \in S$, $\mathcal{X}(\emptyset) \to \mathcal{X}(s)$ is a $k_s$-connected map between $\kappa$-connected spaces, the cube $F(\mathcal{X})$ is $(-c+\sum k_s)$-cartesian.

We say that $F: \finbased \to \spectra$ is \emph{$\rho$-analytic at $*$} if there is a constant $q$ such that $F$ satisfies
\[ E^*_n(n\rho-q,\rho+1) \]
for all $n \geq 1$. We say that $F$ is \emph{analytic at $*$} if it is $\rho$-analytic at $*$ for some number $\rho$.
\end{definition}

\begin{remark}
These conditions should be compared to those of Goodwillie \cite[4.1]{goodwillie:1991}. Since a map between $\kappa$-connected spaces is $\kappa$-connected, it is easy to see that Goodwillie's condition $E_n(c,\kappa)$ implies our $E^*_n(c,\kappa)$. It follows that a functor that is $\rho$-analytic in the sense of Goodwillie is, in particular, $\rho$-analytic at $*$. More generally, we can say that $F$ is $\rho$-analytic at $X$, for some based space $X$, if it satisfies the analyticity condition on cubes of spaces that are $(\rho+1)$-connected over $X$.
\end{remark}

\begin{definition} \label{def:symseq-analytic}
Let $\A$ be a symmetric sequence of spectra. We say that $\A$ is \emph{$\rho$-analytic} if there is a constant $c$ such that $\A(n)$ is $(-\rho n + c)$-connected for all $n$. We say $\A$ is \emph{analytic} if it is $\rho$-analytic for some $\rho$.
\end{definition}

\begin{lemma} \label{lem:ana1}
Let $F: \finbased \to \spectra$ be a functor that is $\rho$-analytic at $*$. Then the symmetric sequence $\der_*F$ is $\rho$-analytic in the sense of Definition~\ref{def:symseq-analytic}.
\end{lemma}
\begin{proof}
For each $n \geq 2$, we apply the condition $E^*_{n-1}((n-1)\rho-q,\rho+1)$ to the strongly-cocartesian $n$-cube with initial maps $* \to S^L$ for $L \geq \rho+2$. These maps are $(L-1)$-connected and so we deduce that the cube $T \mapsto F( \Wdge_{T} S^L )$ is $(q-(n-1)\rho+n(L-1))$-cartesian. It follows that the total homotopy cofibre of this cube of spectra, which is precisely the \ord{n} co-cross-effect $\creff^nF(S^L,\dots,S^L)$, is $(q-(n-1)\rho+n(L-1)+n-1)$-connected.

Since the cross-effect is equivalent to the co-cross-effect, we deduce that the desuspended cross-effect
\[ \Sigma^{-nL}\creff_nF(S^L,\dots,S^L) \]
is $(-n\rho + \rho + q - 1)$-connected. Taking the homotopy colimit as $L \to \infty$, we deduce that $\der_nF$ has this same connectivity, and so the symmetric sequence $\der_*F$ is $\rho$-analytic (with the constant $c$ in Definition~\ref{def:symseq-analytic} equal to $\rho + q - 1$).
\end{proof}

\begin{lemma} \label{lem:ana2}
Consider a sequence of functors
\begin{equation} \label{eq:seqFn} \dots \to F_n \to F_{n-1} \to \dots \to F_1 \end{equation}
that looks like a Taylor tower in the sense it induces equivalences $P_n(F_{n+1}) \weq P_n(F_n) \homeq F_n$ for each $n$. Suppose that the symmetric sequence $\{\der_n(F_n)\}$ is $\rho$-analytic for some $\rho \geq 0$. Then
\[ F := \holim_n F_n \]
is $\rho$-analytic at $*$, and the sequence (\ref{eq:seqFn}) is equivalent to the Taylor tower of $F$.
\end{lemma}
\begin{proof}
Let us write $D_n$ for the homotopy fibre of the natural transformation $F_n \to F_{n-1}$. Then we have
\[ D_n(X) \homeq (\A_n \smsh X^{\smsh n})_{h\Sigma_n} \]
where $\A_n := \der_n(F_n)$. Choose a constant $c$ such that $\A_n$ is $(-\rho n + c)$-connected for all $n$.

To show that $F$ is $\rho$-analytic at $*$, consider a strongly cocartesian $(n+1)$-cube $\mathcal{X}$ of $(\rho+1)$-connected based spaces with each initial map $\mathcal{X}(\emptyset) \to \mathcal{X}(s)$ being $k_s$-connected. In particular, then, the cube $(\Sigma^{-\rho} \Sigma^\infty \mathcal{X})$ is a strongly-cocartesian cube of $1$-connected spectra with each initial map being $(-\rho + k_s)$-connected. We can apply a version of \cite[4.4]{goodwillie:1991} for cubes of spectra to deduce that, for each $m$, the cube $(\Sigma^{-\rho m} \Sigma^\infty X^{\smsh m})$ is $(-\rho(n+1) + \sum k_s)$-cartesian.

Since $\A_m$ is $(-\rho m + c)$-connected, it therefore follows that the cube $D_m(\mathcal{X})$ is
\[ (\rho m - \rho m + c -\rho(n+1) + \sum k_s) = (-\rho n - \rho + c + \sum k_s) \]
-cartesian. By induction on $m$, using the fibre sequences $D_m \to F_m \to F_{m-1}$ we deduce that each cube $F_m(\mathcal{X})$ is $(- \rho n - \rho + c + \sum k_s)$-cartesian, and so the homotopy limit $F(\mathcal{X})$ is $(- \rho n - \rho + c - 1 + \sum k_s)$-cartesian. This verifies that $F$ is $\rho$-analytic at $*$ with the number $q$ in Definition~\ref{def:analytic} equal to $c - \rho - 1$.

Now suppose that $X$ is a $k$-connected based space. Then $D_n(X) = (\A_n \smsh X^{\smsh n})_{h\Sigma_n}$ is $[(k-\rho)n+c]$-connected. As long as $k \geq \rho$, it follows that the map $F(X) \to F_n(X)$ is $[(k-\rho)(n+1)+c-1]$-connected, for each $n$. This means that $F$ and $F_n$ agree to order $n$ in the sense of Goodwillie \cite[1.2]{goodwillie:2003} from which it follows by \cite[1.6]{goodwillie:2003} that $P_nF \homeq P_nF_n \homeq F_n$.
\end{proof}

\begin{corollary} \label{cor:analytic}
If $F: \finbased \to \spectra$ is analytic at $*$, then $P_\infty F := \holim_n P_nF$ is analytic at $*$ and the map $F \to P_\infty F$ induces an equivalence of Taylor towers.
\end{corollary}

\begin{definition} \label{def:strongly-analytic}
Let us say that a functor $F: \finbased \to \spectra$ is \emph{strongly-analytic at $*$} if $F$ is analytic at $*$ and the map $F \to P_\infty F := \holim P_nF$ is an equivalence. It follows from Corollary~\ref{cor:analytic} that, for any $F: \finbased \to \spectra$ that is analytic at $*$, the functor $P_\infty F$ is strongly-analytic at $*$, and the map $F \to P_\infty F$ determines an equivalence of Taylor towers. We can therefore think of the category of functors strongly-analytic at $*$ as a model for the category of all functors that are analytic at $*$ up to equivalence of Taylor towers.

We write $[\finbased,\spectra]_{\mathsf{an}(*)}^{\mathsf{h}}$ for the subcategory of the homotopy category $[\finbased,\spectra]^{\mathsf{h}}$ consisting of the functors that are strongly-analytic at $*$.
\end{definition}

Now consider the restriction of $\d: [\finbased,\spectra] \to \Coalg(\C_{\K\E_\bullet})$ to strongly-analytic functors. Firstly, the argument of Theorem~\ref{thm:equiv-topsp-poly} implies that $\d$ determines weak equivalences
\[ \Nat_{\finbased}(F,G) \weq \widetilde{\Map}_{\C_{\K\E_\bullet}}(\d[F],\d[G]) \]
for any $F,G \in [\finbased,\spectra]$ such that $G \homeq P_\infty G$. It follows that $\d$ is a fully-faithful embedding of $[\finbased,\spectra]_{\mathsf{an}(*)}^{\mathsf{h}}$ into the homotopy category of $\C_{\K\E_\bullet}$-coalgebras. We have not been able to explicitly identify the image of this embedding. Instead, however, by considering the `pro-truncated' version of the homotopy theory of $\C_{\K\E_\bullet}$-coalgebras, as constructed in Definition~\ref{def:truncated-category-coalgebras}, we have the following result.

\begin{theorem} \label{thm:equiv-topsp-ana}
The functor $\d$ sets up an equivalence
\[ [\finbased,\spectra]_{\mathsf{an}(*)}^{\mathsf{h}} \homeq \Coalg^{\mathsf{t}}_{\mathsf{an}}(\C_{\K\E_\bullet})^{\mathsf{h}} \]
between the homotopy category of functors $\finbased \to \spectra$ that are strongly-analytic at $*$, and the pro-truncated homotopy category of analytic $\C_{\K\E_\bullet}$-coalgebras. Moreover, for cofibrant $F,G \in [\finbased,\spectra]_{\mathsf{an}(*)}$, we have equivalences
\[ \Nat_{\finbased}(F,G) \weq \widetilde{\Map}_{\C_{\K\E_\bullet}}(\d[F],\d[G]) \homeq \widetilde{\Map}^{\mathsf{t}}_{\C_{\K\E_\bullet}}(\d[F],\d[G]). \]
\end{theorem}
\begin{proof}
For two functors $F,G$ strongly-analytic at $*$ we have already noted that there is an equivalence of spectra
\[ \Nat_{\finbased}(F,G) \weq \widetilde{\Map}_{\C_{\K\E_\bullet}}(\d[F],\d[G]). \]
Since the map
\[ \Nat_{\finbased}(F,G) \to \holim_n \Nat_{\finbased}(F,P_nG) \]
is an equivalence, so too is the map
\[ \widetilde{\Map}_{\C_{\K\E_\bullet}}(\d[F],\d[G]) \to \holim_n \widetilde{\Map}_{\C_{\K\E_\bullet}}(\d[F],\d[G]_{\leq n}) = \widetilde{\Map^{\mathsf{t}}}_{\C_{\K\E_\bullet}}(\d[F],\d[G]). \]
So, between $\C_{\K\E_\bullet}$-coalgebras of the form $\d[G]$ where $G$ is strongly-analytic at $*$, pro-truncated mapping spectra are equivalent to ordinary mapping spectra. In particular, then, the functor $\d$ is an embedding of the homotopy category of strongly-analytic functors into the pro-truncated homotopy category of analytic $\C_{\K\E_\bullet}$-coalgebras.

To see that this embedding is an equivalence, take an arbitrary analytic $\C_{\K\E_\bullet}$-coalgebra $\A$. Let $F$ be the functor given by
\[ F := \holim_n F_{\A_{\leq n}} \]
where $F_{\A_{\leq n}}$ is as in the proof Theorem~\ref{thm:topsp}. Then $F$ is strongly-analytic at $*$ by Lemma \ref{lem:ana2}.

We claim that there is a derived equivalence (in the pro-truncated sense of Definition~\ref{def:truncated-map-coalgebras}) of $\C_{\K\E_\bullet}$-coalgebras $f: \d[\c F] \to \A$. This consists of the maps
\[ f_{k,n}: \Delta^k_+ \smsh \d[\c \holim_n F_{\A_{\leq n}}]  \to \Delta^k_+ \smsh \d[\c F_{\A_{\leq n}}] \to \C_{\K\E_\bullet}^k(\A_{\leq n}) \]
constructed by projecting from the homotopy limit, together with the corresponding maps from the proof of Theorem~\ref{thm:topsp}. It is sufficient then to show that the composite
\[ f_0 : \d[\c \holim_n F_{\A_{\leq n}}] \to \holim_n \d[\c F_{\A_{\leq n}}] \to \holim_n \A_{\leq n} \homeq \A \]
is an equivalence. The first map is an equivalence by Lemma~\ref{lem:ana2}, and the second map is an equivalence by Theorem~\ref{thm:topsp}, so the required result follows.
\end{proof}

\begin{examples}
Under the equivalence of Theorem~\ref{thm:equiv-topsp-ana}, we can identify the analytic functors with split Taylor tower with those $\C_{\K\E_\bullet}$-coalgebras $\A$ whose structure is induced via the map of comonads (of symmetric sequences)
\[ I_{\symseq} = \C_{\one} = \C_{\K\E_0} \to \hocolim \C_{\K\E_L} = \C_{\K\E_\bullet}. \]
More generally, it is possible to characterize those functors whose corresponding $\C_{\K\E_\bullet}$-coalgebra is induced via the map of comonads
\[ \C_{\K\E_L} \to \C_{\K\E_\bullet} \]
for some fixed $L$, i.e. those functors whose derivatives possess a $\K\E_L$-module structure from which the Taylor tower can be constructed. We leave this analysis to a future paper.
\end{examples}

\section{Functors from spectra to spectra} \label{sec:spsp}

We now turn to pointed simplicial functors $\finspec \to \spectra$, where $\finspec$ is the category of finite cell spectra. We denote the category of such functors by $[\finspec,\spectra]$. As with the spaces to spectra case there is an inverse sequence of operads that acts on the sequence of partially-stabilized cross-effects of a functor $F \in [\finspec,\spectra]$. We start by constructing this sequence.

\subsection{Desuspensions of the commutative operad}

Underlying the sequence of operads we are interested in is a notion of `desuspension' for operads of spectra. This uses a certain cooperad $\mathbb{S}$, in the category of based spaces, whose terms are homeomorphic to spheres. The structure maps for this cooperad are homeomorphisms and so $\mathbb{S}$ is also an operad in a canonical way. This operad is isomorphic to the operad $S_\infty$ of \cite{arone/kankaanrinta:2014}.

\begin{definition} \label{def:spheres}
For a nonempty finite set $I$ we set
\[ \mathbb{R}^I_0 := \{t \in \mathbb{R}^I \; | \; \min_{i \in I} t_i = 0 \}. \]
We then write
\[ \mathbb{S}(I) := (\mathbb{R}^I_0)^+ \]
This is the based space given by the one-point compactification of $\mathbb{R}^I_0$. Note that $\mathbb{S}(I)$ is homeomorphic to $S^{|I|-1}$. The permutation action of the symmetric group $\Sigma_I$ on $\mathbb{R}^I$ restricts to an action on $\mathbb{R}^I_0$ and hence on $\mathbb{S}(I)$. This makes $\mathbb{S}$ into a symmetric sequence of based spaces.
\end{definition}

\begin{definition} \label{def:sphere-cooperad}
We construct a cooperad structure on the symmetric sequence $\mathbb{S}$ as follows. It arises from a cooperad structure on the unbased spaces $\mathbb{R}^n_0$ of Definition \ref{def:spheres}. For a surjection of nonempty finite sets $\alpha: I \epi J$, writing $I_j := \alpha^{-1}(j)$, there is a continuous map
\[ d_{\alpha}: \mathbb{R}^I_0 \to \mathbb{R}^J_0 \times \prod_{j \in J} \mathbb{R}^{I_j}_0 \]
given by
\[ d_{\alpha} \left((t_i)_{i \in I} \right) := \left( (v_j)_{j \in J}, \{(u_{j,i})_{i \in I_j}\}_{j \in J} \right) \]
where
\[ v_j := \min_{i \in I_j} t_i \]
and
\[ u_{j,i} := t_i - v_j, \quad \text{for $i \in I_j$}. \]
The condition that the minimum of the $t_i$ is zero implies that the minimum of the $v_j$ is zero. For each $j \in J$, we clearly have $\min_{i \in I_j} u_{j,i} = 0$.

The functions $d_\alpha$ extend to the one-point compactifications to give maps
\[ d_{\alpha}: \mathbb{S}(I) \to \mathbb{S}(J) \smsh \Smsh_{j \in J} \mathbb{S}(I_j). \]
\end{definition}

\begin{lemma} \label{lem:sphere-cooperad}
Each map $d_{\alpha}$ of Definition~\ref{def:sphere-cooperad} is a homeomorphism, and together they make $\mathbb{S}$ into a reduced cooperad of based spaces. (The inverses $d_{\alpha}^{-1}$ thus make $\mathbb{S}$ into a reduced operad.)
\end{lemma}
\begin{proof}
We define an inverse to $d_{\alpha}$ on non-basepoints by
\[ d_{\alpha}^{-1} \left( (v_j)_{j \in J} , \{(u_{j,i})_{i \in I_j}\}_{j \in J} \right) := (t_i)_{i \in I} \]
where
\[ t_i := u_{\alpha(i),i} + v_{\alpha(i)}. \]
The map $d_{\alpha}^{-1}$ is continuous and inverse to $d_{\alpha}$. We leave the reader to check the unit and associativity conditions for $\mathbb{S}$ to be a cooperad.
\end{proof}

\begin{definition}
Let $\A$ be a symmetric sequence of spectra. We define the \emph{desuspension} of $\A$ to be the symmetric sequence $\S^{-1}\A$ given by
\[ (\S^{-1}\A)(n) := \Map(\mathbb{S}(n),\A(n)) \]
with the diagonal $\Sigma_n$-action. Here $\Map(-,-)$ denotes the cotensoring of $\spectra$ over based spaces.

Now suppose that $\P$ is an operad of spectra. Then we make $\S^{-1}\P$ into an operad of spectra by convolving the operad structure on $\P$ with the cooperad structure on $\mathbb{S}$ in the following way. For a surjection $\alpha: I \epi J$ we have structure map
\[ \begin{split} \Map(\mathbb{S}(J),\P(J)) \smsh \Smsh_{j \in J} \Map(\mathbb{S}(I_j),\P(I_j)) &\to \Map(\mathbb{S}(J) \smsh \Smsh_{j \in J} \mathbb{S}(I_j), \P(J) \smsh \Smsh_{j \in J} \P(I_j)) \\ &\to \Map(\mathbb{S}(I),\P(I)) \end{split} \]
where the first map is the canonical smash product of mapping spectra and the second employs the cooperad structure on $\mathbb{S}$ and the operad structure on $\P$. We refer to $\S^{-1}\P$ with this structure as the \emph{(operadic) desupension} of $\P$.

We can iterate the desuspension process to get operads $\S^{-L}\P$ given by
\[ \S^{-L}\P(I) \isom \Map(\mathbb{S}(I)^{\smsh L},\P(I)). \]
\end{definition}

\begin{definition}
For each nonempty finite set $I$, we define a map
\[ \eta_I: S^0 \to \mathbb{S}(I) \]
by sending the non-basepoint in $S^0$ to the point in $\mathbb{S}(I)$ given by the origin $0 \in \mathbb{R}^I_0$.
\end{definition}

\begin{lemma} \label{lem:map-sphere-cooperad}
The maps $\eta_I$ together form a morphism of cooperads (of based spaces)
\[ \eta: \mathbbm{Com}_+^c \to \mathbb{S} \]
where $\mathbbm{Com}_+^c$ is the commutative cooperad in based spaces (given by $S^0$ in each term with trivial $\Sigma_n$-actions and homeomorphisms as structure maps).
\end{lemma}
\begin{proof}
It is sufficient to check that for each surjection $\alpha: I \epi J$, we have
\[ d_{\alpha}(0) = \left(0, \{0\}_{j \in J} \right) \]
which follows from the definition of $d_{\alpha}$.
\end{proof}

\begin{definition} \label{def:desuspension-sequence}
For any operad $\P$ of spectra, the map $\eta$ of Lemma \ref{lem:map-sphere-cooperad} induces a natural map of operads
\[ \eta: \S^{-1}\P \to \P \]
Iterating this construction, we get, for each operad $P$ of spectra, an inverse sequence
\[ \S^{-\bullet}\P : \quad \dots \to \S^{-2}\P \to \S^{-1}\P \to \P. \]
\end{definition}

The example we are concerned with in this paper is the sequence of desuspensions of $\P = \Com$, the commutative operad of spectra.

\begin{definition} \label{def:sphere-comonad}
We write $\S^{-L}$ for a (functorial) cofibrant replacement (in the projective model structure on operads) of $\S^{-L}\Com$. Applying this cofibrant replacement to the sequence in Definition~\ref{def:desuspension-sequence} gives us an inverse sequence of operads
\[ \S_\bullet: \quad \dots \to \S^{-2} \to \S^{-1} \to \S^{0} \weq \Com. \]
Associated to this sequence, by the construction of Definition~\ref{def:inverse-operad-comonad}, is a cooperad $\C_{\S_\bullet}$ on the category of symmetric sequences given by
\[ \C_{\S_\bullet}(\A)(k) := \hocolim_{L} \C_{\S^{-L}}(\A)(k) = \hocolim_{L} \prod_{n} \left[ \prod_{\un{n} \epi \un{k}} \Map(\S^{-L}(n_1) \smsh \dots \smsh \S^{-L}(n_k), \A(n)) \right]^{\Sigma_n}. \]
\end{definition}

\subsection{Models for the derivatives of a functor from spectra to spectra}

Our next task is to construct models for the partially-stabilized cross-effects of a pointed simplicial functor $F \in [\finspec,\spectra]$ that admit the structure of a $\S^{-L}$-module. From these we obtain models for the derivatives of $F$ that form a $\C_{\S_\bullet}$-coalgebra. We focus first on the representable functors.

\begin{definition} \label{def:rep-spsp}
Let $X \in \finspec$. The \emph{functor represented by $X$} is the pointed simplicial functor $R_X: \finspec \to \spectra$ given by
\[ R_X(Y) := \Sigma^\infty \Hom_{\spectra}(X,Y). \]
\end{definition}

A significant difference from the case $[\finbased,\spectra]$ is that there are simple models for the cross-effects of the representable functors from spectra to spectra.

\begin{lemma} \label{lem:creff-spsp}
The \ord{n} cross-effect of $R_X$ is given by
\[ \creff_nR_X(Y_1,\dots,Y_n) \homeq \Sigma^\infty \Hom_{\spectra}(X,Y_1) \smsh \dots \smsh \Hom_{\spectra}(X,Y_n). \]
\end{lemma}
\begin{proof}
Finite coproducts of spectra are equivalent to the corresponding products. Therefore, the co-cross-effect (and hence the desired cross-effect) is given by the total homotopy cofibre
\[ \thocofib_{S \subseteq \un{n}} \left\{ \Sigma^\infty \Hom_{\spectra} \left( X, \prod_{i \in S} Y_i \right) \right\} \]
which is equivalent to
\[  \Sigma^\infty \thocofib_{S \subseteq \un{n}} \left\{ \prod_{i \in S} \Hom_{\spectra} \left( X, Y_i \right) \right\}. \]
The total homotopy cofibre of the cube of simplicial sets whose terms are $\prod_{i \in S} Z_i$ is equivalent to the smash product $Z_1 \smsh \dots \smsh Z_n$. We thus obtain the required formula for the cross-effect.
\end{proof}

\begin{corollary} \label{cor:der-hocolim-spsp}
The partially-stabilized cross-effects of $R_X$ are therefore given by
\[ \Sigma^{-nL}\creff_n(R_X)(S^L,\dots,S^L) \homeq \Sigma^{-nL} \Sigma^\infty \Hom_{\spectra}(X,S^L)^{\smsh n} \]
with the stabilization maps of (\ref{eq:creff-map}) given by the maps
\[ \Map \left( S^{nL}, \Sigma^\infty \Hom_{\spectra}(X,S^L)^{\smsh n} \right) \to \Map \left( S^{n(L+1)}, \Sigma^\infty \Hom_{\spectra}(X,S^{L+1})^{\smsh n} \right) \]
induced by the $n$-fold smash power of the canonical map
\[ S^1 \smsh \Sigma^\infty \Hom_{\spectra}(X,S^L) \to \Sigma^\infty \Hom_{\spectra}(X,S^{L+1}). \]
\end{corollary}
\begin{proof}
The models for $\creff_nR_X$ in Lemma~\ref{lem:creff-spsp} are pointed simplicial in each variable and so, by Lemma~\ref{lem:creff-maps}, the maps (\ref{eq:creff-map}) are given by the relevant tensoring maps which are as shown.
\end{proof}

We now construct models for these partially-stabilized cross-effects that admit actions of the operads $\S^{-L}$. The key is to identify the desuspension spheres with terms of the cooperad $\S$.

\begin{definition} \label{def:creff-module}
For $X \in \finspec$ and a non-negative integer $L$, we define a symmetric sequence $\un{\d}^L[R_X]$ by
\[ \un{\d}^L[R_X](I) := \Map \left( (S^1 \smsh \mathbb{S}(I))^{\smsh L}, \Sigma^\infty \Hom_{\spectra}(X,S^L)^{\smsh I} \right). \]
We then make the symmetric sequence $\un{\d}^L[R_X]$ into a right $\S^{-L}$-module by combining the cooperad structure on $\mathbb{S}$ with the right $\Com$-module structure on $\Sigma^\infty \Hom_{\spectra}(X,S^L)$ described in Definition~\ref{def:SigmainftyX-Com}. Explicitly, for a surjection $\alpha: I \epi J$, we have a map
\[ \nu_\alpha: \un{\d}^L[R_X](J) \smsh \Smsh_{j \in J} \S^{-L}(I_j) \to \un{\d}^L[R_X](I) \]
given by the composite
\[ \begin{split} &\Map \left( (S^1)^{\smsh L} \smsh \mathbb{S}(J)^{\smsh L}, \Sigma^\infty \Hom_{\spectra}(X,S^L)^{\smsh J} \right) \smsh \Smsh_{j \in J} \Map \left(\mathbb{S}(I_j)^{\smsh L},\Com(I_j) \right) \\
    &\to \Map \left( (S^1)^{\smsh L} \smsh \mathbb{S}(J)^{\smsh L} \smsh \Smsh_{j \in J} \mathbb{S}(I_j)^{\smsh L}, \Sigma^\infty \Hom_{\spectra}(X,S^L)^{\smsh J} \smsh \Smsh_{j \in J} \Com(I_j) \right) \\
    &\to \Map \left( (S^1)^{\smsh L} \smsh \mathbb{S}(I)^{\smsh L}, \Sigma^\infty \Hom_{\spectra}(X,S^L)^{\smsh I} \right)
\end{split} \]
where the first map is the canonical smash product of mapping spectra, and the second is induced by the $L$-fold smash power of the cooperad structure map
\[ d_{\alpha} : \mathbb{S}(I) \to \mathbb{S}(J) \smsh \Smsh_{j \in J} \mathbb{S}(I_j) \]
and the right $\Com$-module structure map
\[ \Delta_\alpha: \Sigma^\infty \Hom_{\spectra}(X,S^L)^{\smsh J} \to \Sigma^\infty \Hom_{\spectra}(X,S^L)^{\smsh I} \]
associated to $\alpha$.
\end{definition}

\begin{lemma} \label{lem:dL-RX}
We have natural equivalences of symmetric sequences
\[ \Sigma^{-*L}\creff_*(R_X)(S^L,\dots,S^L) \homeq \un{\d}^L[R_X]. \]
\end{lemma}
\begin{proof}
This follows from Corollary~\ref{cor:der-hocolim-spsp}
\end{proof}

Our next task is to provide models for the maps between the desuspended cross-effects that respect these operad actions. This is somewhat more delicate than one might hope since the `obvious' choice does not quite work.

\begin{definition} \label{def:homeo}
There are homeomorphisms
\[ h_I: \mathbb{R} \times \mathbb{R}^I_0 \arrow{e,t}{\isom} \mathbb{R}^I \]
given by
\[ (t,(t_i)_{i \in I}) \mapsto (t+t_i)_{i \in I} \]
which extend to homeomorphisms (which we also call $h_I$):
\[ S^1 \smsh \mathbb{S}(I) \arrow{e,t}{\isom} (S^1)^{\smsh I} \]
where we identify $(S^1)^{\smsh I}$ with the one-point compactification of $\mathbb{R}^I$. The following diagram of homeomorphisms then commutes
\begin{equation} \label{eq:diagram-homeo}
\begin{diagram}
  \node{\mathbb{R} \times \mathbb{R}^J_0 \times \prod_{j \in J} \mathbb{R}^{I_j}_0} \arrow{e,t}{d_\alpha^{-1}} \arrow{s,l}{h_J}
    \node{\mathbb{R} \times \mathbb{R}^I_0} \arrow{s,r}{h_I} \\
  \node{\mathbb{R}^J \times \prod_{j \in J} \mathbb{R}^{I_j}_0} \arrow{e,t}{\prod h_{I_j}}
    \node{\mathbb{R}^I}
\end{diagram}
\end{equation}
\end{definition}

\begin{definition} \label{def:mL}
We construct a map of symmetric sequences
\[ \un{m}_L: \un{\d}^L[R_X] \to \un{\d}^{L+1}[R_X] \]
as follows. For a positive integer $n$ the required map
\[ \un{m}_L(n): \Map((S^1 \smsh \mathbb{S}(n))^{\smsh L}, \Sigma^\infty \Hom_{\spectra}(X,S^L)^{\smsh n}) \to \Map((S^1 \smsh \mathbb{S}(n))^{\smsh L+1}, \Sigma^\infty \Hom_{\spectra}(X,S^{L+1})^{\smsh n}) \]
is adjoint to a composite of the homeomorphism, given by applying $h_n$ to the final copy of $S^1 \smsh \mathbb{S}(n)$,
\[ (S^1 \smsh \mathbb{S}(n))^{\smsh (L+1)} \arrow{e,t}{\isom} (S^1 \smsh \mathbb{S}(n))^{\smsh L} \smsh S^n \]
and the $n$-fold smash power of the canonical map
\[ S^1 \smsh \Sigma^\infty \Hom_{\spectra}(X,S^L) \to \Sigma^\infty \Hom_{\spectra}(X,S^{L+1}). \]
\end{definition}

\begin{lemma} \label{lem:hocolim-m}
The homotopy colimit of the sequence
\[ \dgTEXTARROWLENGTH=3em\un{\d}^0[R_X] \arrow{e,t}{\un{m}_0} \un{\d}^1[R_X] \arrow{e,t}{\un{m}_1} \un{\d}^2[R_X] \arrow{e,t}{\un{m}_2} \dots \]
is naturally (in $X$) equivalent to the symmetric sequence $\der_*(R_X)$ of derivatives of the representable functor $R_X$.
\end{lemma}
\begin{proof}
It follows from the Definition of $\un{m}_L$, and Corollary~\ref{cor:der-hocolim-spsp}, that $\un{m}_L$ models the stabilization map
\[ \Sigma^{-*L}\creff_*(R_X)(S^L,\dots,S^L) \to \Sigma^{-*(L+1)}\creff_*(R_X)(S^{L+1},\dots,S^{L+1}) \]
of (\ref{eq:creff-map}). The claim follows.
\end{proof}

Unfortunately the map $\un{m}_L$ is not a morphism of $\S^{-(L+1)}$-modules so we cannot apply our general theory to produce a $\C_{\S_\bullet}$-coalgebra structure on the above homotopy colimit. However, it is possible to extend $\un{m}_L$ to a morphism of $\S^{-(L+1)}$-modules
\[ \un{m}'_L : \B(\un{\d}^L[R_X],\S^{-(L+1)},\S^{-(L+1)}) \to \un{\d}^{L+1}[R_X] \]
where the left-hand side is the standard bar resolution of $\un{\d}^L[R_X]$ as a $\S^{-(L+1)}$-module.

Recall that for any operad $\P$ and $\P$-module $\M$, the bar resolution is a weak equivalence of $\P$-modules
\[ p: \B(\M,\P,\P) \weq \M. \]
The map $p$ has a section
\[ s: \M \weq \B(\M,\P,\P) \]
that is a morphism of symmetric sequences, but not of $\P$-modules. The map $s$ is given by the composite
\[ \M \arrow{e,t}{\eta} \M \circ \P \to \B(\M,\P,\P) \]
of the unit map for the operad $\P$ with the inclusion of the $0$-simplices into the bar construction. We then have the following proposition.

\begin{proposition} \label{prop:bar-map}
There is a natural morphism of $\S^{-(L+1)}$-modules
\[ \un{m}'_L : \B(\un{\d}^L[R_X],\S^{-(L+1)},\S^{-(L+1)}) \to \un{\d}^{L+1}[R_X] \]
such that the following diagram commutes
\[ \begin{diagram}
  \node{\B(\un{\d}^L[R_X],\S^{-(L+1)},\S^{-(L+1)})} \arrow{se,t}{\un{m}'_L} \\
  \node{\un{\d}^L[R_X]} \arrow{n,lr}{s}{\sim} \arrow{e,b}{\un{m}_L} \node{\un{\d}^{L+1}[R_X]}
\end{diagram} \]
\end{proposition}

The construction of the map $\un{m}'_L$ is quite involved and we defer the proof of Proposition~\ref{prop:bar-map} to Section~\ref{sec:bar-map}.

We now use the maps $\un{m}'_L$ to construct a sequence of $\S^{-L}$-modules whose actions do commute with the operad maps $\S^{-(L+1)} \to \S^{-L}$.

\begin{definition} \label{def:d'L}
For $X \in \finspec$, we construct the following diagram
\[ \begin{diagram}
  \node[3]{\B(\un{\d}^2[R_X],\S^{-3},\S^{-3})} \arrow{e,t}{\un{m}'_2} \arrow{s,lr}{\sim}{p} \node{\dots} \\
  \node[2]{\B(\un{\d}^1[R_X],\S^{-2},\S^{-2})} \arrow{e,t}{\un{m}'_1} \arrow{s,lr}{\sim}{p} \node{\un{\d}^2[R_X]} \arrow[2]{s,l}{\sim} \\
  \node{\B(\un{\d}^0[R_X],\S^{-1},\S^{-1})} \arrow{e,t}{\un{m}'_0} \arrow{s,lr}{\sim}{p} \node{\un{\d}^1[R_X]} \arrow{s,l}{\sim} \\
  \node{\un{\d}^0[R_X]} \arrow{e,b}{\un{\un{m}}_0} \node{\un{\un{\d}}^1[R_X]} \arrow{e,b}{\un{\un{m}}_1} \node{\un{\un{\d}}^2[R_X]} \arrow{e,b}{\un{\un{m}}_2} \node{\dots}
\end{diagram} \]
so that each rectangle is a homotopy pushout in the category of $\S^{-L}$-modules for respective values of $L$. We thus obtain a sequence
\[ \dgTEXTARROWLENGTH=2.5em \un{\d}^0[R_X] = \un{\un{\d}}^0[R_X] \arrow{e,t}{\un{\un{m}}_0} \un{\un{\d}}^1[R_X] \arrow{e,t}{\un{\un{m}}_1} \un{\un{\d}}^2[R_X] \arrow{e,t}{\un{\un{m}}_2} \dots \]
such that $\un{\un{\d}}^L[R_X]$ is a $\S^{-L}$-module, and such that the map
\[ \un{\un{m}}_L: \un{\un{\d}}^L[R_X] \to \un{\un{\d}}^{L+1}[R_X] \]
is a morphism of $\S^{-(L+1)}$-modules that is equivalent (in the homotopy category) to $\un{m}_L$. In other words $\un{\un{\d}}^\bullet[R_X]$ is a module over the pro-operad $\S_\bullet$.

Finally, by applying the construction of Definition~\ref{def:cofcoalg}, we obtain a diagram
\begin{equation} \label{eq:d-seq} \begin{diagram}
  \node{\d^0[R_X]} \arrow{e,t}{m_0} \arrow{s,l}{\sim} \node{\d^1[R_X]} \arrow{e,t}{m_1} \arrow{s,l}{\sim} \node{\d^2[R_X]} \arrow{e,t}{m_2} \arrow{s,l}{\sim} \node{\dots} \\
  \node{\un{\un{\d}}^0[R_X]} \arrow{e,t}{\un{\un{m}}_0} \node{\un{\un{\d}}^1[R_X]} \arrow{e,t}{\un{\un{m}}_1} \node{\un{\un{\d}}^2[R_X]} \arrow{e,t}{\un{\un{m}}_2} \node{\dots}
\end{diagram}
\end{equation}
in which each vertical map is an equivalence of $\S^{-L}$-modules for the appropriate $L$, and each $\d^L[R_X]$ is $\Sigma$-cofibrant.
\end{definition}

We are now in a position to choose models for the derivatives of representable functors as $\C_{\S_\bullet}$-coalgebras.

\begin{definition} \label{def:d-RX}
For $X \in \finspec$, we define
\[ \d[R_X] := \hocolim_{L} \d^L[R_X]. \]
As in Definition~\ref{def:hocolim-module-coalgebra}, this symmetric sequence has the structure of a $\C_{\S_\bullet}$-coalgebra, where $\C_{\S_\bullet}$ is the comonad of Definition~\ref{def:sphere-comonad}. Moreover, $\d[R_X]$ is $\Sigma$-cofibrant and the functor $X \mapsto \d[R_X]$ is simplicially-enriched.
\end{definition}

\begin{lemma} \label{lem:d-RX}
The $\C_{\S_\bullet}$-coalgebra $\d[R_X]$ is naturally equivalent to the symmetric sequence $\der_*(R_X)$ of derivatives of the representable functor $R_X = \Sigma^\infty \Hom_{\spectra}(X,-)$.
\end{lemma}
\begin{proof}
By Proposition~\ref{prop:bar-map}, the homotopy colimits of the horizontal sequences in (\ref{eq:d-seq}) are equivalent to the homotopy colimit in Lemma~\ref{lem:hocolim-m}.
\end{proof}

Having established a $\C_{\S_\bullet}$-coalgebra structure on the derivatives of the representable functors, we left Kan extend to all pointed simplicial $F: \finspec \to \spectra$.

\begin{definition} \label{def:dF-spsp}
For an arbitrary pointed simplicial functor $F: \finspec \to \spectra$, we define $\d^L[F]$ by the enriched coend
\[ \d^L[F] := F(X) \smsh_{X \in \finspec} \d^L[R_X] \]
over the simplicial category $\finspec$, and
\[ \d[F] := F(X) \smsh_{X \in \finspec} \d[R_X] \isom \hocolim_{L} \d^L[F]. \]
The diagram $\d^\bullet[F]$ is a module over the pro-operad $\S_\bullet$ and its homotopy colimit $\d[F]$ is therefore a $\C_{\S_\bullet}$-coalgebra.
\end{definition}

\begin{proposition} \label{prop:dF-spsp}
For $F \in [\finspec,\spectra]$ cofibrant, we have natural equivalences of symmetric sequences
\[ \Sigma^{-*L}\creff_*F(S^L,\dots,S^L) \weq \d^L[F] \]
and, taking the homotopy colimit as $L \to \infty$, an equivalence
\[ \der_*(F) \weq \d[F]. \]
\end{proposition}
\begin{proof}
Since taking cross-effects commutes with homotopy colimits, this follows from Lemmas~\ref{lem:dL-RX} and \ref{lem:d-RX}.
\end{proof}

Note that we have a stronger result than in the case of functors from spaces to spectra: the terms $\d^L[F]$ are equivalent to the partially-stabilized cross-effects without any analyticity condition on $F$.

\subsection{Classification of polynomial functors from spectra to spectra}

We now turn to the classification of polynomial functors from spectra to spectra. The structure of this section is similar to the corresponding section for functors from spaces to spectra. The main result is that the comonad $\C_{\S_\bullet}$ is equivalent to the comonad $\C$ constructed in \cite{arone/ching:2014} which is known to act on the derivatives of such a functor, and from which the Taylor tower can be recovered.

\begin{definition} \label{def:comonad-C}
The functor $\d: [\finspec,\spectra] \to \symseq$ is simplicially-enriched and has a simplicial right adjoint $\Phi: \symseq \to [\finspec,\spectra]$ given by
\[ \Phi(\A)(X) := \Map_{\Sigma}(\d[R_X],\A). \]
We then define a comonad $\C$ on $\symseq$ by
\[ \C = \d\c\Phi \]
where $\c$ is a comonad cofibrant replacement functor for $[\finspec,\spectra]$. We obtain a map of comonads $\theta: \C \to \C_{\S_\bullet}$ given by the composites
\[ \theta_{\A}: \C(\A) \to \C_{\S_\bullet}(\C(\A)) \to \C_{\S_\bullet}(\A) \]
in which the first map is given by the $\C_{\S_\bullet}$-coalgebra structure on $\d[\c\Phi\A]$, and the second by the counit for $\C$.
\end{definition}

\begin{theorem} \label{thm:spsp}
For any symmetric sequence $\A$, the comonad map
\[ \theta_{\A} : \C(\A) \to \C_{\S_\bullet}(\A) \]
is a weak equivalence.
\end{theorem}
\begin{proof}
We follow a similar approach to the proof of Theorem~\ref{thm:topsp}. The main step is to construct, when $\A$ is bounded, equivalences of $\S^{-L}$-modules of the form
\[ \zeta^L: \d^L[\c\Phi\A] \homeq \C_{\S^{-L}}(\A) \]
that commute with the maps $\d^L \to \d^{L+1}$ on the left-hand side, and that induced by the operad morphism $\S^{-(L+1)} \to \S^{-L}$ on the right-hand side. The remainder of the proof of \ref{thm:topsp} then applies in the same way to deduce this result.

Since both sides of the required equivalence commute with products, we can reduce to the case that $\A$ is concentrated in a single term, say the \ord{$n$}. Recall that by \cite[5.11]{arone/ching:2014} there is then a natural equivalence
\[ \Phi\A(X) \homeq (\A(n) \smsh X^{\smsh n})^{h\Sigma_n}. \]
Then there is a $\Sigma_n$-equivariant equivalence
\[ \d^L[\c\Phi\A](n) \homeq \Sigma^{-nL}\creff_n(\Phi\A)(S^L,\dots,S^L) \homeq \Sigma^{-nL} \left[ \Wdge_{\un{n} \epi \un{n}} \A(n) \smsh (S^L)^{\smsh n} \right]^{h\Sigma_n} \homeq \A(n). \]
It is therefore enough that there is an equivalence of $\S^{-L}$-modules
\[ \d^L[\c\Phi\A] \homeq \C_{\S^{-L}}(\d^L[\c\Phi\A(n)]) \]
or alternatively that the $\S^{-L}$-module structure map
\[ \d^L[\c\Phi\A](k) \to \left[ \prod_{\un{n} \epi \un{k}} \Map(\S^{-L}(n_1) \smsh \dots \smsh \S^{-L}(n_k),\d^L[\c\Phi\A](n)) \right]^{\Sigma_n} \]
is an equivalence for each $k$.

Since the functor $\d^L$ (which computes partially-stabilized cross-effects) commutes both with taking homotopy fixed points and with smashing with a fixed spectrum, it is then sufficient to show that the $\S^{-L}$-action map
\[ \d^L[\c(X^{\smsh n})](k) \to \left[ \prod_{\un{n} \epi \un{k}} \Map(\S^{-L}(n_1) \smsh \dots \smsh \S^{-L}(n_k),\d^L[\c(X^{\smsh n})](n)) \right]^{\Sigma_n} \]
is an equivalence.

Recall that we have
\[ \d^L[\c(X^{\smsh n})](k) \homeq \c(X^{\smsh n}) \smsh_{X \in \finspec} \Map((S^1 \smsh \S(k))^{\smsh L}, \Sigma^\infty \Hom_{\spectra}(X,S^L)^{\smsh k}). \]
Considering the definition of the $\S^{-L}$-module structure here, and consolidating the sphere factors on one side, it is sufficient to show that there is an equivalence
\begin{equation} \label{eq:c-map} \c(X^{\smsh n}) \smsh_{X \in \finspec} \Sigma^\infty \Hom_{\spectra}(X,S^L)^{\smsh k} \weq \prod_{\un{n} \epi \un{k}} S^{nL} \end{equation}
where the component corresponding to a surjection $\alpha: \un{n} \epi \un{k}$ is given by the composite
\[ \c(X^{\smsh n}) \smsh_{X \in \finspec} \Sigma^\infty \Hom_{\spectra}(X,S^L)^{\smsh k} \arrow{e,t}{\Delta_\alpha} X^{\smsh n} \smsh_{X \in \finspec} \Sigma^\infty \Hom_{\spectra}(X,S^L)^{\smsh n} \arrow{e} (S^L)^{\smsh n} \isom S^{nL} \]
where the first map is given by the diagonal map associated to the surjection $\alpha$, and the second is the canonical evaluation map. To see that (\ref{eq:c-map}) is an equivalence, consider the commutative diagram
\[ \begin{diagram}
  \node{\thocofib_{I \subseteq \un{k}} X^{\smsh n} \smsh_{X \in \finspec} \Sigma^\infty \Hom_{\spectra}\left(X, \prod_{I} S^L \right)} \arrow{s,l}{\sim} \arrow{e,t}{\sim}
    \node{\thocofib_{I \subseteq \un{k}} \left( \prod_{I} S^L \right)^{\smsh n}} \arrow{s,r}{\sim} \\
  \node{X^{\smsh n} \smsh_{X \in \finspec} \Sigma^\infty \Hom_{\spectra}(X,S^L)^{\smsh k}} \arrow{e}
    \node{\prod_{\un{n} \epi \un{k}} S^{nL}}
\end{diagram} \]
where the top map is an equivalence by the enriched co-Yoneda Lemma, the left-hand map is an equivalence because for any based space $Y$, the smash power $Y^{\smsh k}$ is equivalent to the total homotopy cofibre of the $k$-cube with terms $\prod_{I} Y$ for $I \subseteq \un{k}$, and the right-hand map is an equivalence by direct calculation.
\end{proof}

\begin{theorem} \label{thm:equiv-spsp-poly}
The functor $\d$ induces an equivalence
\[ [\finspec,\spectra]^{\mathsf{h}}_{\mathsf{poly}} \homeq \Coalg_{\mathsf{b}}(\S_\bullet)^{\mathsf{h}} \]
between the homotopy category of polynomial pointed simplicial functors $\finspec \to \spectra$ and the homotopy category of bounded $\C_{\S_\bullet}$-coalgebras of Definition~\ref{def:derived-map-coalgebras}. Moreover, for cofibrant polynomial functors $F,G \in [\finspec,\spectra]$, we have equivalences of simplicial mapping objects
\begin{equation} \label{eq:nat-map-spsp} \Nat_{\finspec}(F,G) \homeq \widetilde{\Map}_{\C_{\S_\bullet}}(\d[F],\d[G]). \end{equation}
\end{theorem}
\begin{proof}
The proof is entirely analogous to that of Theorem~\ref{thm:equiv-topsp-poly}, using Theorem~\ref{thm:spsp} instead of Theorem~\ref{thm:topsp}.
\end{proof}

\subsection{Classification of analytic functors from spectra to spectra}

As in the case of functors from spaces to spectra, the classification result of the previous section extends naturally to analytic functors in the following way.

\begin{definition}
The notions of \emph{analytic at $*$} and \emph{strongly-analytic at $*$} from Definitions~\ref{def:analytic} and \ref{def:strongly-analytic} generalize immediately to functors from spectra to spectra. We therefore have a category $[\finspec,\spectra]_{\mathsf{an}(*)}$ of pointed simplicial functors $F: \finspec \to \spectra$ that are strongly-analytic at $*$, and an associated homotopy category $[\finspec,\spectra]_{\mathsf{an}(*)}^{\mathsf{h}}$.
\end{definition}

\begin{theorem} \label{thm:equiv-spsp-ana}
The functor $\d$ sets up an equivalence
\[ [\finspec,\spectra]_{\mathsf{an}(*)}^{\mathsf{h}} \homeq \Coalg^{\mathsf{t}}_{\mathsf{an}}(\C_{\S_\bullet})^{\mathsf{h}} \]
between the homotopy category of functors $\finspec \to \spectra$ that are strongly-analytic at $*$ and the pro-truncated homotopy category of analytic $\C_{\S_\bullet}$-coalgebras. Moreover, for cofibrant $F,G \in [\finspec,\spectra]_{\mathsf{an}(*)}$ we have equivalences
\[ \Nat_{\finspec}(F,G) \weq \widetilde{\Map}_{\C_{\S_\bullet}}(\d[F],\d[G]) \weq \widetilde{\Map^{\mathsf{t}}}_{\C_{\S_\bullet}}(\d[F],\d[G]). \]
\end{theorem}
\begin{proof}
This is analogous to the proof of Theorem~\ref{thm:equiv-topsp-ana}.
\end{proof}

\begin{example} \label{ex:SO}
The derivatives of the functor $\Sigma^\infty \Omega^\infty: \spectra \to \spectra$ are given by
\[ \der_*(\Sigma^\infty \Omega^\infty) \homeq \Com. \]
The symmetric sequence $\Com$ has a canonical right $\Com$-module structure coming directly from the operad structure, and hence $\Com$ is a $\C_{\Com} = \C_{\S^0}$-coalgebra. The comonad map
\[ \C_{\S^0} \to \hocolim_{L \to \infty} \C_{\S^{-L}} = \C_{\S_\bullet} \]
then endows $\Com$ with a $\C_{\S_\bullet}$-coalgebra structure. It follows from \cite[5.19]{arone/ching:2014} that this structure encodes the Taylor tower of $\Sigma^\infty \Omega^\infty$.
\end{example}

\section{Proof of Proposition~\ref{prop:bar-map}} \label{sec:bar-map}

The goal of this section is to prove Proposition~\ref{prop:bar-map}, that is, to construct a map
\begin{equation} \label{eq:bar-map} \un{m}'_L : \B(\un{\d}^L[R_X],\S^{-(L+1)},\S^{-(L+1)}) \to \un{\d}^{L+1}[R_X] \end{equation}
where the source is the standard simplicial bar resolution of $\un{\d}^L[R_X]$ as a $\S^{-(L+1)}$-module.

We construct the map $\un{m}'_L$ by defining, for each non-negative integer $N$, a map of $\S^{-(L+1)}$-modules
\[ \un{m}^N_L : \Delta^N_+ \smsh \un{\d}^L[R_X] \circ \underbrace{\S^{-(L+1)} \circ \dots \circ \S^{-(L+1)}}_{\text{$N+1$ terms}} \to \un{\d}^{L+1}[R_X]. \]
We then show that these maps commute with the face and degeneracy maps in the simplicial bar construction and the coface and codegeneracy maps on the simplexes $\Delta^N$.

The $(N+1)$-fold composition product in the source of the map $\un{m}^N_L$ can be expressed as a wedge product of terms indexed by sequences of partitions. Specifically, consider a sequence of surjections
\[ \dgTEXTARROWLENGTH=2.5em \lambda: I = I^{(N+1)} \arrow{e,t,A}{\alpha_n} I^{(N)} \arrow{e,t,A}{\alpha_{n-1}} \cdots \arrow{e,t,A}{\alpha_1} I^{(1)} \arrow{e,t,A}{\alpha_0} I^{(0)} = J \]
between two finite sets $I$ and $J$ and recall that for $j \in I^{(k)}$ we are writing $I^{(k+1)}_j := \alpha_k^{-1}(j)$. Associated to each such $\lambda$, the required map $\un{m}^N_L$ has a component of the form
\[ \un{m}^\lambda_L: \Delta^N_+ \smsh \un{\d}^L(R_X)(J) \smsh \Smsh_{k = 0}^{N} \Smsh_{j \in I^{(k)}} \S^{-(L+1)}(I^{(k+1)}_j) \to \un{\d}^{L+1}(R_X)(I) \]
which we will now describe.

\begin{definition} \label{def:homotopy}
The maps $\un{m}^\lambda_L$ are built from coherent higher homotopies of the following form. Let $V_0,\dots,V_N$ be finite dimensional real vector spaces and, for each $i = 0,\dots,N$, suppose $U_i \subseteq V_i$ is a subset that is closed under non-negative scalar multiplication. Then we define continuous functions
\[ r^N: \Delta^N \times U_0 \times \dots \times U_N \to U_0 \times \dots \times U_N \times U_0 \times \dots \times U_N \]
by identifying a point in the simplex $\Delta^N$ as a sequence $t = (t_0,\dots,t_N)$, with all $t_j \geq 0$ and $\sum_{j = 0}^{N} t_j = 1$, and then by setting
\[ r^N(t, u_0, \dots, u_N) \]
to be equal to
\[ \left( t_0u_0, \dots, (t_0+\dots+t_j)u_j, \dots, u_N, (t_1+\dots+t_N)u_0, \dots, (t_{j+1}+\dots+t_N)u_j, \dots, t_Nu_{N-1}, 0 \right). \]
Though the $U_N$ term does not really feature in the homotopy itself, we include it to make the interaction with coface and codegeneracy maps more clear. Those interactions are described by the following result.
\end{definition}

\begin{lemma} \label{lem:homotopy}
(a) Let $U_0,\dots,U_N$ be as in Definition~\ref{def:homotopy}. For an integer $1 \leq j \leq N$, the following diagram commutes
\[ \begin{diagram}
  \node{\Delta^{N-1} \times U_0 \times \dots \times (U_{j-1} \times U_j) \times \dots \times U_N} \arrow{e,t}{r^{N-1}} \arrow{s,l}{\delta^j}
    \node{U_0 \times \dots \times U_N \times U_0 \times \dots \times U_N} \arrow{s,r}{\isom} \\
  \node{\Delta^N \times U_0 \times \dots \times U_N} \arrow{e,t}{r^N}
    \node{U_0 \times \dots \times U_N \times U_0 \times \dots \times U_N}
\end{diagram} \]
where $\delta^j: \Delta^{N-1} \to \Delta^N$ is the coface map given by
\[ (t_0,\dots,t_{N-1}) \mapsto (t_0,\dots,t_{j-1},0,t_j,\dots,t_{N-1}). \]
(b) For $j = 0$, the following diagram commutes
\[ \begin{diagram}
  \node{\Delta^{N-1} \times U_1 \times \dots \times U_N} \arrow{e,t}{r^{N-1}} \arrow{s,l}{\delta^0 \times \iota}
    \node{U_1 \times \dots \times U_N \times U_1 \times \dots \times U_N} \arrow{s,r}{\iota \times \iota} \\
  \node{\Delta^N \times U_0 \times U_1 \times \dots \times U_N} \arrow{e,t}{r^N}
    \node{U_0 \times U_1 \times \dots \times U_N \times U_0 \times U_1 \times \dots \times U_N}
\end{diagram} \]
where $\iota: U_1 \times \dots \times U_N \to U_0 \times U_1 \times \dots \times U_N$ is the inclusion $(u_1,\dots,u_N) \mapsto (0,u_1,\dots,u_N)$.

(c) For an integer $0 \leq i \leq N$, the following diagram commutes
\[ \begin{diagram}
  \node{\Delta^{N+1} \times U_0 \times \dots \times U_{i-1} \times \mathbf{0} \times U_i \times \cdots} \arrow{s,l}{\sigma^i} \arrow{e,t}{r^{N+1}}
    \node{U_0 \times \dots \times \mathbf{0} \times \dots \times U_0 \times \dots \times \mathbf{0} \times \dots \times U_N} \arrow{s,r}{\isom} \\
  \node{\Delta^N \times U_0 \times \dots \times U_N} \arrow{e,t}{r^N}
    \node{U_0 \times \dots \times U_N \times U_0 \times \dots \times U_N}
\end{diagram} \]
where $\mathbf{0}$ is a zero-dimensional vector space and $\sigma^i: \Delta^{N+1} \to \Delta^N$ is the codegeneracy map given by
\[ (t_0,\dots, t_{N+1}) \mapsto (t_0,\dots,t_{i-1},t_i+t_{i+1},t_{i+2},\dots,t_{N+1}). \]
\end{lemma}
\begin{proof}
Each of these is easily checked directly from the definitions.
\end{proof}

\begin{definition} \label{def:m-lambda}
Now let $I,J$ be nonempty finite sets and suppose that
\[ \dgTEXTARROWLENGTH=2.5em \lambda: I = I^{(N+1)} \arrow{e,t,A}{\alpha_N} I^{(N)} \arrow{e,t,A}{\alpha_{N-1}} \dots \arrow{e,t,A}{\alpha_{1}} I^{(1)} \arrow{e,t,A}{\alpha_0} I^{(0)} = J \]
is a sequence of surjections (which determines a sequence of partitions of $I$). We apply the construction of Definition~\ref{def:homotopy} to the sets
\[ U_k := \prod_{j \in I^{(k)}} \mathbb{R}^{I^{(k+1)}_j}_0 \]
which are subsets of the vector spaces $V_k = \prod_{j \in I^{(k)}} \mathbb{R}^{I^{(k+1)}_j}$ that are closed under non-negative scalar multiplication. This gives us a map
\[ r^{\lambda}: \Delta^N \times U_0 \times \dots \times U_N \to U_0 \times \dots \times U_N \times U_0 \times \dots \times U_N. \]

Homeomorphisms of the form $d_{\alpha}^{-1}$ (of Definition~\ref{def:sphere-cooperad}) and $h_I$ (of Definition~\ref{def:homeo}) combine to give a homeomorphism
\begin{equation} \label{eq:U-homeo} \dgTEXTARROWLENGTH=2.5em h^\lambda: \mathbb{R}^J \times U_0 \times \dots \times U_N \arrow{e,t}{h_J^{-1}} \mathbb{R} \times \mathbb{R}^J_0 \times U_0 \times \dots \times U_N \arrow{e,t}{\prod d_{\alpha_i}^{-1}} \mathbb{R} \times \mathbb{R}^I_0 \arrow{e,t}{h_I} \mathbb{R}^I. \end{equation}
We can define maps
\[ \tilde{r}^{\lambda} : \Delta^N \times \mathbb{R}^I \to U_0 \times \dots \times U_N \times \mathbb{R}^I \]
such that the following diagram commutes
\[ \begin{diagram}
  \node{\Delta^N \times \mathbb{R}^J \times U_0 \times \dots \times U_N} \arrow{e,t}{r^{\lambda}} \arrow{s,lr}{\isom}{h^{\lambda}}
    \node{\mathbb{R}^J \times U_0 \times \dots \times U_N \times U_0 \times \dots \times U_N} \arrow{s,lr}{\isom}{h^{\lambda}} \\
  \node{\Delta^N \times \mathbb{R}^I} \arrow{e,t}{\tilde{r}^{\lambda}}
    \node{\mathbb{R}^I \times U_0 \times \dots \times U_N}
\end{diagram} \]
The map $\tilde{r}^{\lambda}$ extends to one-point compactifications to give a map
\[ s^\lambda: \Delta^N_+ \smsh S^I \to \left( \Smsh_{k = 0}^{N} \Smsh_{j \in I^{(k)}} \mathbb{S}(I^{(k+1)}_j) \right) \smsh S^I. \]
Taking the one-point compactification of the map $h^{\lambda}$ of (\ref{eq:U-homeo}) we get a homeomorphism
\[ S^J \smsh \Smsh_{k = 0}^{N} \Smsh_{j \in I^{(k)}} \mathbb{S}(I^{(k+1)}_j) \isom S^I \]
which combined with the maps $s^{\lambda}$ gives us maps
\[ t^{\lambda}_L: \Delta^N_+ \smsh (S^I)^{\smsh(L+1)} \to (S^J)^{\smsh L} \smsh \left( \Smsh_{k = 0}^{N} \Smsh_{j \in I^{(k)}} \mathbb{S}(I^{(k+1)}_j)^{\smsh (L+1)} \right) \smsh S^I. \]
We now combine the map $t^\lambda_L$ with the composite
\[ \begin{split} S^I \smsh \Sigma^\infty \Hom_{\spectra}(X,S^L)^{\smsh J} &\to S^I \smsh \Sigma^\infty \Hom_{\spectra}(X,S^L)^{\smsh I} \\
    &\isom (S^1 \smsh \Sigma^\infty \Hom_{\spectra}(X,S^L))^{\smsh I} \\
    &\to \Sigma^\infty \Hom_{\spectra}(X,S^{L+1})^{\smsh I}
\end{split} \]
of: (i) the map induced by the diagonal on $\Hom_{\spectra}(X,S^L)$ associated to the composed surjection $\alpha_N \circ \dots \circ \alpha_0 : I \epi J$; (ii) the canonical twist map for the symmetric monoidal smash product; and (iii) the tensoring map for $\Sigma^\infty \Hom_{\spectra}(X,-)$. Applying suitable adjoints these determine the required map
\[ \un{m}^\lambda_L: \Delta^N_+ \smsh \un{\d}^L[R_X](J) \smsh \Smsh_{k = 0}^{N} \Smsh_{j \in I^{(k)}} \S^{-(L+1)}(I^{(k+1)}_j) \to \un{\d}^{L+1}[R_X](I). \]
\end{definition}

We now have to check that the maps $\un{m}^{\lambda}_L$ respect the face and degeneracy maps in the bar construction and hence determine the required map (\ref{eq:bar-map}).

\begin{lemma} \label{lem:bar-map}
Let $I,J$ be nonempty finite sets and suppose that as before $\lambda$ is a sequence of surjections
\[ \dgTEXTARROWLENGTH=2.5em I = I^{(N+1)} \arrow{e,t,A}{\alpha_N} \dots \arrow{e,t,A}{\alpha_0} I^{(0)} = J. \]
(a) For an integer $1 \leq r \leq N$, let $d_r(\lambda)$ denote the sequence of surjections $(\alpha_0,\dots,\alpha_{r-1} \circ \alpha_r, \dots, \alpha_N)$. Also write
\[ \tilde{I}^{(r)}_j := (\alpha_{r-1} \circ \alpha_r)^{-1}(j). \]
Then the following diagram commutes:
\[ \begin{diagram}
  \node{\Delta^{N-1}_+ \smsh \un{\d}^L[R_X](J) \smsh \dots \smsh \Smsh_{j \in I^{(r-1)}} \S^{-(L+1)}(\tilde{I}^{(r)}_j) \smsh \dots}
    \arrow{se,t}{\un{m}^{d_r(\lambda)}_L} \\
  \node{\Delta^{N-1}_+ \smsh \un{\d}^L[R_X](J) \smsh \dots \smsh \Smsh_{j \in I^{(r-1)}} \S^{-(L+1)}(I^{(r)}_j) \smsh \Smsh_{j \in I^{(r)}} \S^{-(L+1)}(I^{(r+1)}_j) \smsh \dots}
    \arrow{n,l}{\mu} \arrow{s,l}{\delta^j}
  \node{\un{\d}^{L+1}[R_X](I)} \\
  \node{\Delta^N_+ \smsh \un{\d}^L[R_X](J) \smsh \dots \smsh \Smsh_{j \in I^{(r-1)}} \S^{-(L+1)}(I^{(r)}_j) \smsh \Smsh_{j \in I^{(r)}} \S^{-(L+1)}(I^{(r+1)}_j) \smsh \dots}
    \arrow{ne,b}{\un{m}_L^\lambda}
\end{diagram} \]
where
\[ \mu : \Smsh_{j \in I^{(r-1)}} \S^{-(L+1)}(I^{(r)}_j) \smsh \Smsh_{j \in I^{(r)}} \S^{-(L+1)}(I^{(r+1)}_j) \to \Smsh_{j \in I^{(r-1)}} \S^{-(L+1)}(\tilde{I}^{(r)}_j) \]
is the smash product of operad composition maps for $\S^{-(L+1)}$, and $\delta^j$ is induced by the corresponding coface map $\Delta^{N-1} \to \Delta^N$.

(b) For the case $r = 0$, let $d_0(\lambda)$ denote the sequence $(\alpha_1,\dots,\alpha_N)$. Then the following diagram commutes:
\[ \begin{diagram}
  \node{\Delta^{N-1}_+ \smsh \un{\d}^L[R_X](I^{1}) \smsh \Smsh_{k = 1}^{N} \Smsh_{j \in I^{(k)}} \S^{-(L+1)}(I^{(k+1)}_j)}
    \arrow{se,t}{\un{m}_L^{d_0(\lambda)}} \\
  \node{\Delta^{N-1}_+ \smsh \un{\d}^L[R_X](J) \smsh \Smsh_{k = 0}^{N} \Smsh_{j \in I^{(k)}} \S^{-(L+1)}(I^{(k+1)}_j)}
    \arrow{n,l}{\nu} \arrow{s,l}{\delta^0}
  \node{\un{\d}^{L+1}[R_X](I)} \\
  \node{\Delta^N_+ \smsh \un{\d}^L[R_X](J) \smsh \Smsh_{k = 0}^{N} \Smsh_{j \in I^{(k)}} \S^{-(L+1)}(I^{(k+1)}_j)}
    \arrow{ne,b}{\un{m}_L^{\lambda}}
\end{diagram} \]
where
\[ \nu : \un{\d}^L[R_X](J) \smsh \Smsh_{j \in I^{(0)}} \S^{-(L+1)}(I^{(1)}_j) \to \un{\d}^L[R_X](I^{(1)}) \]
is the $\S^{-(L+1)}$-module structure map on $\un{\d}^L[R_X]$ (pulled back from the $\S^{-L}$-module structure of Definition~\ref{def:creff-module} along the operad map $\eta: \S^{-(L+1)} \to \S^{-L}$).

(c) For an integer $0 \leq r \leq N$, let $s_r(\lambda)$ denote the sequence $(\alpha_0,\dots,\alpha_{r-1},1_{I^{(r)}},\alpha_r,\dots,\alpha_N)$. Then the following diagram commutes:
\[ \begin{diagram}
  \node{\Delta^{N+1}_+ \smsh \un{\d}^L[R_X](J) \smsh \Smsh_{k=0}^{N} \Smsh_{j \in I^{(k)}} \S^{-(L+1)}(I^{(k+1)}_j) \smsh \Smsh_{j \in I^{(r)}} \S^{-(L+1)}(1)}
    \arrow{se,t}{\un{m}_L^{s_i(\lambda)}} \\
  \node{\Delta^{N+1}_+ \smsh \un{\d}^L[R_X](J) \smsh \Smsh_{k = 0}^{N} \Smsh_{j \in I^{(k)}} \S^{-(L+1)}(I^{(k+1)}_j)}
    \arrow{n,l}{\eta} \arrow{s,l}{\sigma^r}
  \node{\un{\d}^{L+1}[R_X](I)} \\
  \node{\Delta^N_+ \smsh \un{\d}^L[R_X](J) \smsh \Smsh_{k = 0}^{N} \Smsh_{j \in I^{(k)}} \S^{-(L+1)}(I^{(k+1)}_j)}
    \arrow{ne,b}{\un{m}_L^{\lambda}}
\end{diagram} \]
where $\eta: S^0 \weq \S^{-(L+1)}(1)$ is the unit map for the operad $\S^{-(L+1)}$, and $\sigma^r$ is induced by the corresponding codegeneracy map $\Delta^{N+1} \to \Delta^N$.
\end{lemma}
\begin{proof}
Working through the definitions of the maps $\un{m}_L^\lambda$, and the relevant operad and module structure maps, each of these claims follows from the corresponding part of Lemma \ref{lem:homotopy}.
\end{proof}

The following result is then a restatement of the required Proposition~\ref{prop:bar-map}.

\begin{proposition} \label{prop:bar-map1}
The maps $\un{m}_L^\lambda$ of Definition~\ref{def:m-lambda} together define a morphism of $\S^{-(L+1)}$-modules
\[ \un{m}'_L: \B(\un{\d}^L[R_X],\S^{-(L+1)},\S^{-(L+1)}) \to \un{\d}^{L+1}[R_X] \]
such that the composite
\[ \dgTEXTARROWLENGTH=2em \un{\d}^L[R_X] \arrow{e,tb}{s}{\sim} \B(\un{\d}^L[R_X],\S^{-(L+1)},\S^{-(L+1)}) \arrow{e,t}{\un{m}'_L} \un{\d}^{L+1}[R_X] \]
is equal to the map $\un{m}_L$ of Definition~\ref{def:mL}.
\end{proposition}
\begin{proof}
The existence of a map $\un{m}'_L$ follows from Lemma~\ref{lem:bar-map}. To see this is a morphism of $\S^{-(L+1)}$-modules, it is sufficient to show that each map $m^N_L$ is a morphism of $\S^{-(L+1)}$-modules. Let $\lambda = (\alpha_0,\dots,\alpha_N)$ be a sequence of surjections of finite sets
\[ \dgTEXTARROWLENGTH=2.5em I = I^{(N+1)} \arrow{e,t,A}{\alpha_N} \cdots \arrow{e,t,A}{\alpha_0} I^{(0)} = J \]
and let $\lambda' = (\alpha_0,\dots,\alpha_{N-1})$. Then it is sufficient to show that the following diagram commutes
\[ \begin{diagram}
  \node[2]{\un{\d}^{L+1}[R_X](I^{(N)}) \smsh \Smsh_{j \in I^{(N+1)}} \S^{-(L+1)}(I^{(N+1)}_j)} \arrow{s,r}{\nu_{L+1}} \\
  \node{\un{\d}^L[R_X](J) \smsh \Smsh_{k = 0}^{N} \Smsh_{j \in I^{(k)}} \S^{-(L+1)}(I^{(k+1)}_j)} \arrow{e,t}{\un{m}^\lambda_L} \arrow{ne,t}{\un{m}^{\lambda'}_L}
    \node{\un{\d}^{L+1}[R_X](I)}
\end{diagram} \]
where $\nu_{L+1}$ is the $\S^{-(L+1)}$-module structure map for $\un{\d}^{L+1}[R_X]$ (from Definition~\ref{def:creff-module}) associated to the surjection $\alpha_N$. This follows from the fact that in the map $r^N$ of Definition~\ref{def:homotopy}, the term $U_N$ is mapped by the identity into the first copy of $U_N$ in the target, and by zero into the second copy.

Finally, the composite $\un{m}'_L \circ s$ is equal to the composite
\[ \dgTEXTARROWLENGTH=2em \un{\d}^L[R_X](I) \arrow{e,t}{\eta} \un{\d}^L[R_X](I) \smsh \Smsh_{i \in I} \S^{-(L+1)}(1) \arrow{e,t}{\un{m}_L^\lambda} \un{\d}^{L+1}[R_X](I) \]
where $\lambda$ is the sequence consisting just of the identity map $I \arrow{e,t}{=} I$. Following through the definitions we see that this is precisely $\un{m}_L$.
\end{proof}

\bibliographystyle{amsplain}
\bibliography{mcching}

\end{document}